\documentclass[11pt]{article}
\usepackage[utf8]{inputenc}
{\tiny {\tiny }}
\usepackage[numbers,sort&compress]{natbib}

\usepackage{hyperref}
\usepackage{algorithm}
\usepackage{algorithmic}
\usepackage{comment}
\usepackage{pgf}
\usepackage{tikz}
\usetikzlibrary{arrows, decorations.pathmorphing, backgrounds, positioning, fit, petri, automata}
\definecolor{yellow1}{rgb}{1,0.8,0.2}

\usepackage{makecell}
\usepackage[numbers]{natbib} 

\usepackage{amsmath,amsthm,amsfonts,amssymb,amsbsy,amsopn,amstext}
\usepackage{mathtools}
\usepackage{graphicx,color}
\usepackage{mathbbol,mathrsfs}
\usepackage{float,ccaption}
\usepackage{indentfirst}
\usepackage{appendix}
\usepackage{subfigure}
\usepackage{graphics} 
\usepackage{epsfig} 
\usepackage{epstopdf}
\usepackage{subcaption}

\usepackage{emptypage}

\usepackage{booktabs}
\usepackage{threeparttable}
\usepackage{multirow}
\usepackage{diagbox}

 \newtheorem{thm}{Theorem}
 
 \newtheorem{lem}{Lemma}
 
 \newtheorem{defn}{Definition}
 
 \newtheorem{exm}{Example}

 \newtheorem{ass}{Assumption}

\newcommand{\normm}[1]{{\left\vert\kern-0.25ex\left\vert\kern-0.25ex\left\vert #1
		\right\vert\kern-0.25ex\right\vert\kern-0.25ex\right\vert}}

\usepackage{geometry}
\begin{document}

\title{Adaptive Stochastic Gradient Descent Ascent Algorithm for Nonconvex Minimax Problems with Decision-Dependent Distributions}

\author{Yan~Gao and Yongchao~Liu\thanks{School of Mathematical Sciences, Dalian University of Technology, Dalian 116024, China, e-mail: gydllg123@mail.dlut.edu.cn (Yan~Gao), lyc@dlut.edu.cn (Yongchao~Liu) }}
\date{}
\maketitle
\noindent{\bf Abstract.} 
In this paper, we study stochastic minimax problems with decision-dependent distributions (SMDD), where the probability distribution of stochastic variable depends on decision variable.  
For SMDD with nonconvex-(strongly) concave  objective function, we propose an adaptive stochastic gradient descent ascent algorithm (ASGDA) to find the stationary points of SMDD, which learns the unknown distribution map dynamically and optimizes the minimax problem simultaneously. When the distribution map follows a location-scale model, we show that ASGDA finds an $\epsilon$-stationary point within $\mathcal{O}\left(\epsilon^{-\left(4+\delta\right)} \right)$ for $\forall\delta>0$, and $\mathcal{O}(\epsilon^{-8})$ stochastic gradient evaluations in nonconvex-strongly concave and nonconvex-concave settings respectively. When the  objective function of SMDD is nonconvex in $x$ and satisfies Polyak-{\L}ojasiewicz (P{\L}) inequality in $y$, we propose an alternating adaptive  stochastic gradient descent ascent algorithm (AASGDA) and show that AASGDA finds an $\epsilon$-stationary point within $\mathcal{O}(\kappa_y^4\epsilon^{-4})$ stochastic gradient evaluations, where $\kappa_y$ denotes the condition number. 
We verify the effectiveness of the proposed algorithms through numerical experiments on both synthetic and real-world data.


\section{Introduction}\label{Section_1}
Stochastic optimization problems with decision-dependent distributions (SODD) \cite{Cutler2023Stochastic, dupacova2006optimization, Goel2006, Hellemo2018decision, Jonsbraten1998, jonsbraaten1998class} may characterize the
phenomenon where the probability distribution of uncertainty depends on or shifts in reaction to decision variables,
which have many applications in the real world, such as, pricing problems \cite{cheung2017dynamic,cooper2006models}, route planning problems \cite{liebig2017dynamic, perdomo2020performative} and loan approval problems \cite{inga2022credit,robinson2024loan}. When the decision variable in a stochastic minimax problem affects the probability distribution, the corresponding stochastic minimax problems with decision-dependent distributions (SMDD) \cite{wood2023stochastic} can be formulated as follows:
\begin{equation}\label{eq:f_minimax_D}
	\min_{x\in\mathcal{X}}\max_{y\in \mathcal{Y}}\mathcal{L}(x,y):=\underset{z\sim \mathcal{D}(x, y)}{\mathbb{E}}\left[l(x,y,z)\right],
\end{equation}
where $z$ is a random variable supported on $\mathcal{Z}\subset\mathbb{R}^{d}$, $\mathcal{X}\subset\mathbb{R}^{n}$, $\mathcal{Y}\subset\mathbb{R}^{m}$,  $l(\cdot):\mathbb{R}^{n}\times\mathbb{R}^{m}\times\mathbb{R}^{d}\rightarrow\mathbb{R}$ is a smooth function, $\mathcal{D}(\cdot):\mathbb{R}^{n}\times\mathbb{R}^{m}\rightarrow\mathcal{P}(\mathbb{R}^{d})$ is a distribution map
and $\mathbb{E}\left[\,\cdot\,\right]$ denotes the expectation with respect to the distribution $\mathcal{D}(x,y)$. 
Some important applications can be cast in this form, as illustrated in the examples below.

\begin{exm}[\textbf{Distributionally robust strategic classification \cite{perdomo2020performative}, SMDD-I}]\label{example_1}
{\em Consider a bank that uses machine learning to classify the creditworthiness of loan applicants in a loan approval task, where the applicants may react to the bank’s classifier by manipulating their features to increase the likelihood of receiving a favorable classification. 
To prevent loan default risk, the bank considers the decision-dependent nature of the data distribution and integrates a distributionally robust framework to formulate a classification model.
Specifically, given the dataset $\{(a_i,b_i)\}_{i=1}^{N}$ is strategically manipulated based on the true information of the applicants, i.e., $\{(a_i^0,b_i^0)\}_{i=1}^{N}$, by $a_i=a_i^0+\psi(x)$, $b_i=b_i^0$, where $\psi(\cdot)$ is a response function to the classifier $x$, $a_i\in\mathbb{R}^{n}$ represents the features of the $i$-th sample corresponding to historical information about an individual, such as, their
monthly income and number of credit lines, and $b_{i} \in\{-1,1\}$ is the corresponding label, then the bank formulates the model as follows
\begin{equation}\label{euq:experiment_real_example}
	\min _{x \in \mathbb{R}^{n}} \max _{y \in Y \subset \mathbb{R}^{N}}\mathcal{L}(x, y):=\frac{1}{N} \sum_{i=1}^{N} y_{i}\ell(x;a_{i}^{0}+\psi(x),b_{i})+f(x)-g(y),
\end{equation}
where $x$ represents the classifier, $\ell(\cdot)$ is the logistic loss function over the sample $(a_i,b_i)$, $f(\cdot)$ is a regularizer, $g(\cdot)$ is a distributionally robust regularizer,
$Y \triangleq\left\{y \in \mathbb{R}_{+}^{N}: \mathbf{1}^{\top} y=1\right\}$ is a simplex, $\mathcal{L}(\cdot)$ is the average loss function over the entire dataset.
Given $(a_i,b_i)$ is strategically manipulated in response to the decision variable $x$, the probability distribution of the dataset $\{(a_i,b_i)\}_{i=1}^{N}$ depends on $x$, problem \eqref{euq:experiment_real_example} is a stochastic minimax problem with decision-dependent distribution.
}  
\end{exm}
\begin{exm}[\textbf{Election prediction \cite{narang2023multiplayer, perdomo2020performative}, SMDD-II}]
	\label{example_2}
{\em
Suppose that there are two election prediction platforms,
which aim to predict the final vote margin in an election contest.
People may use the two platforms as sources of information and may be influenced by the predictions. Let $x$ and $y$ denote the parameters of the prediction models corresponding to the two platforms and
$\theta$ represent the feature corresponding to an individual, which includes age, gender, past polling averages, etc. 
Suppose the election outcome $z$ corresponding to $\theta$ follows the probability distribution that depends on $x$ and $y$, where $\theta$ follows a fixed distribution $\mathcal{P}_{\theta}$. Suppose also that the platforms observe the samples of election outcome $z$ drawn from the decision-dependent distribution. Then the optimal parameter of the prediction model of platform one relative to platform two can be characterized by the stationary point of the following stochastic minimax problem with decision-dependent distribution
\begin{equation}\label{eq_subsection1_1_2}
	\min_{x\in\mathbb{R}^{n}}\max_{y\in\mathbb{R}^{n}}\mathcal{L}(x,y):=\underset{(\theta,z)\sim\mathcal{D}(x, y) }{\mathbb{E}}\frac{1}{2}\left[\|z_{1}-\theta^{\top}x\|^{2}-\|z_{2}-\theta^{\top}y\|^{2}+\gamma_{1}\|x\|^{2}-\gamma_{2}\|y\|^{2}\right],
\end{equation}
where $\gamma_1, \gamma_2>0$ denote the regularization parameters, 
$\theta\sim\mathcal{P}_{\theta}$ and the distribution of $z$ depends on $(x, y, \theta)$. 
}
\end{exm}

\vspace{-0.3cm}
Stochastic minimax problem  with decision-dependent distribution is proposed by Wood and Dall’Anese \cite{wood2023stochastic}. 
The authors introduce the notion of performative equilibrium point which is the saddle point for the stochastic minimax problem whose probability distribution is induced by itself, and provide sufficient conditions that guarantee the existence and uniqueness of the  performative equilibrium point. Algorithmically, they
propose (stochastic) primal-dual algorithm for seeking the performative equilibrium point of SMDD. Under constant and dynamic stepsize policies, the authors show that the proposed algorithm achieves a linear convergence rate and an $\mathcal{O}(\frac{1}{\sqrt{t}})$ convergence rate,  respectively. Moreover, when $\mathcal{L}(\cdot)$ is strongly convex-strongly concave, a zeroth-order algorithm for finding the saddle point of SMDD is also proposed. 
Narang et al.~\cite{narang2023multiplayer} study a decision-dependent  stochastic game model and introduce the notion of performatively stable equilibrium, which is Nash equilibrium of the game whose probability distribution is induced by itself. The authors show the existence and uniqueness of the performatively stable equilibrium
when the stochastic game is strongly monotone, and propose the repeated retraining and repeated (stochastic) gradient methods for finding performatively stable equilibrium. When the decision-dependent stochastic game is strongly monotone, they propose an adaptive stochastic gradient method for finding Nash equilibrium. When the distribution map follows a location-scale model, i.e., $\mathcal{D}(x) = Ax + \xi$ with $\xi \sim \mathbb{P}$, the authors show that the proposed algorithm achieves a  convergence rate $\mathcal{O}(\frac{1}{\sqrt{t}})$.
In the setting of monotonicity of the decision-dependent stochastic games, Wood et al. \cite{wood2024solvingdecisiondependentgameslearning} propose a two-stage method for seeking Nash equilibrium of the stochastic game, which estimates the distribution map first and then solves the game by stochastic gradient descent method.

Note that the objective function of SMDD is generally nonconvex-nonconcave given the probability distribution of the random variable depends on $x$ and $y$, we view the problem as minimization of the following nonconvex primal function
\begin{equation}\label{eq:max_value}
	\Phi(x):=\max_{y\in \mathcal{Y}}\mathcal{L}(x,y).
\end{equation} 
We focus on algorithms for finding stationary points of $\Phi(x)$ rather than performative equilibrium points \cite{wood2023stochastic}.
As far as we are concerned, the contributions of this paper can be summarized as follows.
\begin{itemize}
	\item 
	we propose an adaptive stochastic gradient descent ascent algorithm (ASGDA). In each iteration, ASGDA updates the estimation of the distribution map $\mathcal{D}(\cdot)$ adaptively by using the data sampling from the true distribution and updates the decision variables with  stochastic gradient descent ascent iteration equipped with two-timescale stepsizes. 
	When the objective function is nonconvex-strongly concave and the distribution map follows a location-scale model, i.e., $\mathcal{D}(x,y) := Ax + By + \xi$ with the random variable $\xi$ following unknown distribution $\mathbb{P}$, we show that ASGDA finds an $\epsilon$-stationary point  within $\mathcal{O}\left(\epsilon^{-\left(4+\delta\right)} \right)$ stochastic gradient evaluations for $\forall\delta>0$. When the objective function is nonconvex-concave and the distribution map follows a location-scale model, we employ the gradient of the Moreau envelope of the primal function as the stability metric and show that ASGDA finds an $\epsilon$-stationary point within $\mathcal{O}(\epsilon^{-8})$ stochastic gradient evaluations. 
	\item For SMDD-II, we propose an alternating adaptive stochastic gradient descent ascent algorithm (AASGDA), which updates the minimization and maximization variables alternately. When the objective function is nonconvex in $x$ and satisfies Polyak-{\L}ojasiewicz (P{\L}) inequality in $y$ and the distribution map follows a location-scale model, we show that the stochastic gradient complexity of AASGDA is $\mathcal{O}\left(\kappa_{y}^{4}\epsilon^{-4}\right)$, where $\kappa_y$ denotes the condition number. While ASGDA and AASGDA may seem like natural extensions of SGDA, the challenge lies in handling of the iteration-varying probability distribution and the biased stochastic gradient induced by the dynamic learning process of the unknown distribution map in  convergence analysis.
\end{itemize}
 
The structure of this paper is organized as follows.
Section \ref{section2} introduces ASGDA and AASGDA algorithms and presents necessary assumptions on SMDD \eqref{eq:f_minimax_D}.
Section \ref{section3} presents the convergence analysis of ASGDA, where
Subsection \ref{subsection3_1} focuses on SMDD \eqref{eq:f_minimax_D} with nonconvex-strongly concave objective function and
Subsection \ref{subsection3_2}  focuses on  SMDD \eqref{eq:f_minimax_D} with nonconvex-concave objective function. Section \ref{subsection4} presents the convergence of AASGDA for SMDD \eqref{eq:f_minimax_D} with nonconvex-P{\L} objective function.
Section \ref{section5} verifies the effectiveness of the proposed algorithms on synthetic examples and two real-world applications introduced in Examples \ref{example_1} and \ref{example_2}. 
All the proofs of the lemmas are delegated to the Appendix.

\noindent\textbf{Notation.} $\mathbb{R}^{n}$ denotes the $n$-dimensional Euclidean space endowed with norm $\|x\|=\sqrt{\langle x,x\rangle}$. For a matrix $X\in\mathbb{R}^{n\times m}$, $\left\|X\right\|_F$ denotes the Frobenious norm. 
For a differentiable function $f(x, y)$, $\nabla_{x}f(x, y)$ and $\nabla_{y}f(x, y)$ denote the partial gradients with respect to $x$ and $y$ at $(x,y)$, respectively. $\operatorname{proj}_{\mathcal{Y}}(y)$ denotes the Euclidean projection of a point $y\in\mathbb{R}^n$ onto the set $\mathcal{Y}\subset\mathbb{R}^n$. We denote $a=\mathcal{O}(b)$ if $|a|\leq C|b|$ for some constant $C>0$ and $[M]:=\{1,2, \ldots, M\}$.

At the end of this section, we review the studies on stochastic first-order methods for smooth nonconvex minimax optimization.
For nonconvex-(strongly) concave stochastic minimax problems, a natural method is
stochastic gradient descent with max-oracle (SGDmax) \cite{jin2019minmax}, which alternates between a stochastic gradient descent step on
$x$ and an (approximate) maximization step on $y$.
SGDmax is a double-loop algorithm that finds an $\epsilon$-stationary point within $\tilde{\mathcal{O}}\left(\kappa_y^{3}\epsilon^{-4}\right)$\footnote{$\tilde{\mathcal{O}}(\cdot)$ denotes to hide both absolute constants and log factors.} and $\mathcal{O}\left(\epsilon^{-8}\right)$ stochastic gradient evaluations for nonconvex-strongly concave and nonconvex-concave stochastic minimax problems, respectively.
A number of double-loop algorithms based on either the inexact proximal point framework \cite{rafique2022weakly,zhang2022sapd+} or variance reduction techniques \cite{luo2020stochastic,xu2021enhanced} have been proposed, which achieve improved complexity guarantees.
Compared to double loop algorithms,
Lin et al. \cite{lin2020gradient} propose a single loop algorithm, stochastic gradient descent ascent algorithm (SGDA), in which the stepsize $\eta_y$ is larger than $\eta_x$ to force $y$ moves faster than $x$. 
SGDA can find an $\epsilon$-stationary point within $\mathcal{O}(\kappa_y^3 \epsilon^{-4})$ and $\mathcal{O}\left(\epsilon^{-8}\right)$ stochastic gradient evaluations for nonconvex-strongly concave and nonconvex-concave stochastic minimax problems, respectively. 
Zhang et al. \cite{zhang2020single} propose a smoothed GDA algorithm, which stabilizes GDA by introducing the smoothing scheme and achieves $\mathcal{O}(\epsilon^{-4})$ iteration complexity for nonconvex-concave minimax problems.
Huang et al. \cite{huang2022accelerated} propose a momentum acceleration version of SGDA (MSGDA), which achieves the optimal $\tilde{\mathcal{O}}(\epsilon^{-3})$ stochastic gradient complexity for nonconvex-strongly concave stochastic minimax problems.
Zhang and Xu \cite{zhang2024accelerated} extend MSGDA by introducing a regularization term to enhance the smoothness and concavity of the objective function, which achieves $\tilde{\mathcal{O}}(\epsilon^{-6.5})$ stochastic gradient complexity for nonconvex-concave stochastic minimax problems. 
Xu et al. \cite{xu2024stochastic} propose a SGDA method with backtracking~(SGDA-B), which estimates the unknown Lipschitz constant with backtracking technique. SGDA-B finds an $\epsilon$-stationary point within $\mathcal{O}\left(L \kappa^{3} \epsilon^{-4} \log (1 / p)\right)$ and $\tilde{\mathcal{O}}\left(L^{4} \epsilon^{-7} \log (1 / p)\right)$ stochastic gradient evaluations with probability at least $1 - p$ for nonconvex-strongly concave and nonconvex-concave stochastic minimax problems, respectively.
More recently, Jiang et al. \cite{jiang2025singleloopvariancereducedstochasticalgorithm} focus on finite-sum minimax problems and propose a probabilistic variance-reduced smoothed gradient descent-ascent (PVR-SGDA) algorithm, which achieves $\mathcal{O}(\epsilon^{-4})$ stochastic gradient complexity when the objective function is nonconvex-concave. 
On the other hand, Bo{\c{t}} et al. \cite{boct2023alternating} propose an alternating version of SGDA, which updates the minimization variable and maximization variable alternately and achieves the same complexity as SGDA. 

For nonconvex-P{\L} minimax optimization problems, Nouiehed et al. \cite{nouiehed2019solving} propose a multi-step GDA method, which alternates between multiple gradient ascent steps on $y$ and a gradient descent step on $x$. 
Chen et al. \cite{chen2022faster} focus on finite sum minimax optimization under nonconvex-P{\L} condition and  propose a variance reduction version of stochastic GDA along with its acceleration version.
Yang et al. \cite{yang2022faster} propose a stochastic alternating gradient descent ascent algorithm (AGDA) and a stochastic smoothed AGDA, which can find $\epsilon$-stationary points of nonconvex-P{\L} stochastic minimax problems within $\mathcal{O}(\kappa_y^4\epsilon^{-4})$ and $\mathcal{O}(\kappa_y^2\epsilon^{-4})$ stochastic gradient evaluations, respectively. Huang \cite{huang2023enhancedadaptivegradientalgorithms} proposes a momentum acceleration version of stochastic GDA, which achieves an improved $\tilde{\mathcal{O}}(\epsilon^{-3})$ stochastic gradient complexity
for finding $\epsilon$-stationary points of nonconvex-P{\L} stochastic minimax problems. 
More recently, Laguel et al. \cite{laguel2024high} establish the first high-probability complexity guarantee for the sm-AGDA algorithm under sub-Gaussian gradient noise, 
showing that sm-AGDA finds $\epsilon$-stationary points of nonconvex-P{\L} stochastic minimax problems within $\tilde{\mathcal{O}}\!\bigl(\epsilon^{-4}+\epsilon^{-2}\log(1/q)\bigr)$ stochastic gradient evaluations with probability at least $1-q$.

\vspace{-0.2cm}
\section{Preliminaries and Algorithms }\label{section2}
\vspace{-0.3cm}
We first present the ASGDA algorithm tailored for SMDD-I, which reads as follows.
\begin{algorithm}[H]
	\caption{ Adaptive Stochastic Gradient Descent Ascent  Algorithm (ASGDA)}
	\label{algorithm:ASGDA}
	\begin{algorithmic}[1]
		\STATE \textbf{Initialization:}
		$x^{0}\in\mathbb{R}^{n}, y^{0}\in\mathbb{R}^{m}$, stepsizes $\eta_x, \eta_y$, batchsize $M>0$, distribution $\mathcal{D}^{0}(\cdot)$.
		\FOR{$t=0$ \TO $T$}
		\STATE \textbf{Query the environment:} Draw a collection of i.i.d. data samples $z_i^t \sim \mathcal{D}(x^t,y^t)$, $i\in[M].$
		\STATE \textbf{Calculate} $G_{x}\left(x^t, y^t, z_i^t, \mathcal{D}^{t}(\cdot)\right),\, G_{y}\left(x^t, y^t, z_i^t, \mathcal{D}^{t}(\cdot)\right), i\in[M].$
		\STATE \textbf{Update decision variables:}
		{\small\begin{equation}\label{equ:alg_x}
					x^{t+1}=x^{t}-\eta_{x}\left[\frac{1}{M}\sum_{i=1}^{M} G_{x}\left(x^t, y^t, z_i^t, \mathcal{D}^{t}(\cdot)\right)\right],
		\end{equation}}
		{\small\begin{equation}\label{equ:alg_y}
					y^{t+1}=\operatorname{proj}_{\mathcal{Y}}\left(y^{t}+\eta_{y}\left[\frac{1}{M}\sum_{i=1}^M G_{y}\left(x^t, y^t, z_i^t, \mathcal{D}^{t}(\cdot)\right)\right]\right).
		\end{equation}}
	\vspace{-0.4cm}
		\STATE \textbf{Update distribution map:}\\ 
			{\small\begin{equation}\label{equ:alg_z}
				\mathcal{D}^{t+1}(\cdot)=\pi\left(\mathcal{D}^t(\cdot),  x^{t}, y^{t}, \left\{z_i^{t}\right\}_{i=1}^{M}\right), z_i^t \sim \mathcal{D}(x^t,y^t), i\in[M].
		\end{equation}}  
		\ENDFOR
		\STATE Randomly draw $\bar{x}$ from $\{x^t\}_{t=1}^{T+1}$ at uniform and return $\bar{x}$.
	\end{algorithmic}
\end{algorithm}

In Algorithm \ref{algorithm:ASGDA}, Step 4 computes the stochastic gradient of the objective function, where $\mathcal{D}^{t}(\cdot)$ represents an estimate of
$\mathcal{D}(\cdot)$, $\{z_i^t\}_{i=1}^{M}$ denote the independent and identically distributed (i.i.d.) samples from the true distribution $\mathcal{D}(x^t,y^t)$. Step
5 performs a standard variable updating process, which updates the decision variables simultaneously by stochastic gradient descent ascent method. Step 6 learns the unknown distribution map $\mathcal{D}(\cdot)$, which updates the current estimate of the distribution map via an online estimation oracle $\pi(\cdot)$ using data samples from the true distribution. As the algorithm proceeds, the dynamically estimated distribution map $\mathcal{D}^{t}(\cdot)$ converges to the true distribution map $\mathcal{D}(\cdot)$.

Next, we present the AASGDA algorithm tailored for SMDD-II, as detailed in Algorithm \ref{algorithm:AASGDA}.
\begin{algorithm}[H] 
	\caption{Alternating Adaptive Stochastic Gradient Descent Ascent  Algorithm (AASGDA)}
	\label{algorithm:AASGDA}
	\begin{algorithmic}[1]
		\STATE \textbf{Initialization:} $x^{0}\in\mathbb{R}^{n}, y^{0}\in\mathbb{R}^{m}$, stepsizes $\eta_x, \eta_y$, distribution $\mathcal{D}^{0}(\cdot)$.
		\FOR{$t=0$ \TO $T$}
		\STATE \textbf{Query the environment:} Draw a data sample
		$z_x^t \sim \mathcal{D}(x^t,y^t)$.
		\STATE \textbf{Update decision variable $x$:}
		{\small\begin{equation}\label{equ:PL_alg_x}
				\begin{aligned}
					&x^{t+1}=x^{t}-\eta_{x}\left[G_{x}\left(x^t, y^t, z_x^t, \mathcal{D}^{t}(\cdot)\right)\right].
				\end{aligned}
		\end{equation}}	
		\STATE \textbf{Query the environment:} Draw a data sample
		$z_y^t \sim \mathcal{D}(x^{t+1},y^t)$.
		\STATE \textbf{Update decision variable $y$:}
		{\small\begin{equation}\label{equ:PL_alg_y}
				\begin{aligned}
					y^{t+1}=y^{t}+\eta_{y}\left[G_{y}\left(x^{t+1}, y^t, z_y^t, \mathcal{D}^{t}(\cdot)\right)\right].
				\end{aligned}
		\end{equation}}
		\STATE \textbf{Update distribution map:}\\ 
		{\small\begin{equation}\label{equ:PL_alg_z}
				\mathcal{D}^{t+1}(\cdot)=\pi\left(\mathcal{D}^t(\cdot), x^{t+1}, y^{t}, z^t\right), z^t \sim \mathcal{D}(x^{t+1},y^t).
		\end{equation}} 
		\ENDFOR
		\STATE Randomly draw $\bar{x}$ from $\{x^t\}_{t=1}^{T+1}$ at uniform and return $\bar{x}$.
	\end{algorithmic}
\end{algorithm}

Different from Algorithm \ref{algorithm:ASGDA}, Algorithm \ref{algorithm:AASGDA} updates the decision variables by alternating stochastic gradient descent ascent method in Step 4 and Step 6. Note that the minimization variable $x$ and maximization variable $y$ are updated sequentially, $x^{t+1}$ in step 4 is used to update $y$ in Step 6. 

In what follows, we present some necessary assumptions on SMDD \eqref{eq:f_minimax_D}. 
\begin{ass}\label{ass-primal}
	The primal function $\Phi(x)$ is bounded below.
\end{ass}
\begin{ass}\label{ass_stronglyconcave} 
	$\mathcal{L}(\cdot)$ is $\ell$-smooth, and $\mathcal{L}(\cdot,y)$ is $L$-Lipschitz continuous for $\forall y\in\mathbb{R}^{m}$. 
\end{ass}
\begin{ass}\label{ass_bounded_moment}
	For $\forall x\in\mathbb{R}^{n}, y\in\mathbb{R}^{m}$, there exists a constant $L_1 > 0$ such that
	\begin{equation*}
		\begin{aligned}
			&\mathbb{E}_{z\sim\mathcal{D}(x,y)}\left\|\nabla_{x}l(x,y,z)\right\|\leq L_1,\\
			&\mathbb{E}_{z\sim\mathcal{D}(x,y)}\left\|\nabla_{z}l(x,y,z)\right\|\leq L_1.
		\end{aligned}
	\end{equation*}	
\end{ass}
\begin{ass}{\em\cite{narang2023multiplayer}}\label{ass:l_variance_bound}
	For $\forall x\in\mathbb{R}^{n}, y\in\mathbb{R}^{m}$, there exists a constant $\sigma > 0$ such that
	\begin{equation*}
		\begin{aligned}
			&\mathbb{E}_{z\sim\mathcal{D}(x,y)}\left[\|\nabla_{x,z}l(x,y,z) - \mathbb{E}_{z\sim\mathcal{D}(x,y)}[\nabla_{x,z}l(x,y,z)]\|^2\right] \leq \sigma^{2},\\
			&\mathbb{E}_{z\sim\mathcal{D}(x,y)}\left[\|\nabla_{y}l(x,y,z) - \mathbb{E}_{z\sim\mathcal{D}(x,y)}[\nabla_{y}l(x,y,z)]\|^2\right] \leq \sigma^{2}.
		\end{aligned}
	\end{equation*}
\end{ass}
\begin{ass}{\em\cite{narang2023multiplayer}}\label{ass-objective}
	There exists a probability measure $\mathbb{P}$ and an $L_{0}$-$Lipschitz$ continuous and $\ell_{0}$-smooth function $\psi(\cdot):\mathbb{R}^{n}\times\mathbb{R}^{m}\rightarrow \mathbb{R}^{d}$ satisfying
	$$\mathcal{D}(x,y)=\psi(x,y) + \xi, \;\;\;\;\xi \sim \mathbb{P}.$$
\end{ass}

Assumption \ref{ass-primal} guarantees the feasibility of SMDD \eqref{eq:f_minimax_D}. Assumption \ref{ass_stronglyconcave} 
is a standard condition for stochastic minimax problems
\cite{huang2023adagda,li2022tiada,lin2020gradient}, which is implied by Assumption \ref{ass-objective} and smoothness of $l(\cdot)$. 
Assumption \ref{ass_bounded_moment} requires the stochastic gradients to have uniformly bounded first moment.
Assumption \ref{ass:l_variance_bound} is the boundedness of the variance of the stochastic gradient. 
Assumption \ref{ass-objective} asserts that the distribution of $z$ follows a regression model of the decision variables, which allows us to measure the estimation error of the distribution map conveniently.

Under Assumptions \ref{ass_stronglyconcave} and \ref{ass-objective}, by the chain rule for composite function, the gradient of $\mathcal{L}(\cdot)$ at $(x, y)$
\begin{align*}
	&\nabla_{x}\mathcal{L}(x,y)=\underset{z\sim \mathcal{D}(x,y)}{\mathbb{E}}\left[\nabla_{x}l(x,y,z)+\nabla_{x}\psi(x,y)^{\top}\nabla_{z}l(x,y,z)\right], \\
	&\nabla_{y}\mathcal{L}(x,y)=\underset{z\sim\mathcal{D}(x,y)}{\mathbb{E}}\left[\nabla_{y}l(x,y,z)+\nabla_{y}\psi(x,y)^{\top}\nabla_{z}l(x,y,z)\right],
\end{align*}
and the corresponding stochastic gradient of $\mathcal{L}(\cdot)$ at $(x, y)$
\begin{equation}\label{equ:stograd_unbia_x}
	G_x(x,y,z)=\nabla_{x}l(x,y,z)+\left(\nabla_{x}\psi(x,y)\right)^{\top}\nabla_{z}l(x,y,z),
\end{equation}
\begin{equation}\label{equ:stograd_unbia_y}
	G_y(x,y,z)=\nabla_{y}l(x,y,z)+\left(\nabla_{y}\psi(x,y)\right)^{\top}\nabla_{z}l(x,y,z),
\end{equation}
where $z\sim\mathcal{D}(x,y)$.
As the function $\psi(\cdot)$ is unknown, $G_x(x,y,z)$ and $G_y(x,y,z)$ in \eqref{equ:stograd_unbia_x} and \eqref{equ:stograd_unbia_y} are not available. We approximate them by the following stochastic gradients
\begin{align}
	\label{eq:adaptive_stograd}
	&G_{x}(x^t, y^t, z^t, \mathcal{D}^{t}(\cdot))=\nabla_{x}l(x^t,y^t,z^t)+\left(\nabla_{x}\psi^t(x^t,y^t)\right)^{\top}\nabla_{z}l(x^t,y^t,z^t), \\
	&G_{y}(x^t, y^t, z^t, \mathcal{D}^{t}(\cdot))=\nabla_{y}l(x^t,y^t,z^t)+\left(\nabla_{y}\psi^t(x^t,y^t)\right)^{\top}\nabla_{z}l(x^t,y^t,z^t),\\
	&G_{x}(x^t, y^t, z_x^t, \mathcal{D}^{t}(\cdot))=\nabla_{x}l(x^t,y^t,z_x^t)+\left(\nabla_{x}\psi^t(x^t,y^t)\right)^{\top}\nabla_{z}l(x^t,y^t,z_x^t), \\
	&G_{y}(x^{t+1}, y^t, z_y^t, \mathcal{D}^{t}(\cdot))=\nabla_{y}l(x^{t+1},y^t,z_y^t)+\left(\nabla_{y}\psi^t(x^{t+1},y^t)\right)^{\top}\nabla_{z}l(x^{t+1},y^t,z_y^t),
\end{align}
where $z^t\sim\mathcal{D}(x^t,y^t)$, $z_x^t\sim\mathcal{D}(x^t,y^t)$,  $z_y^t\sim\mathcal{D}(x^{t+1},y^t)$ and $\psi^{t}(\cdot)$ is the current estimate of $\psi(\cdot)$ at the $t$-th iteration. Obviously, $G_{x}(x^t, y^t, z^t, \mathcal{D}^{t}(\cdot))$ and $G_{y}(x^t, y^t, z^t, \mathcal{D}^{t}(\cdot))$ are biased stochastic gradients of $\nabla_{x}\mathcal{L}(x^{t},y^{t})$ and $\nabla_{y}\mathcal{L}(x^{t},y^{t})$, respectively, and $G_x(x^t,y^t,z_x^t)$ and $G_y(x^{t+1},y^t,z_y^t)$ are biased stochastic gradients of $\nabla_{x}\mathcal{L}(x^{t},y^{t})$ and $\nabla_{y}\mathcal{L}(x^{t+1},y^{t})$, respectively.

At the end of this section, we provide some technical lemmas. For ease of the notation, we use the following notation throughout the paper,
\begin{equation*}
	\begin{array}{c}
		G^t_{x}(x, y, z):=G_{x}(x, y, z,  \mathcal{D}^{t}(\cdot)),\;\;\; G^t_{y}(x, y, z):=G_{y}(x, y, z,  \mathcal{D}^{t}(\cdot)),\\
		\mathbb{E}_{t}\left[\,\cdot\,\right]:=\mathbb{E}\left[\,\cdot\,| (x^t,y^t,\psi^t(\cdot))\right],\;\;\; \mathbb{E}_{\widehat{t}}\left[\,\cdot\,\right]:=\mathbb{E}\left[\,\cdot\,| (x^{t+1},y^t,\psi^t(\cdot))\right],\\
		y^{\star}(\cdot)\in\operatorname{argmax}_{y\in\mathcal{Y}}\mathcal{L}(\cdot, y).
	\end{array}
\end{equation*}

\begin{lem}\label{lem:expvar}
	Suppose that Assumptions \ref{ass_bounded_moment}--\ref{ass-objective} hold. Then
	\begin{small}
		\begin{align}
			&\left\|\mathbb{E}_{t}\left[\frac{1}{M}\sum_{i=1}^{M}G_{x}^{t}(x^{t},y^{t},z_{i}^{t})\right]-
			\nabla_{x}\mathcal{L}(x^{t},y^{t})\right\|\leq L_{1}\left\|\nabla_{x}\psi^{t}(x^t,y^t)-\nabla_{x}\psi(x^t,y^t) \right\|_{F}, \tag{a}\\
			&\mathbb{E}_{t}\left\|\frac{1}{M}\sum_{i=1}^{M}G_x^{t}(x^t,y^t,z^t_{i})\right\|^{2}\leq \left\|\mathbb{E}_{t}\left[\frac{1}{M}\sum_{i=1}^{M} G_{x}^{t}\left(x^{t}, y^{t}, z_{i}^{t}\right)\right]\right\|^{2}+\left(1+\left\|\nabla_{x}\psi^{t}(x^t,y^t)\right\|_F^{2} \right)\frac{\sigma^{2}}{M}, \tag{b}\\
			&\left\|\mathbb{E}_{t}\left[\frac{1}{M}\sum_{i=1}^{M}G_y^{t}(x^t,y^t,z^t_{i})\right]-\nabla_{y}\mathcal{L}(x^{t},y^{t})\right\|\leq L_{1}\left\|\nabla_{y}\psi^{t}(x^t,y^t)-\nabla_{y}\psi(x^t,y^t) \right\|_{F}, \tag{c}\\
			&\mathbb{E}_{t}\left\|\frac{1}{M}\sum_{i=1}^{M}G^{t}_y(x^t,y^t,z^t_{i})
			-\mathbb{E}_{t}\left[\frac{1}{M}\sum_{i=1}^{M}G_y^{t}(x^t,y^t,z^t_{i})\right] \right\|^{2}\leq  \left(1+\left\|\nabla_{y}\psi^{t}(x^t,y^t)\right\|_F^{2} \right)\frac{\sigma^{2}}{M},\tag{d}
		\end{align}
	\end{small}where $z_i^t\sim\mathcal{D}(x^t,y^t), i\in[M]$, $\mathbb{E}_{t}\left[\,\cdot\,\right]=\mathbb{E}\left[\,\cdot\,| (x^t,y^t,\psi^t(\cdot))\right]$. 
\end{lem}
\begin{proof}
	The proof of Lemma \ref{lem:expvar} is provided in Appendix \ref{Appendix_for_Section2_1_Lem2}.
\end{proof}

\begin{lem}\label{lem:expvar_PL}
	Suppose that Assumptions \ref{ass_bounded_moment}--\ref{ass-objective} hold. Then
	\begin{align}
		&\left\| \mathbb{E}_t \left[ G_x^t(x^t, y^t, z_x^t) \right] - \nabla_x \mathcal{L}(x^t, y^t) \right\| \leq L_1 \left\| \nabla_x \psi^t(x^t, y^t) - \nabla_x \psi(x^t, y^t) \right\|_F,\tag{a}\\
		&\mathbb{E}_t \left\| G_x^t(x^t, y^t, z_x^t) \right\|^2 \leq \left\| \mathbb{E}_t \left[ G_x^t(x^t, y^t, z_x^t) \right] \right\|^2 + \left( 1 + \left\| \nabla_x \psi^t(x^t, y^t) \right\|_F^2 \right) \sigma^2,\tag{b}\\
		&\left\| \mathbb{E}_{\widehat{t}} \left[ G_y^t(x^{t+1}, y^t, z_y^t) \right] - \nabla_y \mathcal{L}(x^{t+1}, y^t) \right\| \leq L_1 \left\| \nabla_y \psi^t(x^{t+1}, y^t) - \nabla_y \psi(x^{t+1}, y^t) \right\|_F, \tag{c}\\
		&\mathbb{E}_{\widehat{t}} \left\| G_y^t(x^{t+1}, y^t, z_y^t)\right\|^{2}\leq \left\|\mathbb{E}_{\widehat{t}} \left[ G_y^t(x^{t+1}, y^t, z_y^t) \right] \right\|^2+\left( 1 + \left\| \nabla_y \psi^t(x^{t+1}, y^t) \right\|_F^2 \right) \sigma^2, \tag{d}
	\end{align}
	where $z_x^t\sim\mathcal{D}(x^t,y^t)$, $z_y^t\sim\mathcal{D}(x^{t+1},y^t)$, $\mathbb{E}_{t}\left[\,\cdot\,\right]=\mathbb{E}\left[\,\cdot\,| (x^t,y^t,\psi^t(\cdot))\right]$ and $\mathbb{E}_{\widehat{t}}[\,\cdot\,]=\mathbb{E}\left[\,\cdot\,| (x^{t+1},y^t,\psi^t(\cdot))\right]$. 
\end{lem}
\begin{proof}
	The proof of Lemma \ref{lem:expvar_PL} is similar to Lemma \ref{lem:expvar}.
\end{proof}

\section{Convergence analysis of ASGDA}\label{section3}

In this section, we study the convergence of ASGDA, where Subsections \ref{subsection3_1} and \ref{subsection3_2} focus on SMDD with nonconvex-strongly concave objective function and SMDD with nonconvex-concave objective function, respectively.

\subsection{Nonconvex-strongly concave SMDD}\label{subsection3_1}
Throughout this subsection, we assume that the following assumption holds.
\begin{ass}{\em\cite{lin2020gradient}}\label{ass:stronglyconcave}The objective function and constraint set satisfy
	
	{\rm(a)} For $\forall x\in \mathbb{R}^n$, $\mathcal{L}(x,\cdot)$ is $\mu$-strongly concave with $\mu>0$.
	
	{\rm(b)} The constraint set $\mathcal{Y}$ is a compact convex set with diameter $D > 0$.
\end{ass}

Assumption \ref{ass:stronglyconcave}~(a) holds, if the smooth function $l(\cdot)$ is strongly concave in $y$ and the distribution map $\mathcal{D}(x,y)=\psi(x,y)+\xi$ is Lipschitz continuous and satisfies some appropriate stochastic dominance condition \cite{ wood2023stochastic,miller2021outside}.
Assumption \ref{ass:stronglyconcave} guarantees the primal function $\Phi(\cdot)$ is smooth. Therefore, we employ the norm of the gradient of $\Phi(\cdot)$ to define the stationary point of SMDD.
\begin{defn}\label{def_solution_stronglyconcave}
	A point $\hat{x}$ is an $\epsilon$-stationary point of
	SMDD \eqref{eq:f_minimax_D} if $\|\nabla \Phi(\hat{x})\| \leq \epsilon$.
\end{defn}
Next, we present two lemmas that play a crucial role in proving convergence. The first lemma characterizes the evolution of the primal function $\Phi(\cdot)$ in expectation over the iteration of $x$.
\begin{lem}\label{lem_value}
	Let $$\delta_{t}=\mathbb{E}\Vert y^{\star}(x^{t})-y^{t}\Vert^{2}.$$ Suppose that Assumptions \ref{ass-primal}--\ref{ass:stronglyconcave} hold and the stepsize $\eta_x\leq\frac{1}{16\kappa_{y}\ell}$, where $\kappa_{y}=\ell/\mu$ is the condition number and $\ell$ is the smoothness parameter of $\mathcal{L}(\cdot)$. Then	
	\begin{small}
	\begin{equation}\label{eq:stronglyconcave_descentPhi}
		\begin{aligned}
			\mathbb{E}\left[\Phi(x^{t+1})\right] &\leq \mathbb{E}\left[\Phi(x^t)\right] - \frac{1}{4}\eta_x \mathbb{E}\left[\left\|\nabla\Phi(x^t)\right\|^2\right] + \frac{5}{4}\eta_x\ell^2\delta_t \\
			&+ \left(\eta_x + 2\kappa_y\ell\eta_x^2\right)L_1^2\mathbb{E}\left[\left\|\nabla_x\psi^t(x^t,y^t) - \nabla_x\psi(x^t,y^t)\right\|_F^2\right]\\
			&+ \frac{\kappa_y\ell\eta_x^2\sigma^2}{M}\mathbb{E}\left[1 + \left(\left\|\nabla_x\psi^t(x^t,y^t)\right\|_F^2\right)\right].
		\end{aligned}
	\end{equation}
	\end{small}
\end{lem}
\begin{proof}
The proof of Lemma \ref{lem_value} is provided in Appendix \ref{Appendix_for_Section3_1_Lem1}.
\end{proof}

The second lemma establishes a connection between $\mathbb{E}\Vert y^{\star}(x^{t})-y^{t}\Vert^{2}$ and $\mathbb{E}\left\|\nabla\Phi(x^{t})\right\|^{2}$ along with the estimation error of the distribution map by a recursion inequality.
\begin{lem}\label{lemma_delta_recursion}
	Let $\delta_{t} $ be defined in Lemma \ref{lem_value}.
	Suppose that Assumptions \ref{ass-primal}--\ref{ass:stronglyconcave} hold and the stepsize $\eta_{y}=\frac{1}{2(\ell+\mu)}$, where $\ell$ is the smoothness parameter of $\mathcal{L}(\cdot)$. Then
	\begin{small}
	\begin{equation}\label{eq:f_delta}
		\begin{aligned}
			\delta_{t}
			\leq&\left(1-\frac{1}{8\kappa_{y}}+16\ell^2(\kappa_{y})^{2}(\kappa_{y}+1)\eta_{x}^{2} \right)\delta_{t-1}+16(\kappa_{y}+1)\kappa_{y}^{2}\eta_{x}^{2}\mathbb{E}\left[\Vert\nabla\Phi(x^{t-1})\Vert^2\right]\\
			&+8L_1^2(\kappa_{y}+1)\kappa_{y}^{2}\eta_{x}^{2}
			\mathbb{E}\left[\Vert\nabla_{x}\psi^{t-1}(x^{t-1},y^{t-1})-\nabla_{x}\psi(x^{t-1},y^{t-1}) \Vert^2_{F}\right]\\
			&+\left(4(\kappa_{y}+1)\kappa_{y}^{2}\eta_{x}^{2}\right)\mathbb{E}\left[1+\left\|\nabla_{x}\psi^{t-1}(x^{t-1},y^{t-1})\right\|_F^{2} \right]\frac{\sigma^{2}}{M}\\
			&+\left(\frac{1}{(\mu+\ell)^{2}}+\frac{1}{\mu \ell}\right)L_{1}^{2}\mathbb{E}\left[\left\|\nabla_{y}\psi^{t-1}(x^{t-1},y^{t-1})-\nabla_{y}\psi(x^{t-1},y^{t-1}) \right\|_{F}^{2}\right]\\
			&+\left(\frac{1}{2(\mu+\ell)^{2}}\right)\mathbb{E}\left[1+\left\|\nabla_{y}\psi^{t-1}(x^{t-1},y^{t-1})\right\|_F^{2} \right]\frac{\sigma^{2}}{M}.
		\end{aligned}
	\end{equation}
	\end{small}
\end{lem}
\begin{proof}
The proof of Lemma \ref{lemma_delta_recursion} is provided in Appendix \ref{Appendix_for_Section3_1_Lem2}.
\end{proof}

We are ready to study the convergence of ASGDA.
\begin{thm}\label{theorem_stronglyconcave}
	Suppose that Assumptions \ref{ass-primal}--\ref{ass:stronglyconcave} hold
	and the stepsizes $\eta_{x}=\frac{1}{40(\kappa_{y}+1)^{2}\ell}$, $\eta_{y}=\frac{1}{2(\ell+\mu)}$. Then
	\begin{small}
		\begin{equation}\label{eq:theorem1_concluding}
				\begin{aligned}
				\frac{1}{T}\sum_{t=0}^{T-1} \mathbb{E}\left[\left\|\nabla \Phi\left(x^{t}\right)\right\|^{2}\right]
				\leq& \frac{320(\kappa_{y}+1)^{2}\ell \left(\Phi\left(x^{0}\right)-\mathbb{E}\left[\Phi\left(x^{T}\right)\right]\right)+100\ell^{2}\kappa_{y} D^{2}}{T}\\
				&
				+\left(101\kappa_{y} L_0^2+\frac{L_1^2}{2}+7\right)\frac{\sigma^2}{M}\\
				&+\left(10L_1^2+\frac{\sigma^2}{M}\right)\frac{\sum_{t=0}^{T-1}\mathbb{E}\left[\|\nabla_{x}\psi^{t}(x^{t},y^{t})-\nabla_{x}\psi(x^{t},y^{t})\|_F^{2}\right]}{T}  \\
				&+100\kappa_{y}\left((\kappa_{y}+1)L_1^2+\frac{\sigma^2}{M}\right)\frac{\sum_{t=0}^{T-1} \mathbb{E}\left[\left\|\nabla_{y}\psi^{t}(x^{t},y^{t})-\nabla_{y}\psi(x^{t},y^{t})\right\|_{F}^{2}\right]}{T}.
			\end{aligned}
		\end{equation}
	\end{small}
\end{thm}
\begin{proof}
	By the definitions of $D$ in Assumption \ref{ass:stronglyconcave} (b) and
	$\delta_{t}$  in Lemma \ref{lem_value}, we have $\delta_{0}<D^{2}$. Using inequality \eqref{eq:f_delta} recursively,
		\begin{equation*}\label{equ:stronglyconcave_th_1}
			\begin{aligned}
				\delta_{t} \leq & 16(\kappa_{y}+1)\kappa_{y}^{2}\eta_{x}^{2}\left(\sum_{j=0}^{t-1}\gamma^{t-1-j}\mathbb{E}\left\|\nabla\Phi\left(x^{j}\right)\right\|^{2}\right) \\
				& + 4(\kappa_{y}+1)\kappa_{y}^{2}\eta_{x}^{2}\sum_{j=0}^{t-1}\gamma^{t-1-j}\mathbb{E}\left[\frac{1}{M}\left(1+\left\|\nabla_{x}\psi^{j}(x^{j},y^{j})\right\|_{F}^{2}\right)\sigma^{2}\right] \\
				& + 8L_{1}^{2}(\kappa_{y}+1)\kappa_{y}^{2}\eta_{x}^{2}\left(\sum_{j=0}^{t-1}\gamma^{t-1-j}\mathbb{E}\left\|\nabla_{x}\psi^{j}(x^{j},y^{j})-\nabla_{x}\psi(x^{j},y^{j})\right\|_{F}^{2}\right) + \gamma^t D^{2} \\
				& + \frac{1}{2(\mu+\ell)^{2}}\sum_{j=0}^{t-1}\gamma^{t-1-j}\mathbb{E}\left[\frac{1}{M}\left(1+\left\|\nabla_{y}\psi^{j}(x^{j},y^{j})\right\|_{F}^{2}\right)\sigma^{2}\right] \\
				& + \left(\frac{1}{(\mu+\ell)^{2}}+\frac{1}{\mu\ell}\right)L_{1}^{2}\left(\sum_{j=0}^{t-1}\gamma^{t-1-j}\mathbb{E}\left\|\nabla_{y}\psi^{j}(x^{j},y^{j})-\nabla_{y}\psi(x^{j},y^{j})\right\|_{F}^{2}\right),
			\end{aligned}
		\end{equation*}where $\gamma=1-\frac{1}{8\kappa_{y}}+16(\kappa_{y}+1)\kappa_{y}^{2}\ell^{2}\eta_{x}^{2}.$
	
	Plugging the above inequality into inequality  \eqref{eq:stronglyconcave_descentPhi},
		\begin{equation*}
			\begin{aligned}
				\mathbb{E}\left[\Phi(x^{t+1})\right]&\leq \mathbb{E}\left[\Phi(x^{t})\right]-\frac{1}{4}\eta_{x}\mathbb{E}\left[\left\|\nabla\Phi(x^{t}) \right\|^{2} \right]+\frac{\kappa_{y}\ell\eta_{x}^{2}\sigma^{2}}{M}\mathbb{E}\left[1+\left(\Vert\nabla_{x}\psi^{t}(x^{t},y^{t})\Vert_F^{2}\right) \right]\\
				&+\frac{5}{4}\eta_{x}\ell^{2}\gamma^{t} D^{2}+20(\kappa_{y}+1) \kappa_{y}^{2} \eta_{x}^{3}\ell^{2}\left(\sum_{j=0}^{t-1} \gamma^{t-1-j} \mathbb{E}\left\|\nabla \Phi\left(x^{j}\right)\right\|^{2}\right)\\
				&+5 (\kappa_{y}+1)\kappa_{y}^{2} \eta_{x}^{3}\ell^{2} \sum_{j=0}^{t-1} \gamma^{t-1-j} \mathbb{E}\left[\frac{1}{M}\left(1+\left\|\nabla_{x}\psi^{j}(x^{j},y^{j})\right\|_{F}^{2}\right) \sigma^{2}\right]\\
				&+10L_1^{2}(\kappa_{y}+1) \kappa_{y}^{2} \eta_{x}^{3}\ell^{2}\left(\sum_{j=0}^{t-1} \gamma^{t-1-j}  \mathbb{E}\left\|\nabla_{x}\psi^{j}(x^{j},y^{j})-\nabla_{x}\psi(x^{j},y^{j})\right\|_{F}^{2}\right)\\
				&+\frac{5}{8}\eta_{x}\ell^{2}\frac{1}{(\mu+\ell)^2}\sum_{j=0}^{t-1} \gamma^{t-1-j} \mathbb{E}\left[\frac{1}{M}\left(1+\left\|\nabla_{y}\psi^{j}(x^{j},y^{j})\right\|_{F}^{2}\right) \sigma^{2}\right]\\
				&+\frac{5}{4}\eta_{x}\ell^{2}L_1^2\left(\frac{1}{(\mu+\ell)^2}+\frac{1}{\mu\ell}\right)\left(\sum_{j=0}^{t-1} \gamma^{t-1-j}  \mathbb{E}\left\|\nabla_{y}\psi^{j}(x^{j},y^{j})-\nabla_{y}\psi(x^{j},y^{j})\right\|_{F}^{2}\right)\\
				&+\left(\eta_{x}+2\kappa_{y}\ell\eta_{x}^{2} \right)L_{1}^{2}\mathbb{E}\left\|\nabla_{x}\psi^{t}(x^{t},y^{t})-\nabla_{x}\psi(x^{t},y^{t}) \right\|_F^{2}.
			\end{aligned}
		\end{equation*}

\noindent By the fact that 
	\begin{align*}
		\|\nabla_{x}\psi^{t}(x^{t},y^{t})\|_F^{2} \leq& 2\|\nabla_{x}\psi(x^{t},y^{t})\|_F^{2}+2\|\nabla_{x}\psi^{t}(x^{t},y^{t})-\nabla_{x}\psi(x^{t},y^{t})\|_F^{2},\\
		\|\nabla_{y}\psi^{t}(x^{t},y^{t})\|_F^{2} \leq& 2\|\nabla_{y}\psi(x^{t},y^{t})\|_F^{2}+2\|\nabla_{y}\psi^{t}(x^{t},y^{t})-\nabla_{y}\psi(x^{t},y^{t})\|_F^{2},
	\end{align*}
	the above inequality can be reformulated as
	\begin{small}
		\begin{equation*}
			\begin{aligned}
				\mathbb{E}\left[\Phi(x^{t+1})\right]&\leq \mathbb{E}\left[\Phi(x^{t})\right]-\frac{1}{4}\eta_{x}\mathbb{E}\left[\left\|\nabla\Phi(x^{t}) \right\|^{2} \right]
				+\frac{5}{4}\eta_{x}\ell^{2}\gamma^{t} D^{2}\\
				&+20(\kappa_{y}+1) \kappa_{y}^{2} \eta_{x}^{3}\ell^{2}\left(\sum_{j=0}^{t-1} \gamma^{t-1-j} \mathbb{E}\left\|\nabla \Phi\left(x^{j}\right)\right\|^{2}\right)+\frac{\kappa_{y}\ell\eta_{x}^{2}\sigma^{2}}{M}\mathbb{E}\left[1+\left(2\|\nabla_{x}\psi(x^{t},y^{t})\|_F^{2}\right) \right]\\
				&+\left(5 (\kappa_{y}+1)\kappa_{y}^{2} \eta_{x}^{3}\ell^{2}+\frac{5}{8}\eta_{x}\ell^{2}\frac{1}{(\mu+\ell)^2}\right)\frac{\sigma^{2}}{M} \sum_{j=0}^{t-1} \gamma^{t-1-j} \\
				&+10 (\kappa_{y}+1)\kappa_{y}^{2} \eta_{x}^{3}\ell^{2}\frac{\sigma^{2}}{M} \sum_{j=0}^{t-1} \gamma^{t-1-j} \mathbb{E}\left[\left(\|\nabla_{x}\psi(x^{j},y^{j})\|_F^{2}\right) \right]\\
				&+\frac{5}{4}\frac{\eta_{x}\ell^{2}}{(\mu+\ell)^2}\frac{\sigma^{2}}{M}\sum_{j=0}^{t-1} \gamma^{t-1-j} \mathbb{E}\left[\left(\|\nabla_{y}\psi(x^{j},y^{j})\|_F^{2}\right) \right]\\
				&+10 (\kappa_{y}+1)\kappa_{y}^{2} \eta_{x}^{3}\ell^{2}\left(L_1^2+\frac{\sigma^{2}}{M}\right) \sum_{j=0}^{t-1} \gamma^{t-1-j} \mathbb{E}\left[\left(\|\nabla_{x}\psi^{t+1}(x^{j},y^{j})-\nabla_{x}\psi(x^{j},y^{j})\|_F^{2}\right) \right]\\
				&+\frac{5}{4}\eta_{x}\left(\kappa_{y} L_1^2+\frac{\ell^{2}}{(\mu+\ell)^2}\left(\frac{\sigma^{2}}{M}+L_1^2\right)\right)\left(\sum_{j=0}^{t-1} \gamma^{t-1-j}  \mathbb{E}\left\|\nabla_{y}\psi^{j}(x^{j},y^{j})-\nabla_{y}\psi(x^{j},y^{j})\right\|_{F}^{2}\right)\\
				&+\left(\left(\eta_{x}+2\kappa_{y}\ell\eta_{x}^{2} \right)L_{1}^{2}+\frac{2\kappa_{y}\ell\eta_{x}^{2}\sigma^{2}}{M}\right)\mathbb{E}\left\|\nabla_{x}\psi^{t}(x^{t},y^{t})-\nabla_{x}\psi(x^{t},y^{t}) \right\|_F^{2}.
			\end{aligned}
		\end{equation*}
	\end{small}
	
	\noindent Rearranging the terms and
	summing up the above inequality over $t = 0, 1, . . . , T-1$, we arrive at
	\begin{small}
		\begin{equation}\label{eq:f_aftersum}
			\begin{aligned}
				\mathbb{E}\left[\Phi\left(x^{T}\right)\right]
				&\leq \mathbb{E}\left[\Phi\left(x^{0}\right)\right]-\frac{1}{4} \eta_{x} \sum_{t=0}^{T-1} \mathbb{E}\left[\left\|\nabla \Phi\left(x^{t}\right)\right\|^{2}\right]
				+\frac{5\eta_{x} \ell^{2} D^{2}}{4}  \sum_{t=0}^{T-1} \gamma^{t}\\
				&+20(\kappa_{y}+1) \kappa_{y}^{2} \eta_{x}^{3}\ell^{2}\left(\sum_{t=0}^{T-1}\sum_{j=0}^{t-1} \gamma^{t-1-j} \mathbb{E}\left\|\nabla \Phi\left(x^{j}\right)\right\|^{2}\right)\\
				&+\left(5 (\kappa_{y}+1)\kappa_{y}^{2} \eta_{x}^{3}\ell^{2}+\frac{5}{8}\eta_{x}\ell^{2}\frac{1}{(\mu+\ell)^2}\right)\frac{\sigma^{2}}{M} \sum_{t=0}^{T-1}\sum_{j=0}^{t-1} \gamma^{t-1-j} \\
				&+10 (\kappa_{y}+1)\kappa_{y}^{2} \eta_{x}^{3}\ell^{2}\frac{\sigma^{2}}{M} \sum_{t=0}^{T-1}\sum_{j=0}^{t-1} \gamma^{t-1-j} \mathbb{E}\left[\left(\|\nabla_{x}\psi(x^{j},y^{j})\|_F^{2}\right) \right]\\
				&+\frac{5}{4}\frac{\eta_{x}\ell^{2}}{(\mu+\ell)^2}\frac{\sigma^{2}}{M}\sum_{t=0}^{T-1}\sum_{j=0}^{t-1} \gamma^{t-1-j} \mathbb{E}\left[\left(\|\nabla_{y}\psi(x^{j},y^{j})\|_F^{2}\right) \right]+\frac{2\kappa_{y}\ell\eta_{x}^{2}\sigma^{2}}{M}\sum_{t=0}^{T-1}\mathbb{E}\left[\left(\|\nabla_{x}\psi(x^{t},y^{t})\|_F^{2}\right) \right]\\
				&+10 (\kappa_{y}+1)\kappa_{y}^{2} \eta_{x}^{3}\ell^{2}\left(L_1^2+\frac{\sigma^{2}}{M}\right) \sum_{t=0}^{T-1}\sum_{j=0}^{t-1} \gamma^{t-1-j} \mathbb{E}\left[\left(\|\nabla_{x}\psi^{t}(x^{j},y^{j})-\nabla_{x}\psi(x^{j},y^{j})\|_F^{2}\right) \right]\\
				&+\frac{5}{4}\eta_{x}\left(\kappa_{y} L_1^2+\frac{\ell^{2}}{(\mu+\ell)^2}\left(\frac{\sigma^{2}}{M}+L_1^2\right)\right)\left(\sum_{t=0}^{T-1}\sum_{j=0}^{t-1} \gamma^{t-1-j}  \mathbb{E}\left\|\nabla_{y}\psi^{j}(x^{j},y^{j})-\nabla_{y}\psi(x^{j},y^{j})\right\|_{F}^{2}\right)\\
				&+\left(\left(\eta_{x}+2\kappa_{y}\ell\eta_{x}^{2} \right)L_{1}^{2}+\frac{2\kappa_{y}\ell\eta_{x}^{2}\sigma^{2}}{M}\right)\sum_{t=0}^{T-1}\mathbb{E}\left\|\nabla_{x}\psi^{t}(x^{t},y^{t})-\nabla_{x}\psi(x^{t},y^{t}) \right\|_F^{2}+\frac{\kappa_{y}\ell\eta_{x}^{2}\sigma^{2}}{M}T.
			\end{aligned}
		\end{equation}
	\end{small}Given $\eta_{x}=1 / 40(\kappa_{y}+1)^{2} \ell$, we may bound the terms on the right-hand side of the above inequality by
	\begin{small}
	\begin{equation*}
		\begin{array}{l}
			\gamma \leq 1-\frac{1}{10 \kappa_{y}}, \sum_{t=0}^{T-1}\gamma^{t}\leq 10\kappa_{y},\\ (\kappa_{y}+1)\kappa_{y}^{2}\eta_{x}^{3}\ell^{2}\leq\frac{\eta_{x}}{1600\kappa_{y}}, \kappa_{y}\ell\eta_{x}^{2}\leq\frac{\eta_{x}}{40 \kappa_{y}},(1+\kappa_{y})\kappa_{y}^{2}\eta_{x}^{3}\ell^{2}\leq\frac{\kappa_{y}}{1600}\eta_{x},\\
			\sum_{t=0}^{T-1} \sum_{j=0}^{t-1} \gamma^{t-1-j} \leq 10 \kappa_{y} T,\\
			\sum_{t=0}^{T-1} \sum_{j=0}^{t-1} \gamma^{t-1-j}\mathbb{E}\left\|\nabla \Phi\left(x^{j}\right)\right\|^{2} \leq 10 \kappa_{y}\left(\sum_{t=0}^{T-1}\mathbb{E}\left\|\nabla \Phi\left(x^{t}\right)\right\|^{2}\right),\\
			\sum_{t=0}^{T-1} \sum_{j=0}^{t-1} \gamma^{t-1-j} \mathbb{E}\left\|\nabla_{x}\psi(x^{j},y^{j})\right\|_{F}^{2} \leq 10 \kappa_{y}
			\sum_{t=0}^{T-1} \mathbb{E}\left\|\nabla_{x}\psi(x^{t},y^{t})\right\|_{F}^{2}\leq10\kappa_{y} L_0^2 T,\\
			\sum_{t=0}^{T-1} \sum_{j=0}^{t-1} \gamma^{t-1-j} \mathbb{E}\left\|\nabla_{y}\psi(x^{j},y^{j})\right\|_{F}^{2} \leq 10 \kappa_{y}
			\sum_{t=0}^{T-1} \mathbb{E}\left\|\nabla_{y}\psi(x^{t},y^{t})\right\|_{F}^{2}\leq 10\kappa_{y} L_0^2 T,\\
			\sum_{t=0}^{T-1} \sum_{j=0}^{t-1} \gamma^{t-1-j} \mathbb{E}\left\|\nabla_{x}\psi^{j}(x^{j},y^{j})-\nabla_{x}\psi(x^{j},y^{j})\right\|_{F}^{2} \leq 10 \kappa_{y} \sum_{t=0}^{T-1} \mathbb{E}\left\|\nabla_{x}\psi^{t}(x^{t},y^{t})-\nabla_{x}\psi(x^{t},y^{t})\right\|_{F}^{2},
			\\
			\sum_{t=0}^{T-1} \sum_{j=0}^{t-1} \gamma^{t-1-j} \mathbb{E}\left\|\nabla_{y}\psi^{j}(x^{j},y^{j})-\nabla_{y}\psi(x^{j},y^{j})\right\|_{F}^{2} \leq 10 \kappa_{y} \sum_{t=0}^{T-1} \mathbb{E}\left\|\nabla_{y}\psi^{t}(x^{t},y^{t})-\nabla_{y}\psi(x^{t},y^{t})\right\|_{F}^{2}.
		\end{array}
	\end{equation*}
	\end{small}
	Plugging the above inequalities into (\ref{eq:f_aftersum}) and rearranging the inequality, we arrive at the desired inequality  \eqref{eq:theorem1_concluding}. 
	The proof is complete.
\end{proof}

Clearly, the second term on the right-hand side of inequality \eqref{eq:theorem1_concluding} is a constant dependent on the batchsize $M$. On the other hand, under Assumptions \ref{ass-primal} and \ref{ass:stronglyconcave}, both $\Phi(x^{0})-\mathbb{E}[\Phi(x^{T})]$ and $D$ are bounded, which means the first term on the right-hand side of inequality \eqref{eq:theorem1_concluding} converges to zero at a rate of $\mathcal{O}(1/T)$.  
Moreover, when the distribution map follows a location-scale model, i.e., $\mathcal{D}(x,y) = Ax + By + \xi$ with $\xi \sim \mathbb{P}$, there exists an online least squares oracle outlined in~\cite{narang2023multiplayer} that guarantees  $$\mathbb{E}\left\|\nabla\psi^{t}(x^{t},y^{t})-\nabla\psi(x^{t},y^{t})\right\|_{F}^{2}\leq \frac{c}{t+d}, \forall t\in\mathbb{N}$$ holds for some $c, d>0$, which implies the following bound holds
\begin{equation*}
	\sum_{t=0}^{T-1}\mathbb{E}\|\nabla\psi^{t}(x^{t},y^{t})-\nabla\psi(x^{t},y^{t})\|_{F}^{2} \leq c\ln(T+d),
\end{equation*}
that is, the last two terms in~\eqref{eq:theorem1_concluding} vanish at a rate of $\mathcal{O}(\ln(T)/T)$.
Then we can conclude that Algorithm \ref{algorithm:ASGDA} finds an $\epsilon$-stationary point within $\mathcal{O}(\epsilon^{-(4+\delta)})$ stochastic gradient evaluations for $\forall \delta > 0$. 
More recently, Li et al. \cite{li2022tiada} study the nonconvex-strongly concave minimax problems and propose an stochastic gradient descent ascent algorithm by employing an adaptive stepsize strategy. When the objective function is  second-order Lipschitz continuous for $y$, they show that the proposed algorithm achieves $\mathcal{O}(\epsilon^{-(4+\delta)})$ stochastic gradient complexity with $M = \mathcal{O}(1)$. 
Under the condition that the objective function $\mathcal{L}(\cdot)$ is Lipschitz continuous for $(x,y)$ and second-order Lipschitz continuous for $y$, we may show that ASGDA with the adaptive stepsize strategy achieves $\mathcal{O}(\epsilon^{-(4+\delta)})$ stochastic gradient complexity with $M=\mathcal{O}(1)$ by a similar analysis.

\subsection{Nonconvex-concave SMDD}\label{subsection3_2}
	In this subsection, we assume the following condition.
	\begin{ass}{\em\cite{lin2020gradient}}\label{ass:concave}The objective function and constraint set satisfy
		
		(a) For $\forall x\in \mathbb{R}^n$, $\mathcal{L}(x,\cdot)$ is concave.
		
		(b) The constraint set $\mathcal{Y}$ is a compact convex set with diameter $D > 0$.
	\end{ass}
	Under Assumption \ref{ass:concave}, the primal function $\Phi(\cdot)$ may be non-differentiable, which motivates us to characterize the solution of SMDD \eqref{eq:f_minimax_D} by the Moreau envelope function
\begin{equation*}
	\Phi_{1/2\ell}(\cdot)=\min _{w\in \mathbb{R}^{n}} \Phi(w)+\ell\|w-\cdot\|^{2}.
\end{equation*}

\begin{defn}\label{def_concave_solution}
A point $\hat{x}$ is an $\epsilon$-stationary point of SMDD \eqref{eq:f_minimax_D} if $\left\|\nabla \Phi_{1 / 2 \ell}(\hat{x})\right\|\leq\epsilon$.
\end{defn}
Similarly, we first present two key lemmas.
The first lemma characterizes the evolution of the Moreau envelope of the primal function $\Phi(\cdot)$ in expectation over the iteration of $x$.
\begin{lem}\label{lemma_concave_1}
	Let $$\Delta_t=\mathbb{E}[\Phi(x^t)-\mathcal{L}(x^t,y^t)].$$
	Suppose that Assumptions \ref{ass-primal}--\ref{ass-objective} and 
	\ref{ass:concave} hold.
	Then
	\begin{equation}\label{eq:concave_descentPhi}
		\begin{aligned}
			\mathbb{E}\left[\Phi_{1 / 2 \ell}\left(x^{t}\right)\right]  \leq&\mathbb{E}\left[\Phi_{1 / 2 \ell}\left(x^{t-1}\right)\right]+2 \eta_{x}\ell\Delta_{t-1}-\frac{\eta_{x}}{8}\mathbb{E}\left\|\nabla \Phi_{1 / 2 \ell}\left(x^{t-1}\right)\right\|^{2} \\
			& +2(\eta_x+\eta_x^2\ell)L_{1}^{2}\mathbb{E}\left\|\nabla_{x}\psi^{t-1}(x^{t-1},y^{t-1})-\nabla_{x}\psi(x^{t-1},y^{t-1})\right\|_{F}^{2}\\
			&+2 \eta_{x}^{2}\ell L^{2}+\eta_{x}^{2}\ell\mathbb{E}\left[1+\left\|\nabla_{x}\psi^{t-1}(x^{t-1},y^{t-1})\right\|_{F}^{2}\right] \frac{\sigma^{2}}{M}.
		\end{aligned}
	\end{equation}
\end{lem}
\begin{proof}
	The proof of Lemma \ref{lemma_concave_1} is provided in Appendix \ref{Appendix_for_Section3_2_Lem1}.
\end{proof}

The second Lemma provides an upper bound for the error $\Delta_{t}$.

\begin{lem}\label{lemma_concave_Delta}
	Let $\Delta_{t}$ be defined in Lemma \ref{lemma_concave_1}. Suppose that Assumptions \ref{ass-primal}--\ref{ass-objective} and 
	\ref{ass:concave} hold.
	Then for any integer $0< B\leq T+1$,
		\begin{small}
		\begin{equation}\label{eq:concave_sumDelta_4}
			\begin{aligned}
				\frac{1}{T+1}\sum_{t=0}^{T} \Delta_{t}\leq& \eta_{x} L L_{1}\left(1+L_{0}+L_{1}\right) (B+1) +\frac{\hat{\Delta}_{0}+\eta_{x} L L_{1}\left(1+L_{0}+L_{1}\right) B^{2}}{T+1}\\
				&+\frac{D^{2}}{2 \eta_{y}B}+\frac{D^{2}}{2 \eta_{y}(T+1)}+\frac{\eta_{x}}{2}L B
				\frac{1}{T+1}\sum_{t=0}^{T}\mathbb{E}\left\|\nabla_{x} \psi^{t}\left(x^{t},y^{t}\right)-\nabla_{x} \psi\left(x^{t},y^{t}\right)\right\|_{F}^{2}
				\\
				&+L_{1}D\frac{1}{T+1}\sum_{t=0}^{T}\mathbb{E}\left\|\nabla_{y}\psi^{t}(x^{t},y^{t})-\nabla_{y}\psi(x^{t},y^{t}) \right\|_{F}\\
				&
				+\frac{\eta_{y}}{2}\frac{1}{T+1}\sum_{t=0}^{T}\mathbb{E}\left[1+\left\|\nabla_{y}\psi^{t}(x^{t},y^{t})\right\|_F^{2} \right]\frac{\sigma^{2}}{M}\\
				&+\frac{\eta_{y}}{2}L_{1}^{2}\frac{1}{T+1}\sum_{t=0}^{T}\mathbb{E}\left\|\nabla_{y}\psi^{t}(x^{t},y^{t})-\nabla_{y}\psi(x^{t},y^{t}) \right\|_{F}^{2},
			\end{aligned}
		\end{equation}
		\end{small}
	where $\hat{\Delta}_{0}=\Phi\left(x^{0}\right)-\mathcal{L}\left(x^{0}, y^{0}\right)$.
\end{lem}
\begin{proof}
	The proof of Lemma \ref{lemma_concave_Delta} is provided in Appendix \ref{Appendix_for_Section3_2_Lem2}.
\end{proof}

We are ready to study the convergence of ASGDA.

\begin{thm}\label{theorem_concave}
	Suppose that Assumptions \ref{ass-primal}--\ref{ass-objective} and 
	\ref{ass:concave} hold and the stepsizes $\eta_{x}=\frac{1}{(T+1)^{3/4}}$, $\eta_{y}=\frac{1}{(T+1)^{1/4}}$.
	Then
	\begin{small}
		\begin{equation}\label{eq_th2_concave}
			\begin{aligned}
				\frac{1}{T+1} \sum_{t=0}^{T} \mathbb{E}\left[\left\|\nabla \Phi_{1 / 2 \ell}\left(x^{t}\right)\right\|^{2}\right]\leq&\mathcal{O}\left(\frac{1}{T^{1/4}}\right)
				+\mathcal{O}\left(\frac{\sum_{t=0}^{T} \mathbb{E}\left\|\nabla_{y} \psi^{t}\left(x^{t},y^{t}\right)-\nabla_{y} \psi\left(x^{t},y^{t}\right)\right\|_{F}}{T+1}\right)\\
				&+\mathcal{O}\left( \frac{1}{(T+1)^{7/4}} \sum_{t=0}^{T} \mathbb{E} \left\| \nabla_x \psi^t \left( x^t, y^t \right) - \nabla_x \psi \left( x^t, y^t \right) \right\|_F^2 \right)\\
				&+\mathcal{O}\left(\frac{1}{(T+1)^{5/4}} \sum_{t=0}^{T} \mathbb{E} \left\| \nabla_y \psi^t (x^t, y^t) - \nabla_y \psi(x^t, y^t) \right\|_F^2\right).
			\end{aligned}
		\end{equation}
		\end{small}
\end{thm}
\begin{proof}
	Summing up the inequality \eqref{eq:concave_descentPhi} over $t=1,2,\cdots,T+1$, 
	\begin{equation*}\label{eq:concave_th_sum}
		\begin{aligned}
			\mathbb{E}\left[\Phi_{1 / 2 \ell}\left(x^{T+1}\right)\right]  
			\leq &\Phi_{1 / 2 \ell}\left(x^{0}\right)+2 \eta_{x} \ell \sum_{t=0}^{T} \Delta_{t}-\frac{\eta_{x}}{8} \sum_{t=0}^{T} \mathbb{E}\left\|\nabla \Phi_{1 / 2 \ell}\left(x^{t}\right)\right\|^{2} \\
			&+2\left(\eta_{x}+\eta_{x}^{2} \ell\right) L_{1}^{2} \sum_{t=0}^{T}
			\mathbb{E}\left\|\nabla_{x}\psi^{t}(x^{t},y^{t})-\nabla_{x}\psi(x^{t},y^t)\right\|_{F}^{2}+2 \eta_{x}^{2}\ell L^{2}(T+1)\\
			&+\eta_{x}^{2} \ell \sum_{t=0}^{T} \mathbb{E}\left(1+\left\|\nabla_{x}\psi^{t}(x^{t},y^{t})\right\|_{F}^{2}\right)\frac{\sigma^{2}}{M}.
		\end{aligned}
	\end{equation*}
	
	\noindent Then
	\begin{equation*}\label{eq:concave_th_sum_rerange}
		\begin{aligned}
			\frac{1}{T+1} \sum_{t=0}^{T} \mathbb{E}\left[\left\|\nabla \Phi_{1 / 2 \ell}\left(x^{t}\right)\right\|^{2}\right] \leq& \frac{8}{\eta_{x}(T+1)}\left(\Phi_{1 / 2 \ell}\left(x^{0}\right)-\mathbb{E}\left[\Phi_{1 / 2 \ell}\left(x^{T+1}\right)\right]\right)\\
			+&16 \ell \frac{1}{T+1} \sum_{t=0}^{T} \Delta_{t}+16 \eta_{x} \ell L^2 \\
			+&\frac{16\left(1+\eta_{x} \ell\right) L_{1}^{2}}{T+1} \sum_{t=0}^{T} \mathbb{E}\left\|\nabla_{x} \psi^{t}\left(x^{t},y^{t}\right)-\nabla_{x} \psi\left(x^{t},y^{t}\right)\right\|_{F}^{2} \\
			+& \frac{8\eta_{x}\ell}{T+1}  \sum_{t=0}^{T} \mathbb{E}\left(1+\left\|\nabla_{x} \psi^{t}\left(x^{t},y^{t}\right)\right\|_{F}^{2}\right)\frac{\sigma^{2}}{M}.
		\end{aligned}
	\end{equation*}
	Denote $\Delta_{\Phi}=\Phi_{1 / 2 \ell}\left(x^{0}\right)-\min _{x} \Phi_{1 / 2 \ell}(x)$. 
	Plugging \eqref{eq:concave_sumDelta_4} into the above inequality, by the boundedness of $\left\|\nabla\psi(x^t,y^t)\right\|_{F}$ and the fact that 
	\begin{equation*}
		\begin{aligned}
			\|\nabla_{x}\psi^{t}(x^{t},y^{t})\|_F^{2} \leq 2\|\nabla_{x}\psi(x^{t},y^{t})\|_F^{2}+2\|\nabla_{x}\psi^{t}(x^{t},y^{t})-\nabla_{x}\psi(x^{t},y^t)\|_F^{2},\\
			\|\nabla_{y}\psi^{t}(x^{t},y^{t})\|_F^{2} \leq 2\|\nabla_{y}\psi(x^{t},y^{t})\|_F^{2}+2\|\nabla_{y}\psi^{t}(x^{t},y^{t})-\nabla_{y}\psi(x^{t},y^t)\|_F^{2},
		\end{aligned}
	\end{equation*}
	we arrive at the following inequality
	\begin{equation*}
		\begin{aligned}
			&\frac{1}{T+1} \sum_{t=0}^{T} \mathbb{E}\left[\left\|\nabla \Phi_{1 / 2 \ell}\left(x^{t}\right)\right\|^{2}\right]\\
			\leq& \frac{8\Delta_{\Phi}}{\eta_{x}(T+1)} + \frac{8\ell D^2}{\eta_y (T+1)}+8 \eta_{x} \ell\left(2 L^{2}+\frac{\sigma^{2}}{M}\left(1+2 L_{0}^{2}\right)\right)\\
			&+ 16\ell\left( \eta_x L L_1 (1 + L_0 + L_1) (B + 1)+\frac{D^{2}}{ 2B \eta_{y}}+\eta_y(1+2L_{0}^{2})\frac{\sigma^2}{2M}\right) \\
			&+ \frac{16\ell (\hat{\Delta}_0 + \eta_x L L_1 (1 + L_0 + L_1) B^2)}{T+1} +\frac{16\ell L_1 D}{T+1} \sum_{t=0}^{T} \left\| \nabla_y \psi^t (x^t, y^t) - \nabla_y \psi(x^t, y^t) \right\|_{F} \\
			& + \frac{\eta_x\ell\left(16(L_1^2+\frac{\sigma^2}{M})+8LB\right)+16L_1^2}{T+1} \sum_{t=0}^{T} \mathbb{E} \left\| \nabla_x \psi^t \left( x^t, y^t \right) - \nabla_x \psi \left( x^t, y^t \right) \right\|_F^2 \\		
			&+ \frac{8\eta_y\ell (L_1^2+\frac{2\sigma^2}{M})}{T+1} \sum_{t=0}^{T} \mathbb{E} \left\| \nabla_y \psi^t (x^t, y^t) - \nabla_y \psi(x^t, y^t) \right\|_F^2.	
		\end{aligned}
	\end{equation*}
	
	\noindent Subsequently, setting $B=\left\lfloor\frac{D}{2} \sqrt{\frac{1}{\eta_{x} \eta_{y}LL_{1}(1+L_{0}+L_{1})}}\right\rfloor+1$, $\eta_{x}=\frac{1}{(T+1)^{3/4}}$ and $ \eta_{y}=\frac{1}{(T+1)^{1/4}}$ in the above inequality, we have

	\begin{equation*}
		\begin{aligned}
			&\frac{1}{T+1} \sum_{t=0}^{T} \mathbb{E}\left[\left\|\nabla \Phi_{1 / 2 \ell}\left(x^{t}\right)\right\|^{2}\right]\\
			\leq& \frac{8 \Delta_{\Phi}}{(T+1)^{1/4}}+\frac{8\ell D^{2}}{(T+1)^{3/4}}+24 \ell D \sqrt{\frac{LL_{1}\left(1+L_{0}+L_{1}\right)}{(T+1)^{1/2}}}
			\\
			&+\frac{8 \ell}{(T+1)^{3/4}}\left(2 L^{2}+\left(1+2 L_{0}^{2}\right)\frac{\sigma^{2}}{M}+4LL_{1}(1+L_0+L_1)\right)+\frac{8\ell (1+2L_{0}^{2})}{(T+1)^{1/4}}\frac{\sigma^2}{M}\\
			&+\frac{16\ell \hat{\Delta}_0 }{T+1}+ \frac{16\ell L L_1 (1 + L_0 + L_1) B^2}{(T+1)^{7/4}} +\frac{16\ell L_1 D}{T+1} \sum_{t=0}^{T} \left\| \nabla_y \psi^t (x^t, y^t) - \nabla_y \psi(x^t, y^t) \right\|_{F} \\
			& + \frac{\ell\left(16(L_1^2+\frac{\sigma^2}{M})+8LB\right)+16L_1^2}{(T+1)^{7/4}} \sum_{t=0}^{T} \mathbb{E} \left\| \nabla_x \psi^t \left( x^t, y^t \right) - \nabla_x \psi \left( x^t, y^t \right) \right\|_F^2 \\		
			&+ \frac{8\ell (L_1^2+\frac{2\sigma^2}{M})}{(T+1)^{5/4}} \sum_{t=0}^{T} \mathbb{E} \left\| \nabla_y \psi^t (x^t, y^t) - \nabla_y \psi(x^t, y^t) \right\|_F^2,	
		\end{aligned}
	\end{equation*}
	\noindent which implies inequality \eqref{eq_th2_concave}. The proof is complete.
\end{proof}

The last three terms on the right-hand side of inequality \eqref{eq_th2_concave} are related to the estimation error of the distribution map $\mathcal{D}(\cdot)$. Similarly, when the distribution map follows a location-scale model,  $\mathbb{E}\left\|\nabla\psi^{t}(x^{t},y^{t})-\nabla\psi(x^{t},y^{t})\right\|_{F}^{2}\leq \frac{c}{t+d}, \forall t\in\mathbb{N}$ holds for some $c, d>0$, and they converge to zero at the rate of $\mathcal{O}\left(\frac{1}{T^{1/2}}\right)$, $\mathcal{O}\left(\frac{\ln(T)}{T^{7/4}}\right)$ and $\mathcal{O}\left(\frac{\ln(T)}{T^{5/4}}\right)$, respectively. 
Then the convergence rate of Algorithm \ref{algorithm:ASGDA} is dominated by the first term on the right-hand side of inequality \eqref{eq_th2_concave}, that is, Algorithm \ref{algorithm:ASGDA} outputs an $\epsilon$-stationary point within $T=\mathcal{O}(\epsilon^{-8})$ iterations. This complexity result matches the complexity bound of SGDA \cite{lin2020gradient} for nonconvex-concave stochastic minimax problems.

\section{Convergence analysis of AASGDA}\label{subsection4}
In this section, we study the convergence of AASGDA under the following P{\L} condition.

\begin{ass}[P{\L} condition in  $y$]\label{ass:PL_DD}
	For any fixed  $x, \max _{y \in \mathbb{R}^{m}} \mathcal{L}(x, y)$ has a nonempty solution set and a finite optimal value. There exists  $\mu>0$  such that $$\left\|\nabla_{y} \mathcal{L}(x, y)\right\|^{2} \geq 2 \mu\left[\max _{y^{\prime}} \mathcal{L}(x, y^{\prime})-\mathcal{L}(x, y)\right], \forall (x,y)\in\mathbb{R}^{n}\times\mathbb{R}^{m}.$$
\end{ass}
P{\L} condition 
is originally proposed to show global linear convergence of gradient descent for unconstrained minimization problems \cite{karimi2016linear,polyak1963gradient}.
Under Assumption \ref{ass:PL_DD}, the primal function $\Phi(\cdot)$ is differentiable, and we may adopt Definition \ref{def_solution_stronglyconcave} as the stability metric.
\begin{thm}\label{thm_1_PL}
    Suppose that Assumptions \ref{ass-primal}--\ref{ass-objective} and \ref{ass:PL_DD} hold and the stepsizes $\eta_{x}=\min\{\frac{1}{16\sqrt{T}},\frac{1}{176\ell\kappa_{y}^{2}}\}$, $\eta_{y}=\min\{\frac{11\kappa_{y}^{2}}{\sqrt{T}},\frac{1}{\ell}\}$, where $\kappa_{y}=\ell/\mu$ is the condition number and $\ell$ is the smoothness parameter of $\mathcal{L}(\cdot)$. Then
	\begin{equation}\label{eq:th3_converge}
		\begin{aligned}
			\frac{1}{T}\sum_{t=                                                              0}^{T-1}\mathbb{E}\|\nabla\Phi(x^t)\|^{2}\leq&\frac{512\widehat{\Delta}_{\Phi}+\left(15\widehat{L}+2\ell\right)\sigma^{2}\left(1+2L_{0}^{2}\right)+30976\kappa_{y}^{2}\ell\sigma^{2}\left(1+L_{0}^{2}\right)}{16\sqrt{T}}+\frac{352\ell\kappa_{y}^{2}\widehat{\Delta}_{\Phi}}{T}
			\\
			&+\left(33L_{1}^{2}+\frac{2\ell L_{1}^{2}+(4\ell+30\widehat{L})\sigma^{2}}{176\ell\kappa_{y}^{2}}\right)\frac{1}{T}\sum_{t=0}^{T-1}\mathbb{E}\|\nabla_{x}\psi(x^t,y^t)-\nabla_{x}\psi^{t}(x^t,y^t)\|_{F}^{2}\\
			&+176\kappa_{y}^{2}\left(L_{1}^{2}+2\sigma^{2}\right)\frac{1}{T}\sum_{t=0}^{T-1}\mathbb{E}\|\nabla_{y}\psi(x^{t+1},y^t)-\nabla_{y}\psi^{t}(x^{t+1},y^t)\|_{F}^{2},
		\end{aligned}
	\end{equation}
	where $\widehat{\Delta}_{\Phi}=14\left[\Phi\left(x^{0}\right)-\min_{x}\Phi(x)\right]+\left[\Phi\left(x^{0}\right)-\mathcal{L}\left(x^{0}, y^{0}\right)\right]$, $\hat{L}=\ell+\frac{\ell\kappa_{y}}{2}$.
\end{thm}
\begin{proof}
	Define the following Lyapunov function
	\begin{equation}\label{eq:NCPL_Lyapunov}
		 V(x^t,y^t)=\Phi(x^t)+\alpha\left[\Phi(x^t)-\mathcal{L}(x^t,y^t)\right]=(1+\alpha)\Phi(x^t)-\alpha\mathcal{L}(x^t,y^t), \forall t \geq 0,
	 \end{equation}
	where $\alpha>0$ is an arbitrary constant. 
	
	We first establish the descent property of the first term $\Phi(\cdot)$ on the right-hand side of \eqref{eq:NCPL_Lyapunov} over iterations.
	Since $\mathcal{L}(\cdot)$ is $\ell$-smooth by Assumption \ref{ass_stronglyconcave} and satisfies P{\L}-inequality in $y$, we have by \cite[Lemma A.3]{yang2022faster} that 
	$\Phi(\cdot)$ is $\hat{L}$-smooth with $\hat{L}=\ell+\frac{\ell\kappa_{y}}{2}$, which implies
	\begin{equation*}
		\Phi\left(x^{t+1}\right) \leq \Phi\left(x^{t}\right)+\left\langle\nabla \Phi\left(x^{t}\right), x^{t+1}-x^{t}\right\rangle+\frac{\hat{L}}{2}\left\|x^{t+1}-x^{t}\right\|^{2}.
	\end{equation*}
	
	\noindent By the iteration of $x$ in \eqref{equ:PL_alg_x} and taking expectation on both sides of the above inequality, conditioned on $(x^{t},y^{t}, \psi^{t}(\cdot))$,
	\begin{small}
	\begin{equation*}
		\begin{aligned}
			\mathbb{E}_{t}\left[\Phi\left(x^{t+1}\right)\right]\leq& \Phi\left(x^{t}\right)-\eta_{x}\left\langle\nabla \Phi\left(x^{t}\right),\mathbb{E}_{t}\left[G_{x}^{t}(x^t,y^t,z_x^t)\right]\right\rangle+\frac{\hat{L}\eta_{x}^{2}}{2}\mathbb{E}_{t}\left\|G_{x}^{t}(x^t,y^t,z_x^t)\right\|^{2}\\
			\leq&\Phi\left(x^{t}\right)-\eta_{x}\left\langle\nabla \Phi\left(x^{t}\right),\mathbb{E}_{t}\left[G_{x}^{t}(x^t,y^t,z_x^t)\right]\right\rangle\\
			&+\frac{\widehat{L}\eta_{x}^{2}}{2}\left\|\mathbb{E}_{t}G_x^t(x^t,y^t,z_x^t)\right\|^{2}+\frac{\widehat{L}\eta_{x}^{2}}{2}(1+\left\|\nabla_x \psi^t(x^t, y^t)\right\|_{F}^2)\sigma^{2}\\
			\leq&\Phi\left(x^{t}\right)-\eta_{x}\left\langle\nabla \Phi\left(x^{t}\right),\mathbb{E}_{t}\left[G_{x}^{t}(x^t,y^t,z_x^t)\right]\right\rangle\\
			&+\frac{\eta_{x}}{2}\left\|\mathbb{E}_{t}G_x^t(x^t,y^t,z_x^t)\right\|^{2}+\frac{\widehat{L}\eta_{x}^{2}}{2}(1+\left\|\nabla_x \psi^t(x^t, y^t)\right\|_{F}^2)\sigma^{2}\\
			\leq&\Phi(x^t)-\frac{\eta_x}{2}\|\nabla\Phi(x^t)\|^{2}+\eta_{x}\mathbb{E}_{t}\left\|\nabla_{x}\mathcal{L}(x^t,y^t)-\mathbb{E}_{t}G_{x}^{t}(x^t,y^t,z_x^t)\right\|^{2}\\
			&+\eta_{x}\left\|\nabla_{x}\mathcal{L}(x^t,y^t)-\nabla\Phi(x^t)\right\|^{2}+\frac{\widehat{L}}{2}\eta_{x}^{2}\left( 1+\left\|\nabla_x \psi^t(x^t, y^t)\right\|_{F}^2\right)\sigma^2\\
			\leq&\Phi(x^{t})-\frac{\eta_{x}}{2}\left\|\nabla\Phi(x^{t}) \right\|^{2}+\eta_{x}\left\|\nabla\Phi(x^{t})-\nabla_{x}\mathcal{L}(x^t,y^t)\right\|^{2}\\
			&+\eta_{x}L_{1}^{2}\|\nabla_{x}\psi(x^t,y^t)-\nabla_{x}\psi^{t}(x^t,y^t)\|_{F}^{2}+\frac{\widehat{L}}{2}\eta_{x}^{2}\left( 1+\left\|\nabla_x \psi^t(x^t, y^t)\right\|_{F}^2\right)\sigma^2,
		\end{aligned}
	\end{equation*}
	\end{small}where the second inequality follows from Lemma \ref{lem:expvar_PL} (b), the third inequality follows from the fact that $\eta_{x}\leq\frac{1}{\widehat{L}}$, the fourth inequality follows from $\|a+b\|^{2}\leq 2\|a\|^{2}+2\|b\|^{2}$ and the last inequality follows from Lemma \ref{lem:expvar_PL} (a).
	
	\noindent Taking expectation on both sides of the above inequality,
	\begin{equation}\label{equ:th_Phi}
		\begin{aligned}
			\mathbb{E}\left[\Phi\left(x^{t+1}\right)\right]\leq&\mathbb{E}\left[\Phi(x^t)\right]-\frac{\eta_{x}}{2}\mathbb{E}\left\|\nabla\Phi(x^{t}) \right\|^{2}+\eta_{x}\mathbb{E}\left\|\nabla\Phi(x^{t})-\nabla_{x}\mathcal{L}(x^t,y^t)\right\|^{2}\\
			+&\eta_{x}L_{1}^{2}\mathbb{E}\|\nabla_{x}\psi(x^t,y^t)-\nabla_{x}\psi^{t}(x^t,y^t)\|_{F}^{2}+\frac{\widehat{L}}{2}\eta_{x}^{2}\left( 1+\mathbb{E}\left\|\nabla_x \psi^t(x^t, y^t)\right\|_{F}^2\right)\sigma^2.
		\end{aligned}
	\end{equation}
	
	Next, we establish the one-step improvement of the second term $\mathcal{L}(\cdot)$ on the right-hand side of \eqref{eq:NCPL_Lyapunov}.
	By the smoothness of $\mathcal{L}(x,\cdot)$ and the iteration of $y$ in \eqref{equ:PL_alg_y}, 
	\begin{equation*}
		\begin{aligned}
			\mathcal{L}(x^{t+1},y^{t+1})\geq&\mathcal{L}(x^{t+1},y^{t})+\left\langle\nabla_{y}\mathcal{L}(x^{t+1},y^{t}),y^{t+1}-y^{t}\right\rangle-\frac{\ell}{2}\|y^{t+1}-y^{t}\|^{2}\\
			=&\mathcal{L}(x^{t+1},y^{t})+\eta_{y}\left\langle\nabla_{y}\mathcal{L}(x^{t+1},y^{t}),G_{y}^{t}(x^{t+1},y^{t},z_y^t)\right\rangle-\frac{\ell}{2}\eta_{y}^{2}\|G_{y}^{t}(x^{t+1},y^{t},z_y^t)\|^{2}.
		\end{aligned}
	\end{equation*}
	Taking expectation on both sides of the above inequality, conditioned on $(x^{t+1},y^{t},\psi^{t}(\cdot))$,
	\begin{small}
	\begin{equation*}
		\begin{aligned}
			\mathbb{E}_{\widehat{t}}\left[\mathcal{L}(x^{t+1},y^{t+1})\right]\geq& \mathcal{L}(x^{t+1},y^{t})+\eta_{y}\left\langle\nabla_{y}\mathcal{L}(x^{t+1},y^{t}),\mathbb{E}_{\widehat{t}}\left[G_{y}^{t}(x^{t+1},y^{t},z_y^t)\right]\right\rangle-\frac{\ell}{2}\eta_{y}^{2}\mathbb{E}_{\widehat{t}}\left[\|G_{y}^{t}(x^{t+1},y^{t},z_y^t)\|^{2}\right]\\
			\geq&\mathcal{L}(x^{t+1},y^{t})+\eta_{y}\left\langle\nabla_{y}\mathcal{L}(x^{t+1},y^{t}),\mathbb{E}_{\widehat{t}}\left[G_{y}^{t}(x^{t+1},y^{t},z_y^t)\right]\right\rangle\\
			&-\frac{\eta_{y}}{2}\left\|\mathbb{E}_{\widehat{t}}\left[G_y^t(x^{t+1},y^{t},z_y^{t})\right]\right\|^{2}-\frac{\ell}{2}\eta_{y}^{2}\left(1+\left\|\nabla_{y} \psi^{t}(x^{t+1}, y^t)\right\|_{F}^2\right)\sigma^{2}\\
			\geq&\mathcal{L}(x^{t+1},y^{t})+\frac{\eta_{y}}{2}\|\nabla_{y}\mathcal{L}(x^{t+1},y^{t})\|^{2}-\frac{\eta_{y}L_{1}^{2}}{2}\|\nabla_{y}\psi^t(x^{t+1},y^t)-\nabla_{y}\psi(x^{t+1},y^t)\|_F^{2}\\
			&-\frac{\ell}{2}\eta_{y}^{2}\left(1+\left\|\nabla_{y} \psi^{t}(x^{t+1}, y^t)\right\|_{F}^{2}\right)\sigma^{2},
		\end{aligned}
	\end{equation*}
	\end{small}where $\mathbb{E}_{\widehat{t}}\left[\,\cdot\,\right]=\mathbb{E}\left[\,\cdot\,| (x^{t+1},y^t,\psi^t(\cdot))\right]$, the second inequality follows from Lemma \ref{lem:expvar_PL} (d) and the fact that $\eta_{y}\leq\frac{1}{\ell}$ and the last inequality follows from Lemma \ref{lem:expvar_PL} (c).
	
	\noindent Taking expectation on both sides of the above inequality,
	\begin{equation}\label{equ:th_L_y}
		\begin{aligned}
			\mathbb{E}\left[\mathcal{L}(x^{t+1},y^{t+1})\right]\geq& \mathbb{E}\left[\mathcal{L}(x^{t+1},y^{t})\right]+\frac{\eta_{y}}{2}\mathbb{E}\|\nabla_{y}\mathcal{L}(x^{t+1},y^{t})\|^{2}\\
			&-\frac{\eta_{y}L_{1}^{2}}{2}\mathbb{E}\left\|\nabla_{y}\psi^{t}(x^{t+1},y^t)-\nabla_{y}\psi(x^{t+1},y^t)\right\|^{2}\\
			&-\frac{\ell}{2}\eta_{y}^{2}\left(1+\mathbb{E}\left\|\nabla_{y} \psi^{t}(x^{t+1}, y^t)\right\|_{F}^{2}\right)\sigma^{2}.
		\end{aligned}
	\end{equation}
	
	\noindent On the other hand, by the smoothness of $\mathcal{L}(\cdot,y)$ and the iteration of $x$ in \eqref{equ:PL_alg_x},
	\begin{equation*}
		\begin{aligned}
			\mathcal{L}(x^{t+1},y^{t})\geq& \mathcal{L}(x^{t},y^{t})+\left\langle \nabla_{x}\mathcal{L}(x^t,y^t),x^{t+1}-x^{t}\right\rangle -\frac{\ell}{2}\|x^{t+1}-x^{t}\|^{2}\\
			=&\mathcal{L}(x^t,y^t)-\eta_{x}\left\langle\nabla_{x}\mathcal{L}(x^t,y^t),G_x^t(x^t,y^t,z_x^t) \right\rangle-\frac{\ell}{2}\eta_{x}^{2}\|G_x^t(x^t,y^t,z_x^t)\|^{2}.
		\end{aligned}
	\end{equation*}
	Taking expectation on both sides of the above inequality, conditioned on $(x^{t},y^{t},\psi^{t}(\cdot))$,
	\begin{small}
	\begin{equation*}
		\begin{aligned}
			\mathbb{E}_{t}[\mathcal{L}(x^{t+1},y^{t})]\geq&\mathcal{L}(x^t,y^t)-\eta_{x}\left\langle\nabla_{x}\mathcal{L}(x^t,y^t),\mathbb{E}_{t}\left[G_x^t(x^t,y^t,z_x^t)\right] \right\rangle-\frac{\ell}{2}\eta_{x}^{2}\mathbb{E}_{t}\|G_x^t(x^t,y^t,z_x^t)\|^{2}\\
			\geq&\mathcal{L}(x^t,y^t)-\eta_{x}\left\|\nabla_{x}\mathcal{L}(x^t,y^t) \right\|^{2}-\eta_{x}\left\langle\nabla_{x}\mathcal{L}(x^t,y^t),\mathbb{E}_{t}\left[G_x^t(x^t,y^t,z_x^t)\right] -\nabla_{x}\mathcal{L}(x^t,y^t)\right\rangle\\
			&-\frac{\ell}{2}\eta_{x}^{2}\left\|\mathbb{E}_{t}\left[G_x^t(x^t,y^t,z_x^t)\right] \right\|^{2}-\frac{\ell}{2}\eta_{x}^{2}\left(1+\left\|\nabla_x \psi^t(x^t, y^t)\right\|_{F}^2\right)\sigma^2\\
			\geq&\mathcal{L}(x^t,y^t)-\frac{3}{2}\eta_{x}\left\|\nabla_{x}\mathcal{L}(x^t,y^t) \right\|^{2}-\frac{1}{2}\eta_{x}\left\|\mathbb{E}_{t}\left[G_x^t(x^t,y^t,z_x^t)\right] -\nabla_{x}\mathcal{L}(x^t,y^t)\right\|^{2}\\
			&-\frac{\eta_{x}}{2}\left\|\mathbb{E}_{t}\left[G_x^t(x^t,y^t,z_x^t)\right] \right\|^{2}-\frac{\ell}{2}\eta_{x}^{2}\left(1+\left\|\nabla_x \psi^t(x^t, y^t)\right\|_{F}^2\right)\sigma^2\\
			\geq&\mathcal{L}(x^t,y^t)-\frac{5}{2}\eta_{x}\left\|\nabla_{x}\mathcal{L}(x^t,y^t) \right\|^{2}-\frac{3}{2}\eta_{x}L_{1}^{2}\|\nabla_{x}\psi(x^t,y^t)-\nabla_{x}\psi^{t}(x^t,y^t)\|_F^{2}\\
			&-\frac{\ell}{2}\eta_{x}^{2}\left(1+\left\|\nabla_x \psi^t(x^t, y^t)\right\|_{F}^2\right)\sigma^2,
		\end{aligned}
	\end{equation*}
	\end{small}where the second inequality follows from Lemma \ref{lem:expvar_PL} (b), the third inequality follows from Young's inequality and the fact that $\eta_{x}\leq\frac{1}{\ell}$
	and the last inequality follows from $\|a+b\|^{2}\leq 2\|a\|^{2}+2\|b\|^{2}$ and Lemma \ref{lem:expvar_PL} (a).
	
	\noindent Taking expectation on both sides of the above inequality,
	\begin{equation}\label{equ:th_L_x}
		\begin{aligned}
			\mathbb{E}[\mathcal{L}(x^{t+1},y^{t})]\geq& \mathbb{E}\left[\mathcal{L}(x^t,y^t)\right]-\frac{5}{2}\eta_{x}\mathbb{E}\left\|\nabla_{x}\mathcal{L}(x^t,y^t) \right\|^{2}-\frac{3}{2}\eta_{x}L_{1}^{2}\mathbb{E}\|\nabla_{x}\psi(x^t,y^t)-\nabla_{x}\psi^{t}(x^t,y^t)\|^{2}\\
			&-\frac{\ell}{2}\eta_{x}^{2}\left(1+\mathbb{E}\left\|\nabla_x \psi^t(x^t, y^t)\right\|_{F}^2\right)\sigma^2.
		\end{aligned}
	\end{equation}
	Then, by summing up \eqref{equ:th_L_y} and \eqref{equ:th_L_x}, we derive a unified bound on the variation of $\mathcal{L}(\cdot)$ over iterations
	\begin{equation}\label{eq:th_sumxy}
		\begin{aligned}
			&\mathbb{E}\left[\mathcal{L}(x^{t+1},y^{t+1})\right]-\mathbb{E}\left[\mathcal{L}(x^{t},y^{t})\right]\\
			\geq&\frac{\eta_{y}}{2}\mathbb{E}\|\nabla_{y}\mathcal{L}(x^{t+1},y^{t})\|^{2}-\frac{5}{2}\eta_{x}\mathbb{E}\left\|\nabla_{x}\mathcal{L}(x^t,y^t) \right\|^{2}\\
			&-\frac{3}{2}\eta_{x}L_{1}^{2}\mathbb{E}\|\nabla_{x}\psi(x^t,y^t)-\nabla_{x}\psi^{t}(x^t,y^t)\|_F^{2}-\frac{\ell}{2}\eta_{x}^{2}\left(1+\mathbb{E}\left\|\nabla_x \psi^t(x^t, y^t)\right\|_{F}^2\right)\sigma^2\\
			&-\frac{\eta_{y}L_{1}^{2}}{2}\mathbb{E}\|\nabla_{y}\psi(x^{t+1},y^t)-\nabla_{y}\psi^{t}(x^{t+1},y^t)\|^{2}-\frac{\ell}{2}\eta_{y}^{2}\left(1+\mathbb{E}\left\|\nabla_{y} \psi^{t}(x^{t+1}, y^t)\right\|_{F}^{2}\right)\sigma^{2}.
		\end{aligned}
	\end{equation}
	
	Combining \eqref{equ:th_Phi} and \eqref{eq:th_sumxy}, the Lyapunov function $V(\cdot)$ satisfies
	\begin{small}
		\begin{equation*}
			\begin{aligned}
				&\mathbb{E}\left[V(x^t,y^t)\right]-\mathbb{E}\left[V(x^{t+1}, y^{t+1})\right]\\ \geq&(1+\alpha)\frac{\eta_{x}}{2}\mathbb{E}\left\|\nabla\Phi(x^{t}) \right\|^{2}-(1+\alpha)\eta_{x}\mathbb{E}\left\|\nabla\Phi(x^{t})-\nabla_{x}\mathcal{L}(x^t,y^t)\right\|^{2}\\
				&+\alpha\frac{\eta_{y}}{2}\mathbb{E}\|\nabla_{y}\mathcal{L}(x^{t+1},y^{t})\|^{2}-\frac{5}{2}\alpha\eta_{x}\mathbb{E}\left\|\nabla_{x}\mathcal{L}(x^t,y^t) \right\|^{2}\\
				&-\left((1+\frac{5}{2}\alpha)\eta_{x}L_{1}^{2}\right)\mathbb{E}\|\nabla_{x}\psi(x^t,y^t)-\nabla_{x}\psi^{t}(x^t,y^t)\|_{F}^{2}-\alpha\frac{\eta_{y}L_{1}^{2}}{2}\mathbb{E}\|\nabla_{y}\psi(x^{t+1},y^t)-\nabla_{y}\psi^{t}(x^{t+1},y^t)\|_F^{2}
				\\	
				&-\frac{\eta_{x}^{2}\sigma^2}{2}\left((1+\alpha)\widehat{L}+\alpha\ell\right)\left( 1+\mathbb{E}\left\|\nabla_x \psi^t(x^t, y^t)\right\|_{F}^2\right)-\frac{\eta_{y}^{2}\sigma^{2}}{2}\alpha\ell\left(1+\mathbb{E}\left\|\nabla_{y} \psi^{t}(x^{t+1}, y^t)\right\|_{F}^{2}\right).
			\end{aligned}
		\end{equation*}
	\end{small}Furthermore, by applying the inequality $\|a + b\|^2 \leq 2\|a\|^2 + 2\|b\|^2$ to the last two terms on the right-hand side of the above inequality, and using $\|a\|^2 \geq \frac{1}{2} \|b\|^2 - \|a - b\|^2$ for $\|\nabla_y \mathcal{L}(x^{t+1}, y^t)\|^2$, along with the $\ell$-smoothness of $\mathcal{L}(\cdot)$, we obtain
		\begin{small}
	\begin{equation*}
		\begin{aligned}
			&\mathbb{E}\left[V(x^t,y^t)\right]-\mathbb{E}\left[V(x^{t+1}, y^{t+1})\right]\\ \geq&(1+\alpha)\frac{\eta_{x}}{2}\mathbb{E}\left\|\nabla\Phi(x^{t}) \right\|^{2}-(1+\alpha)\eta_{x}\mathbb{E}\left\|\nabla\Phi(x^{t})-\nabla_{x}\mathcal{L}(x^t,y^t)\right\|^{2}\\
			&+\alpha\frac{\eta_{y}}{4}\mathbb{E}\|\nabla_{y}\mathcal{L}(x^{t},y^{t})\|^{2}-\frac{\alpha\eta_{y}}{2}\ell^{2}\underbrace{\mathbb{E}\left\|x^{t+1}-x^{t}\right\|^{2}}_{I_{1}}-\frac{5}{2}\alpha\eta_{x}\mathbb{E}\left\|\nabla_{x}\mathcal{L}(x^t,y^t) \right\|^{2}\\
			&-\left((1+\frac{5}{2}\alpha)\eta_{x}L_{1}^{2}+\eta_{x}^{2}\sigma^2\left((1+\alpha)\widehat{L}+\alpha\ell\right)\right)\mathbb{E}\|\nabla_{x}\psi(x^t,y^t)-\nabla_{x}\psi^{t}(x^t,y^t)\|_{F}^{2}\\
			&-\left(\alpha\frac{\eta_{y}L_{1}^{2}}{2}+\alpha\eta_{y}^{2}\sigma^{2}\ell\right)\mathbb{E}\|\nabla_{y}\psi(x^{t+1},y^t)-\nabla_{y}\psi^{t}(x^{t+1},y^t)\|_{F}^{2}
			\\	
			&-\frac{\eta_{x}^{2}\sigma^2}{2}\left((1+\alpha)\widehat{L}+\alpha\ell\right)\left( 1+2\mathbb{E}\left\|\nabla_{x} \psi(x^t, y^t)\right\|_{F}^{2}\right)
			-\frac{\eta_{y}^{2}\sigma^{2}}{2}\alpha\ell\left(1+2\mathbb{E}\left\|\nabla_{y} \psi(x^{t+1}, y^t)\right\|_{F}^{2}\right).
		\end{aligned}
	\end{equation*}
	\end{small}
	
	\noindent For $I_1$, by the iteration of $x$ in \eqref{equ:PL_alg_x},
	\begin{small}
	\begin{equation*}
		\begin{aligned}
			I_{1}=&\eta_{x}^{2}\mathbb{E}\left\|G_{x}^{t}(x^t,y^t,z^t)\right\|^{2}\\
			\leq&2\eta_{x}^{2}\|\nabla_{x}\mathcal{L}(x^t,y^t)\|^{2}+2\eta_{x}^{2}\left(L_{1}^{2}+\sigma^{2}\right)\mathbb{E}\|\nabla_{x}\psi(x^t,y^t)-\nabla_{x}\psi^{t}(x^t,y^t)\|_{F}^{2}+\eta_{x}^{2}\left(1+2\left\|\nabla_{x} \psi(x^t, y^t)\right\|_{F}^{2}\right)\sigma^2,
		\end{aligned}
	\end{equation*}
	\end{small}where the inequality follows from Lemma \ref{lem:expvar_PL} (a)-(b) and the fact $\|a+b\|^{2}\leq 2\|a\|^{2}+2\|b\|^{2}$.
	
	\noindent Then,
		\begin{small}
		\begin{equation}\label{eq:th1_LyapDes}
			\begin{aligned}
				&\mathbb{E}\left[V(x^t,y^t)\right]-\mathbb{E}\left[V(x^{t+1}, y^{t+1})\right]\\
				\geq&(1+\alpha)\frac{\eta_{x}}{2}\mathbb{E}\left\|\nabla\Phi(x^{t}) \right\|^{2}
				-(1+\alpha)\eta_{x}\underbrace{\mathbb{E}\left\|\nabla_{x}\mathcal{L}(x^t,y^t)-\nabla\Phi(x^{t})\right\|^{2}}_{I_2}\\
				&+\frac{\alpha\eta_{y}}{4}\mathbb{E}\left\|\nabla_{y}\mathcal{L}(x^{t},y^t)\right\|^{2}-\left(\alpha\eta_{y}\ell^{2}\eta_{x}^{2}+\frac{5}{2}\alpha\eta_{x}\right)\underbrace{\mathbb{E}\|\nabla_{x}\mathcal{L}(x^t,y^t)\|^{2}}_{I_3}\\
				&-\left((1+\frac{5}{2}\alpha)\eta_{x}L_{1}^{2}+\eta_{x}^{2}\sigma^2\left((1+\alpha)\widehat{L}+\alpha\ell\right)+\alpha\eta_{y}\ell^{2}\eta_{x}^{2}\left(L_{1}^{2}+\sigma^{2}\right)\right)\mathbb{E}\|\nabla_{x}\psi(x^t,y^t)-\nabla_{x}\psi^{t}(x^t,y^t)\|_{F}^{2}\\
				&-\left(\alpha\frac{\eta_{y}L_{1}^{2}}{2}+\alpha\eta_{y}^{2}\sigma^{2}\ell\right)\mathbb{E}\|\nabla_{y}\psi(x^{t+1},y^t)-\nabla_{y}\psi^{t}(x^{t+1},y^t)\|_F^{2}\\
				&-\frac{\eta_{x}^{2}\sigma^2}{2}\left((1+\alpha)\widehat{L}+\alpha\ell+\alpha\eta_{y}\ell^{2}\right)\left( 1+2\mathbb{E}\left\|\nabla_{x} \psi(x^t, y^t)\right\|_{F}^{2}\right)
				-\frac{\eta_{y}^{2}\sigma^{2}}{2}\alpha\ell\left(1+2\mathbb{E}\left\|\nabla_{y} \psi(x^{t+1}, y^t)\right\|_{F}^{2}\right).
			\end{aligned}
		\end{equation}
		\end{small}Since $\mathcal{L}(x,\cdot)$ is $\ell$-smooth and satisfies P{\L} condition, by fixing $y^{\star}(x^t)$ as the projection of $y^{t}$ onto the set $\arg\max_{y}\mathcal{L}(x^t,y)$, we have by \cite[Lemma A.2]{yang2022faster} that
	$$I_2\leq \ell^{2}\mathbb{E}\left\|y^{t}-y^{\star}(x^t)\right\|^{2}\leq\kappa_{y}^{2}\mathbb{E}\left\|\nabla_{y}\mathcal{L}(x^t,y^t)\right\|^{2},$$
	where $\kappa_{y}=\frac{\ell}{\mu}$.
	By the fact that $\|a+b\|^{2}\leq2\|a\|^{2}+2\|b\|^{2}$,
	\begin{equation*}
		\begin{aligned}
			I_{3}\leq 2\mathbb{E}\left\|\nabla_{x}\mathcal{L}(x^t,y^t)-\nabla\Phi(x^t)\right\|^{2}+2\mathbb{E}\left\|\nabla\Phi(x^t)\right\|^{2}.
		\end{aligned}
	\end{equation*}
	Plugging the above inequalities into \eqref{eq:th1_LyapDes}, we have
		\begin{small}
		\begin{equation}\label{eq:th1_VV}
			\begin{aligned}
				&\mathbb{E}\left[V(x^t,y^t)\right]-\mathbb{E}\left[V(x^{t+1}, y^{t+1})\right]\\
				\geq&\left[(1+\alpha)\frac{\eta_{x}}{2}-2\left(\alpha\eta_{y}\ell^{2}\eta_{x}^{2}+\frac{5}{2}\alpha\eta_{x}\right)\right]\mathbb{E}\|\nabla\Phi(x^t)\|^{2}\\
				&+\left[\frac{\alpha\eta_{y}}{4}-(1+\alpha)\eta_{x}\kappa_{y}^{2}-2\left(\alpha\eta_{y}\ell^{2}\eta_{x}^{2}+\frac{5}{2}\alpha\eta_{x}\right)\kappa_{y}^{2}\right]\mathbb{E}\left\|\nabla_{y}\mathcal{L}(x^t,y^t)\right\|^{2}\\
				&-\left((1+\frac{5}{2}\alpha)\eta_{x}L_{1}^{2}+\eta_{x}^{2}\sigma^2\left((1+\alpha)\widehat{L}+\alpha\ell\right)+\alpha\eta_{y}\ell^{2}\eta_{x}^{2}\left(L_{1}^{2}+\sigma^{2}\right)\right)\mathbb{E}\|\nabla_{x}\psi(x^t,y^t)-\nabla_{x}\psi^{t}(x^t,y^t)\|_{F}^{2}\\
				&-\left(\alpha\frac{\eta_{y}L_{1}^{2}}{2}+\alpha\eta_{y}^{2}\sigma^{2}\ell\right)\mathbb{E}\|\nabla_{y}\psi(x^{t+1},y^t)-\nabla_{y}\psi^{t}(x^{t+1},y^t)\|_F^{2}\\
				&-\frac{\eta_{x}^{2}\sigma^2}{2}\left((1+\alpha)\widehat{L}+\alpha\ell+\alpha\eta_{y}\ell^{2}\right)\left( 1+2\left\|\nabla_{x} \psi(x^t, y^t)\right\|_{F}^{2}\right)
				\\	
				&-\frac{\eta_{y}^{2}\sigma^{2}}{2}\alpha\ell\left(1+2\left\|\nabla_{y} \psi(x^{t+1}, y^t)\right\|_{F}^{2}\right).
			\end{aligned}
		\end{equation}
		\end{small}Next, we analyze the coefficients of each term on the right-hand side of the above inequality. Setting $\alpha=\frac{1}{14}, \eta_{y}\leq\frac{1}{\ell}$ and $\eta_{x}\leq\frac{\eta_{y}}{176\kappa_{y}^{2}}$, the first term
	$$(1+\alpha)\frac{\eta_{x}}{2}-\frac{1}{7}\left(\eta_{y}\ell^{2}\eta_{x}^{2}+\frac{5}{2}\eta_{x}\right)\geq\frac{\eta_{x}}{28},$$
	the second term
	$$\frac{\alpha\eta_{y}}{4}-(1+\alpha)\eta_{x}\kappa_{y}^{2}-2\left(\alpha\eta_{y}\ell^{2}\eta_{x}^{2}+\frac{5}{2}\alpha\eta_{x}\right)\kappa_{y}^{2}\geq\frac{\eta_{y}}{112}\geq\frac{11}{7}\kappa_{y}^{2}\eta_{x},$$
	the third term 
	$$(1+\frac{5}{2}\alpha)\eta_{x}L_{1}^{2}+\eta_{x}^{2}\sigma^2\left((1+\alpha)\widehat{L}+\alpha\ell\right)+\alpha\eta_{y}\ell^{2}\eta_{x}^{2}\left(L_{1}^{2}+\sigma^{2}\right)
	\leq\frac{33}{28}L_{1}^{2}\eta_{x}+\frac{1}{14}\left(\ell L_{1}^{2}+(2\ell+15\widehat{L})\sigma^{2}\right)\eta_{x}^{2},$$
	the fourth term 
	$$\alpha\frac{\eta_{y}L_{1}^{2}}{2}+\alpha\eta_{y}^{2}\sigma^{2}\ell=\frac{44}{7}\kappa_{y}^{2}L_{1}^{2}\eta_{x}+\frac{15488}{7}\kappa_{y}^{4}\ell\sigma^{2}\eta_{x}^{2},$$
	the fifth term 
	$$\frac{\eta_{x}^{2}\sigma^2}{2}\left((1+\alpha)\widehat{L}+\alpha\ell+\alpha\eta_{y}\ell^{2}\right)\leq\left(\frac{15}{28}\widehat{L}+\frac{1}{14}\ell\right)\sigma^{2}\eta_{x}^{2},$$
	the last term
	$$\frac{\eta_{y}^{2}\sigma^{2}}{2}\alpha\ell=\frac{7744}{7}\kappa_{y}^{4}\ell\sigma^{2}\eta_{x}^{2}.$$
	\noindent Substituting the coefficient inequalities into \eqref{eq:th1_VV},
	\begin{equation*}
		\begin{aligned}
			&\mathbb{E}\left[V(x^t,y^t)\right]-\mathbb{E}\left[V(x^{t+1}, y^{t+1})\right]\\
			\geq&\frac{\eta_{x}}{28}\mathbb{E}\|\nabla\Phi(x^t)\|^{2}+\frac{11}{7}\kappa_{y}^{2}\eta_{x}\mathbb{E}\left\|\nabla_{y}\mathcal{L}(x^t,y^t)\right\|^{2}\\
			&-\left(\frac{33}{28}L_{1}^{2}\eta_{x}+\frac{1}{14}\left(\ell L_{1}^{2}+(2\ell+15\widehat{L})\sigma^{2}\right)\eta_{x}^{2}\right)\mathbb{E}\|\nabla_{x}\psi(x^t,y^t)-\nabla_{x}\psi^{t}(x^t,y^t)\|_{F}^{2}\\
			&-\left(\frac{44}{7}\kappa_{y}^{2}L_{1}^{2}\eta_{x}+\frac{15488}{7}\kappa_{y}^{4}\ell\sigma^{2}\eta_{x}^{2}\right)\mathbb{E}\|\nabla_{y}\psi(x^{t+1},y^t)-\nabla_{y}\psi^{t}(x^{t+1},y^t)\|_F^{2}\\
			&-\left(\frac{15}{28}\widehat{L}+\frac{1}{14}\ell\right)\sigma^{2}\eta_{x}^{2}(1+2\mathbb{E}\left\|\nabla_{x} \psi(x^t, y^t)\right\|_{F}^{2})-\frac{7744}{7}\kappa_{y}^{4}\ell\sigma^{2}\eta_{x}^{2}(1+2\mathbb{E}\left\|\nabla_{y} \psi(x^{t+1}, y^t)\right\|_{F}^{2}).
		\end{aligned}
	\end{equation*}
	Taking an average of the above inequality over $t=0\cdots T-1$,
	\begin{equation*}
		\begin{aligned}
			\frac{1}{T}\sum_{t=0}^{T-1}\mathbb{E}\|\nabla\Phi(x^t)\|^{2}&\leq\frac{28}{\eta_{x}T}\underbrace{\left(V(x^0,y^0)-V(x^T,y^T)\right)}_{I_4}\\
			&+\left(33L_{1}^{2}+\left(2\ell L_{1}^{2}+(4\ell+30\widehat{L})\sigma^{2}\right)\eta_{x}\right)\frac{1}{T}\sum_{t=0}^{T-1}\mathbb{E}\|\nabla_{x}\psi(x^t,y^t)-\nabla_{x}\psi^{t}(x^t,y^t)\|_{F}^{2}\\
			&+\left(176\kappa_{y}^{2}L_{1}^{2}+61952\kappa_{y}^{4}\ell\sigma^{2}\eta_{x}\right)\frac{1}{T}\sum_{t=0}^{T-1}\mathbb{E}\|\nabla_{y}\psi(x^{t+1},y^t)-\nabla_{y}\psi^{t}(x^{t+1},y^t)\|_F^{2}\\
			&+\left(15\widehat{L}+2\ell\right)\sigma^{2}\eta_{x}\left(1+2\mathbb{E}\left\|\nabla_{x} \psi(x^t, y^t)\right\|_{F}^{2}\right)\\
			&+30976\kappa_{y}^{4}\ell\sigma^{2}\eta_{x}\left(1+2\mathbb{E}\left\|\nabla_{y} \psi(x^{t+1}, y^t)\right\|_{F}^{2}\right).
		\end{aligned}
	\end{equation*}
	Since for $\forall x\in\mathbb{R}^n$ there exists $y$ such $\Phi(x)=\mathcal{L}(x,y)$, we have
	\begin{equation*}
		\begin{aligned}
			I_4\leq V(x^0,y^0)-\min _{x, y} V(x, y)
			=&\Phi\left(x^{0}\right)-\min_{x}\Phi(x)+\alpha\left[\Phi\left(x^{0}\right)-\mathcal{L}\left(x^{0}, y^{0}\right)\right].
		\end{aligned}
	\end{equation*}
	
	\noindent Then,
	\begin{equation*}
		\begin{aligned}
			\frac{1}{T}\sum_{t=0}^{T-1}\mathbb{E}\|\nabla\Phi(x^t)\|^{2}\leq&\frac{28}{\eta_{x}T}\left[\Phi(x^{0})-\min_{x}\Phi(x)\right]+\frac{2}{\eta_{x}T}\left[\Phi(x^{0})-\mathcal{L}(x^{0},y^{0})\right]\\
			+&\left(33L_{1}^{2}+\left(2\ell L_{1}^{2}+(4\ell+30\widehat{L})\sigma^{2}\right)\eta_{x}\right)\frac{1}{T}\sum_{t=0}^{T-1}\mathbb{E}\|\nabla_{x}\psi(x^t,y^t)-\nabla_{x}\psi^{t}(x^t,y^t)\|_{F}^{2}\\
			+&\left(176\kappa_{y}^{2}L_{1}^{2}+61952\kappa_{y}^{4}\ell\sigma^{2}\eta_{x}\right)\frac{1}{T}\sum_{t=0}^{T-1}\mathbb{E}\|\nabla_{y}\psi(x^{t+1},y^t)-\nabla_{y}\psi^{t}(x^{t+1},y^t)\|_F^{2}\\
			+&\left(15\widehat{L}+2\ell\right)\sigma^{2}\left(1+2\mathbb{E}\left\|\nabla_{x} \psi(x^t, y^t)\right\|_{F}^{2}\right)\eta_{x}\\
			+&30976\kappa_{y}^{4}\ell\sigma^{2}\left(1+2\mathbb{E}\left\|\nabla_{y} \psi(x^{t+1}, y^t)\right\|_{F}^{2}\right)\eta_{x}.
		\end{aligned}
	\end{equation*}
	Setting $\eta_{x}=\min\{\frac{1}{16\sqrt{T}},\frac{1}{176\ell\kappa_{y}^{2}}\}$,  $\eta_{y}=\min\{\frac{11\kappa_{y}^{2}}{\sqrt{T}},\frac{1}{\ell}\}$ in the above inequality, we have
	\begin{small}
	\begin{equation*}
		\begin{aligned}
			\frac{1}{T}\sum_{t=0}^{T-1}\mathbb{E}\|\nabla\Phi(x^t)\|^{2}\leq&\frac{32\widehat{\Delta}_{\Phi}}{\sqrt{T}}+\frac{352\ell\kappa_{y}^{2}\widehat{\Delta}_{\Phi}}{T}
			\\
			+&\left(33L_{1}^{2}+\frac{2\ell L_{1}^{2}+(4\ell+30\widehat{L})\sigma^{2}}{176\ell\kappa_{y}^{2}}\right)\frac{1}{T}\sum_{t=0}^{T-1}\mathbb{E}\|\nabla_{x}\psi(x^t,y^t)-\nabla_{x}\psi^{t}(x^t,y^t)\|_{F}^{2}\\
			+&176\kappa_{y}^{2}\left(L_{1}^{2}+2\sigma^{2}\right)\frac{1}{T}\sum_{t=0}^{T-1}\mathbb{E}\|\nabla_{y}\psi(x^{t+1},y^t)-\nabla_{y}\psi^{t}(x^{t+1},y^t)\|_{F}^{2}\\
			+&\frac{\left(15\widehat{L}+2\ell\right)\sigma^{2}\left(1+2\mathbb{E}\|\nabla_{x}\psi(x^{t},y^{t})\|_{F}^{2}\right)}{16\sqrt{T}}
			+\frac{30976\kappa_{y}^{2}\ell\sigma^{2}\left(1+2\mathbb{E}\|\nabla_{y}\psi(x^{t+1},y^{t})\|_{F}^{2}\right)}{16\sqrt{T}},
		\end{aligned}
	\end{equation*}
	\end{small}where $\widehat{\Delta}_{\Phi}=14\left[\Phi\left(x^{0}\right)-\min_{x}\Phi(x)\right]+\left[\Phi\left(x^{0}\right)-\mathcal{L}\left(x^{0}, y^{0}\right)\right]$. The desired inequality \eqref{eq:th3_converge} follows directly from $L_0$-Lipschitz continuity of $\psi(\cdot)$. The proof is complete.
\end{proof}
Given that $\widehat{\Delta}_{\Phi}$ is bounded under Assumption \ref{ass-primal}, the first two terms on the right-hand side of inequality \eqref{eq:th3_converge} converge to zero at rates of $\mathcal{O}(\frac{1}{\sqrt{T}})$ and $\mathcal{O}(\frac{1}{T})$, respectively. Moreover, the last two terms on the right-hand side of inequality \eqref{eq:th3_converge} depend on the estimation error of the distribution map, which converge to zero at the rate of $\mathcal{O}\left(\frac{\ln(T)}{T}\right)$ when the distribution map follows a location-scale model. Consequently, the convergence rate of Algorithm \ref{algorithm:AASGDA} is dominated by the first term on the right-hand side of inequality \eqref{eq:th3_converge}, that is, Algorithm \ref{algorithm:AASGDA} outputs an $\epsilon$-stationary point
within $T=\mathcal{O}(\kappa_{y}^{4}\epsilon^{-4})$ iterations.

\section{Experiments}\label{section5}
In this section, we verify the effectiveness of ASGDA and AASGDA. In Subsection \ref{Subsection5_1}, we test the effectiveness of ASGDA for nonconvex-strongly concave SMDD-I on a synthetic example and a real-world application. Similarly, in Subsection \ref{Subsection5_2}, we test the effectiveness of AASGDA for nonconvex-P{\L} SMDD-II. 
Both algorithms utilize the online least squares oracle described in \cite{narang2023multiplayer} for distribution map learning.

\vspace{-0.25cm}
\subsection{Experiments of ASGDA for SMDD-I}\label{Subsection5_1}
Consider the following stochastic minimax problem with  decision-dependent distribution~\cite{boct2023alternating}
\begin{equation}\label{experiment_1}
	\min_{x\in\mathbb{R}}\max_{y\in [-10,10]}\mathcal{L}(x,y):=\underset{z\sim\mathcal{D}(x,y)}{\mathbb{E}}[-xz+yz-\frac{1}{2}y^{2}],
\end{equation}
where $z=4x-y+\xi$, $\xi\sim\mathcal{N}(0,1)$. By some calculation, $\mathcal{L}(x,y)$ is $\ell$-smooth with $\ell=12$ and $\mathcal{L}(x,\cdot)$ is $\mu$-strongly concave with $\mu=3$. Then the condition number $\kappa_{y}=4$, and the corresponding primal function is a nonconvex differentiable function
\begin{equation*}
	\Phi(x)=\max_{y\in[-10,10]}\{-4x^{2}+5xy-\frac{3}{2}y^{2}\}=\left\{\begin{array}{cc}
		-4 x^{2}-50 x - 150, & x<-6, \\
		\frac{1}{6} x^{2}, & x \in\left[-6, 6\right], \\
		-4 x^{2}+50 x-150, & x>6.
	\end{array}\right.
\end{equation*}
Obviously, the unique stationary point of the primal function $\Phi(x)$ is $x^{\star}=0$ with $y^{\star}(x^{\star})=0$.

We compare ASGDA with the stochastic primal-dual (SPD) method \cite{wood2023stochastic}. In Algorithm \ref{algorithm:ASGDA}, the stepsizes $\eta_x=
\frac{1}{40(\kappa_{y}+1)^2\ell}=\frac{1}{12000}, \eta_{y}=\frac{1}{2(\mu+\ell)}=\frac{1}{30}$, the convergence accuracy $\epsilon=10^{-2}$ and the batchsize $M=200$.  In SPD, the constant stepsize $\eta=10^{-5}$, the dynamic stepsize  $\eta_{t}=\frac{1}{8\times10^{4}+t}$ and the batchsize $M=200$.
Moreover, the initial point is $(x^{0},y^{0})=(5, 5)$.

\begin{figure}[H]
	\centering
	\begin{minipage}[t]{0.37\textwidth}
		\centering
		\subfigure[]{
			\includegraphics[scale=0.42]{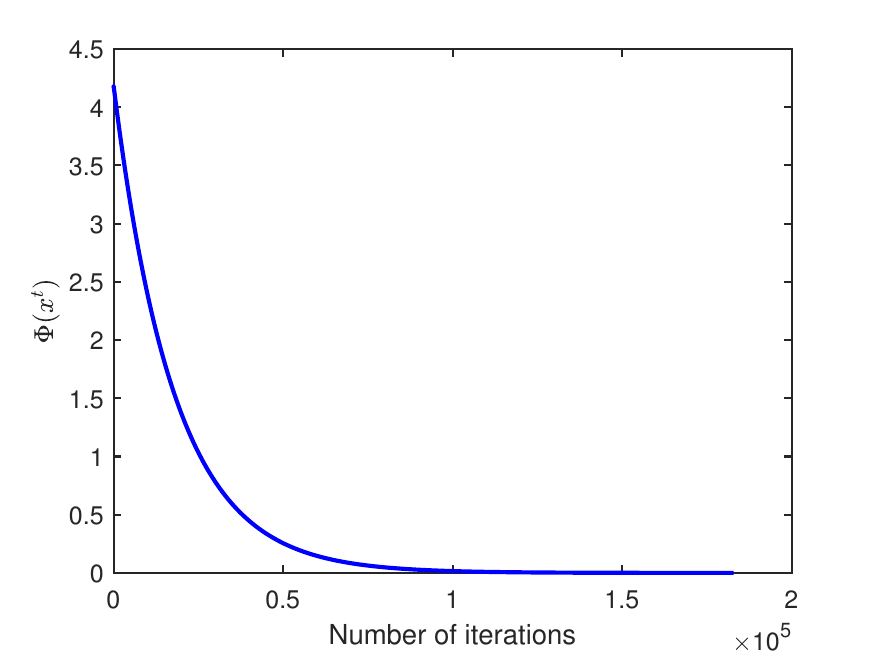}
		}
	\end{minipage}
	\begin{minipage}[t]{0.37\textwidth}
		\centering
		\subfigure[]{
			\includegraphics[scale=0.42]{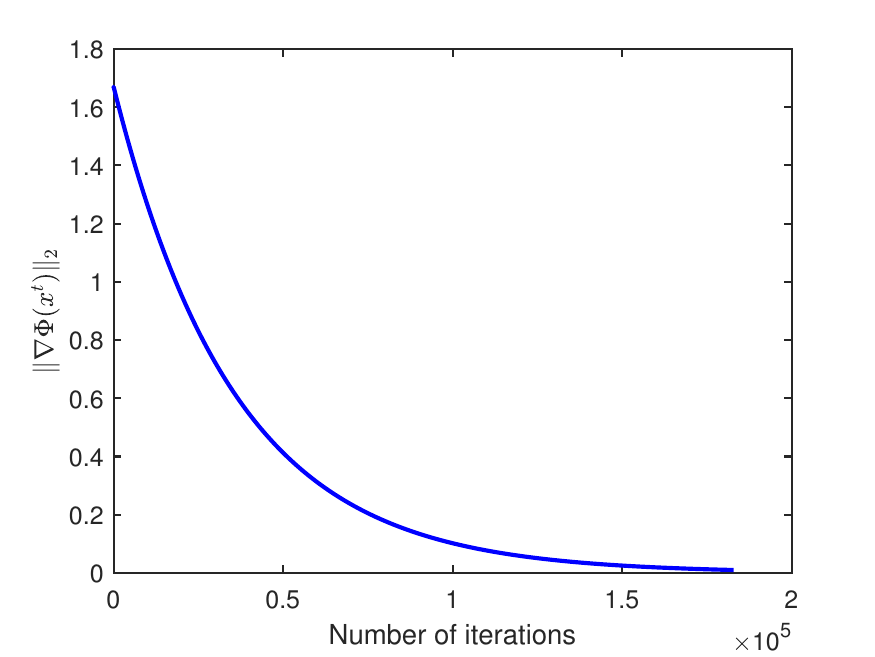}
		}
	\end{minipage}
	\caption{Performance of ASGDA}
	\vspace{-0.5cm}
	\label{fig_renzao_1}
\end{figure}
We record the performance of ASGDA in Figure \ref{fig_renzao_1}, where Figure \ref{fig_renzao_1} (a) depicts the value of the primal function and Figure \ref{fig_renzao_1} (b) depicts the norm of its gradient versus the number of iterations. Figure \ref{fig_renzao_1} shows that both the value of the primal function $\Phi(x)$ and the norm of its gradient tend to zero as the number of iterations increases, which verifies the convergence of ASGDA.

\begin{figure}[H]
	\centering
	\begin{minipage}[t]{0.37\textwidth}
		\centering
		\subfigure[]{
			\includegraphics[scale=0.42]{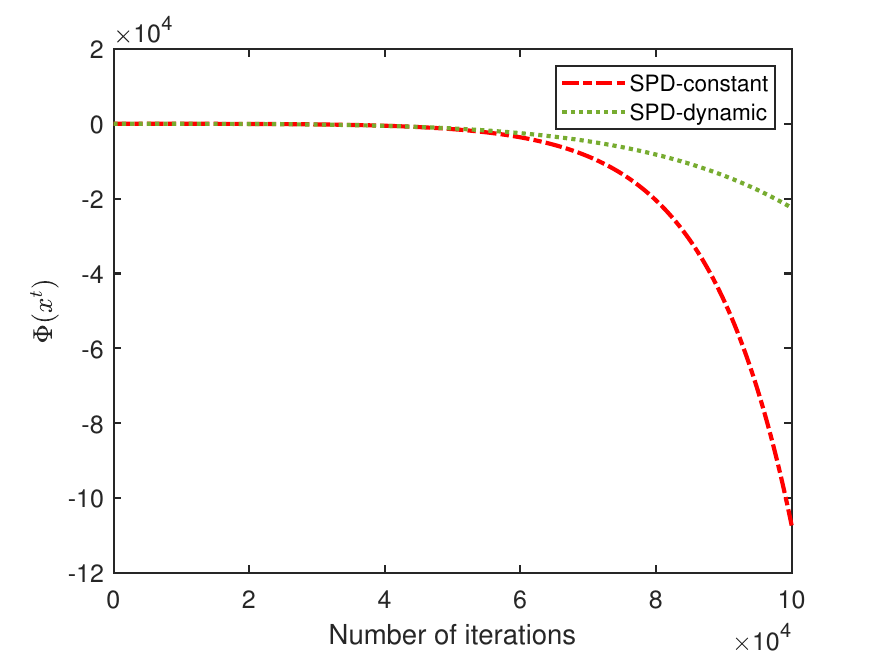}
		}
	\end{minipage}
	\begin{minipage}[t]{0.37\textwidth}
		\centering
		\subfigure[]{
			\includegraphics[scale=0.42]{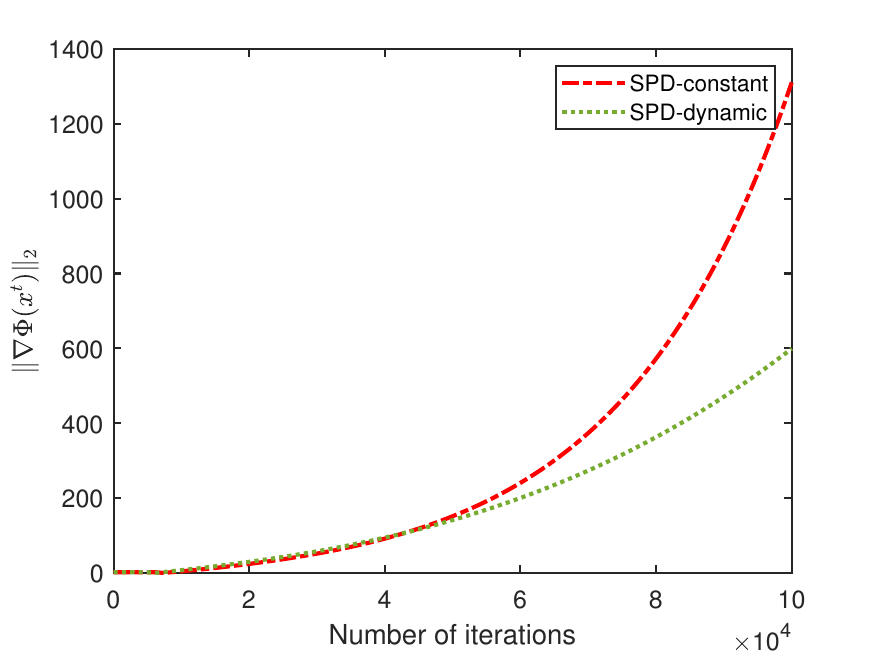}
		}
	\end{minipage}
	
	\caption{Performance of SPD}
	\vspace{-0.5cm}
	\label{fig:rezao_2}
	
\end{figure}
\vspace{-0.25cm}
Figure \ref{fig:rezao_2} (a) records the value of the primal function and Figure \ref{fig:rezao_2} (b) records the norm of its gradient versus the number of iterations, where the red dashed-dot curve and green dotted curve display the results corresponding to constant stepsize and dynamic stepsize, respectively. 
As shown in Figure \ref{fig:rezao_2}, the value of the primal function and the norm of its gradient diverge to infinity as the number of iterations increases.
The underlying reason may be that SPD is designed for finding the performative equilibrium point of SMDD with strongly convex-strongly concave  objective function.

\noindent\textbf{Distributionally robust strategic classification.}
Consider the distributionally robust strategic classification problem for loan approval in the bank, which has been introduced in Example \ref{example_1}. The corresponding logistic loss function, regularizer and distributionally robust regularizer are defined as follows $$\ell(x;a_{i},b_{i})=\log \left(1+\exp \left(-b_{i} (a_{i})^{\top} x\right)\right), $$  $$f(x)=\lambda_{1} \sum_{i=1}^{n} \frac{\alpha x_{i}^{2}}{1+\alpha x_{i}^{2}},\;\;\; g(y)=\frac{1}{2} \lambda_{2}\left\|N y-\mathbf{1}\right\|^{2}.$$
In the simulation, we adopt the Kaggle credit scoring dataset \cite{kaggle2012givemesomecredit} as the base dataset $S_0=\{\left(a_i^0,b_i^0\right)\}_{i=1}^N$, and the dataset with decision-dependent distribution $S=\{(a_i,b_i)\}_{i=1}^N$ follows a location scale model, i.e., $a_i=a_i^0+Ax, b_i=b_i^0$, where $A\in\mathbb{R}^{n\times n}$ is a matrix with all entries equal to 10 except the rows corresponding to the non-strategic features. Moreover, we set the parameters of SMDD (\ref{euq:experiment_real_example}) in Example \ref{example_1} as $\lambda_1=1, \lambda_2=\frac{10}{N^2}$ and $\alpha=1$. For ASGDA, the stepsizes $\eta_x=10^{-3},\eta_y=10^{-1}$, and for SPD, the constant stepsize $\eta=10^{-2}$
and the dynamic stepsize $\eta_{t}=\frac{1}{10+t}$, respectively. We set the batchsizes $M=200$ in both algorithms.
\begin{figure}[H]
	\centering
	\begin{minipage}[t]{0.37\textwidth}
		\centering
		\subfigure[]{		
			\includegraphics[scale=0.42]{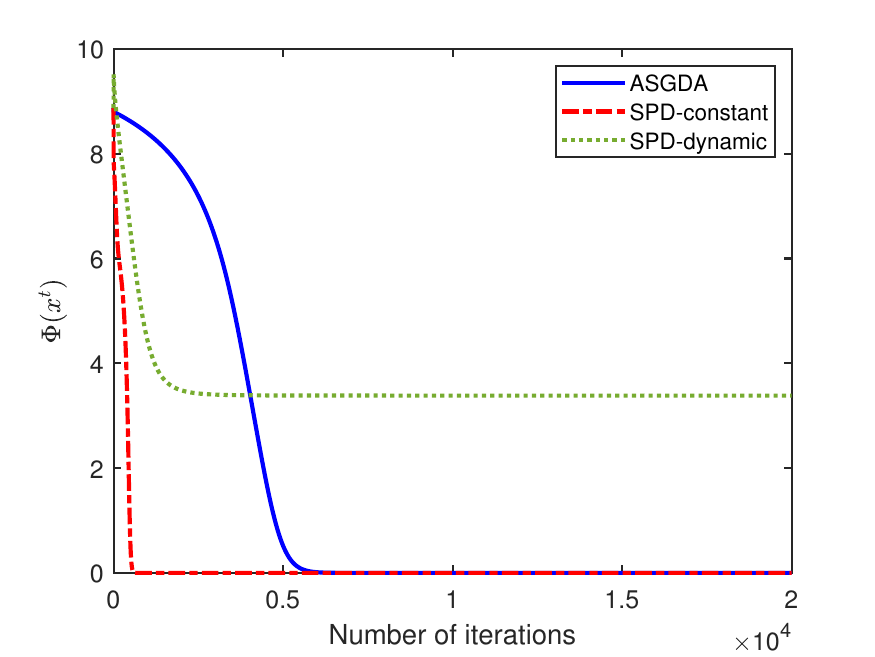}
		}
	\end{minipage}
	\begin{minipage}[t]{0.37\textwidth}
		\centering
		\subfigure[]{
			\includegraphics[scale=0.42]{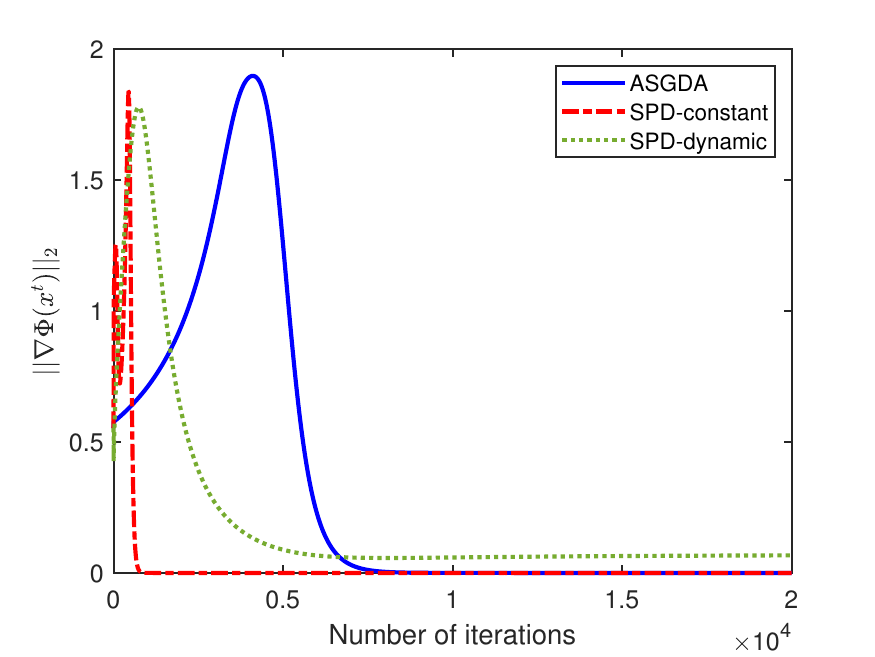}
		}
	\end{minipage}
	\begin{minipage}[t]{0.37\textwidth}
		\centering
		\subfigure[]{
			\includegraphics[scale=0.42]{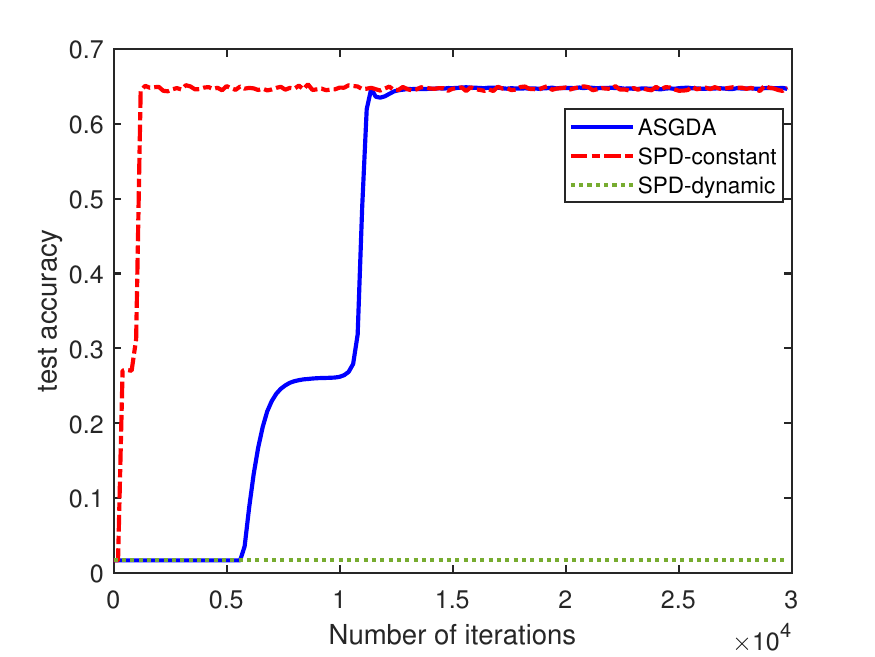}
		}
	\end{minipage}
	\caption{Performance of ASGDA and SPD}
	\label{fig:converge}
\end{figure}
We report the performance of ASGDA and SPD in Figure \ref{fig:converge}, where Figure \ref{fig:converge} (a), (b) and (c) record the performance of the value of the primal function, the norm of its gradient on the training dataset and the classification accuracy on the testing dataset of the trained classifier $x^t$, respectively.
We can observe from Figure \ref{fig:converge} (a) and (b) that the value of the primal function and the norm of its gradient corresponding to ASGDA and SPD with constant stepsize tend to zero as the number of iterations increases. For SPD with dynamic stepsize, both values tend to some constant bounded away from zero. 
Moreover, Figure \ref{fig:converge} (c) shows that both ASGDA and SPD with constant stepsize achieve a classification accuracy of approximately 65 \%.
On the other hand, compared with SPD, ASGDA requires more iterations to reach the stationary point. The underlying reason may be that ASGDA is a 
two-time-scale algorithm, whereas SPD is a single-time-scale algorithm.

\subsection{Experiments of AASGDA for SMDD-II}\label{Subsection5_2}
Consider the following stochastic minimax problem with decision-dependent distribution \cite{laguel2024high}
\begin{equation}\label{eq_experiment_4_3_1}
	\underset{x\in\mathbb{R}}{\min}\,\underset{y\in\mathbb{R}}\max\,\mathcal{L}(x,y):=\underset{z\sim\mathcal{D}(x,y)}{\mathbb{E}}\left[2\left(x+\sin(x)\right)z-4\left(y^{2}+3\sin^{2}(y)\right)\right],
\end{equation}
where $z=x+2y+\xi,\xi\sim\mathcal{N}(0,1)$. By some calculation, $\mathcal{L}(x,y)$ is $\ell$-smooth with $\ell=32$, and $\mathcal{L}(x,\cdot)$ is $\mu$-P{\L} with $\mu=8$, then the condition number $\kappa_{y}=4$. The solution to the problem is $(x^{\star},y^{\star})=(0,0)$.

We choose the initial point $(x^{0},y^{0})=(10, 10)$ and the number of iterations $T=1\times 10^5$. We compare AASGDA with the stochastic primal-dual (SPD) method \cite{wood2023stochastic}. In Algorithm \ref{algorithm:AASGDA}, the stepsizes $\eta_x=1.1097\times 10^{-5}, \eta_y=0.0313$ are calculated according to Theorem \ref{thm_1_PL}. In SPD, the stepsize $\eta=1\times 10^{-3}$. 
 \begin{figure}[H]
 	\centering
 	\begin{minipage}[t]{0.37\textwidth}
 		\centering
 		\subfigure[]{
 			\includegraphics[scale=0.42]{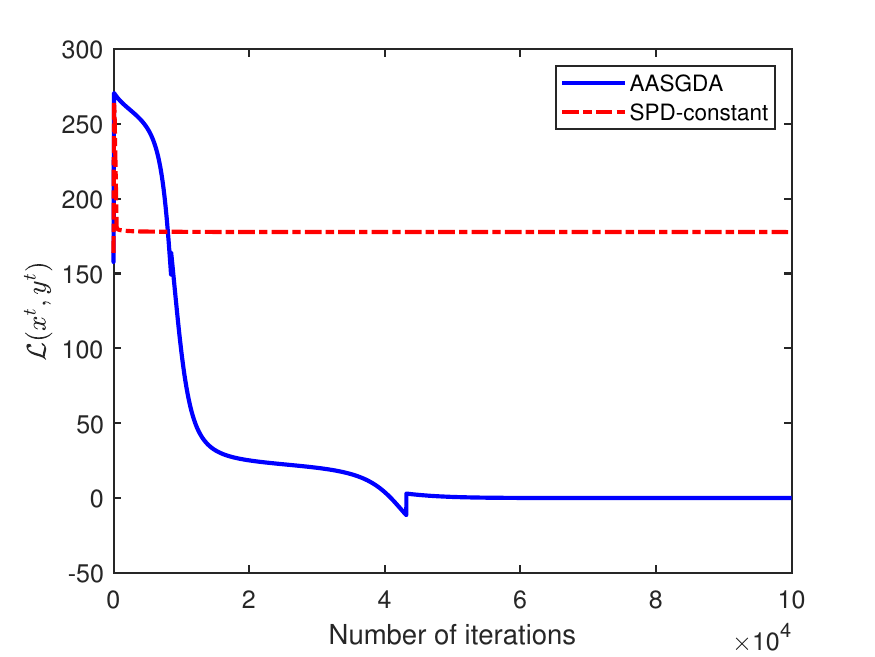}
 		}
 	\end{minipage}
 	\begin{minipage}[t]{0.37\textwidth}
 		\centering
 		\subfigure[]{
 			\includegraphics[scale=0.42]{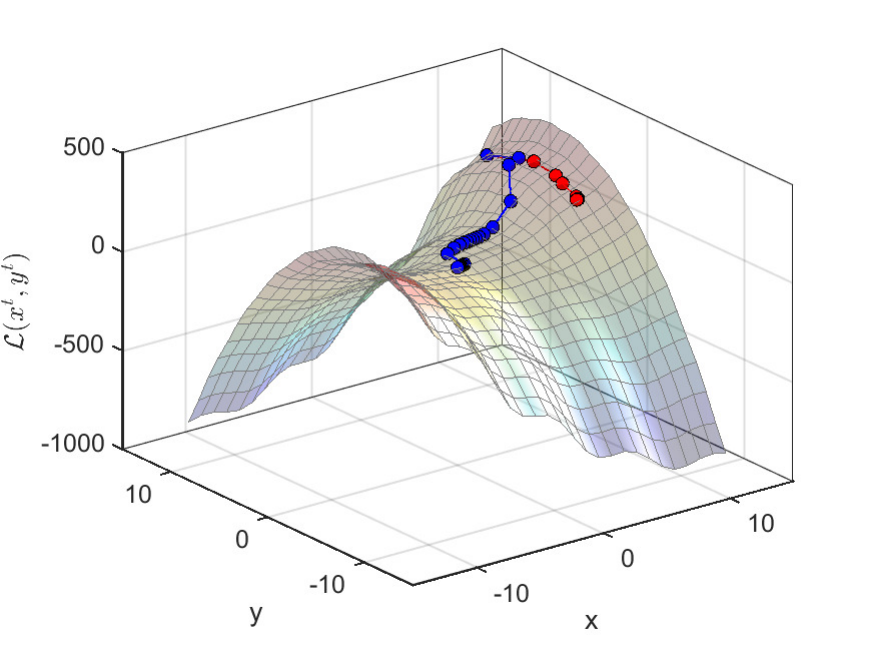}
 		}
 	\end{minipage}
 	\caption{Performance of AASGDA and SPD on synthetic example}
 	\label{fig_renzao_3}
 \end{figure}
 
 We record the performance of AASGDA and SPD in Figure \ref{fig_renzao_3}, where Figure \ref{fig_renzao_3} (a) depicts the value of the  objective function versus the number of iterations and Figure \ref{fig_renzao_3} (b) depicts the trajectory of the iterates $\{x^{t},y^{t}\}$. As we can observe from Figure \ref{fig_renzao_3} (a), the blue dashed curve corresponding to ASGDA tends to zero as the number of iterations increases, and the red dashed-dot curve corresponding to SPD tends to some constant bounded away from zero. 
 At the same time, Figure \ref{fig_renzao_3} (b) shows that AASGDA and SPD converges to different points.
The underlying reason may be that ASGDA is designed for finding the stationary point of SMDD, and SPD is designed for finding the performative equilibrium point of SMDD.

 \noindent\textbf{Election prediction.}
 Consider the following stochastic minimax problem arising from the interaction between two election prediction platforms in an election contest, which has been introduced in Example \ref{example_2}
 \begin{equation}\label{eq_experiment_4_3_2}
 	\min_{x\in\mathbb{R}^{n}}\max_{y\in\mathbb{R}^{n}}\mathcal{L}(x,y)=\frac{1}{2}\mathbb{E}\left[\|z_{1}-\theta^{\top}x\|^{2}-\|z_{2}-\theta^{\top}y\|^{2}\right]+\frac{1}{2}\|x\|^{2}-\frac{1}{2}\|y\|^{2},
 \end{equation}
 where the feature $\theta\in\mathbb{R}^{n\times d}$, $\theta_{ij}\sim\mathcal{N}(0,0.01)$, $i=1\cdots n, j=1\cdots d$, $z_i$ follows the following conditional distribution 
\begin{equation*}
	\begin{aligned}
		z_{i} \mid \theta &\sim \theta^{\top}\mathbf{1}_{n\times 1}+A_{i} x +B_{i} y+\omega_{i}, i=1,2,
	\end{aligned}
\end{equation*}
with the random variable $\omega_{i}\in\mathbb{R}^{d}$,  $\omega_{ij}\sim\mathcal{N}(0,0.01), j=1,\cdots d$ and $A_{i}, B_{i}\in\mathbb{R}^{d\times n}$ are fixed parameter matrices.

\begin{figure}[H]
	\centering
	\begin{minipage}[t]{0.37\textwidth}
		\centering
		\subfigure[]{
			\includegraphics[scale=0.42]{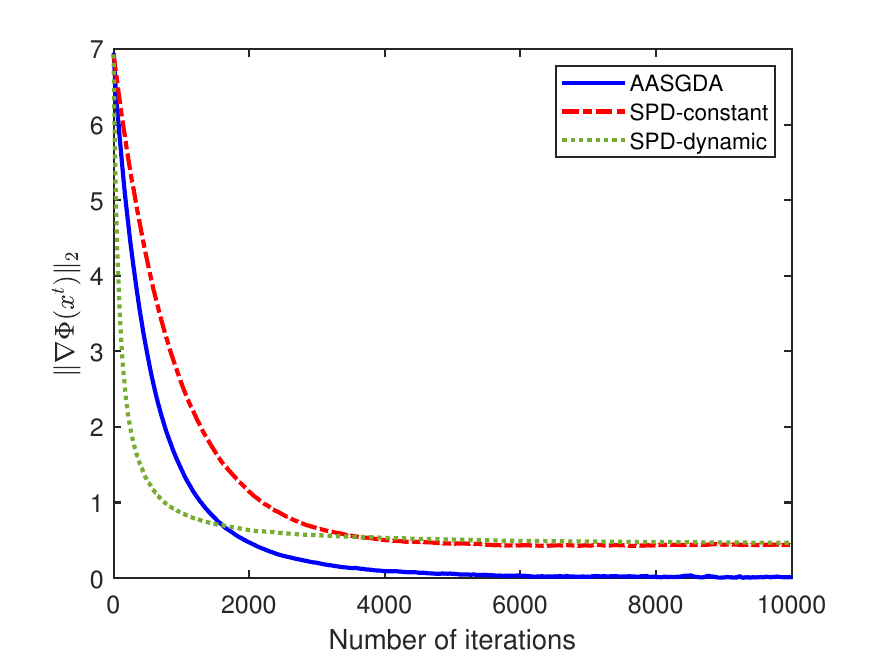}
		}
	\end{minipage}
	\begin{minipage}[t]{0.37\textwidth}
		\centering
		\subfigure[]{
			\includegraphics[scale=0.42]{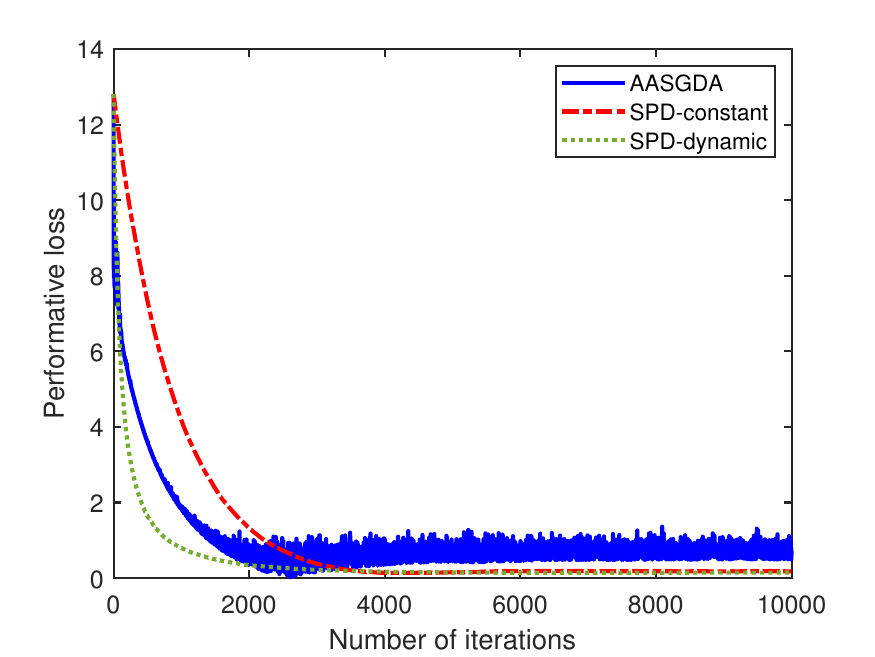}
		}
	\end{minipage}	
	\caption{Performance of AASGDA and SPD on election prediction}
	\label{fig_Election}
\end{figure}
We run the simulation with $d=10, n=10$, $A_{i}, B_{i}\in\mathbb{R}^{d\times n}, i=1,2$ are sparse parameter matrices   generated randomly. For AASGDA, the stepsizes $\eta_x= 2.7951\times 10^{-4}, \eta_y= 0.04919$, and for SPD, the constant stepsize $\eta=10^{-3}$ and the dynamic stepsize $\eta_{t}=\frac{1}{100+t}$. We record the performance of AASGDA and SPD in Figure \ref{fig_Election}, where Figure \ref{fig_Election} (a) depicts the norm of the gradient of the primal function $\Phi(\cdot)$ and Figure \ref{fig_Election} (b) depicts the norm of the performative gradient versus the number of iterations. 
As we can observe from Figure \ref{fig_Election} (a) that the dashed line corresponding to ASGDA tends to zero, while the lines corresponding to SPD with both constant and dynamic stepsizes tend to some constant bounded away from zero.
At the same time, Figure \ref{fig_Election} (b) shows that the lines corresponding to SPD with both constant and dynamic stepsizes tend to the stationary point defined by the norm of the performative gradient. This verifies that AASGDA may find the stationary point of SMDD \eqref{eq_experiment_4_3_2}, and SPD may find the performative equilibrium point of SMDD \eqref{eq_experiment_4_3_2}.

\noindent\textbf{Acknowledgment}  The research is supported by the NSFC \#12471283 and Fundamental Research Funds for the Central Universities DUT24LK001.


\begin{thebibliography}{10}
	
	\bibitem{beck2017first}
	{\sc A.~Beck}, {\em First-Order Methods in Optimization}, SIAM, Philadelphia,
	PA, 2017.
	
	\bibitem{bertsekas2015convex}
	{\sc D.~Bertsekas}, {\em Convex optimization algorithms}, Athena Scientific,
	2015.
	
	\bibitem{boct2023alternating}
	{\sc R.~I. Bo{\c{t}} and A.~B{\"o}hm}, 
	\newblock Alternating proximal-gradient steps for (stochastic)
	nonconvex-concave minimax problems.
	\newblock {\em SIAM Journal on Optimization}, 33: 1884--1913, 2023.
	
	\bibitem{chen2022faster}
	{\sc L.~Chen, B.~Yao, and L.~Luo}, 
	\newblock Faster stochastic algorithms for minimax optimization under
	Polyak--{\L}ojasiewicz conditions.
	\newblock {\em Advances in Neural Information Processing Systems},
	pages 13921--13932, 2022.
	
	\bibitem{cheung2017dynamic}
	{\sc W.~C. Cheung, D.~Simchi-Levi, and H.~Wang}, \newblock Dynamic pricing and demand learning with limited price
	experimentation.
	\newblock {\em Operations Research}, 65: 1722--1731, 2017.
	
	\bibitem{cooper2006models}
	{\sc W.~L. Cooper, T.~Homem-de Mello, and A.~J. Kleywegt}, \newblock Models of the spiral-down effect in revenue management.
	\newblock {\em Operations research}, 54: 968--987, 2006.
	
	
	\bibitem{Cutler2023Stochastic}
	{\sc J.~Cutler, D.~Drusvyatskiy, and Z.~Harchaoui}, {Stochastic
		optimization under distributional drift}, {\em Journal of Machine Learning
	Research}, 24: 1--56, 2023.
	
	\bibitem{dupacova2006optimization}
	{\sc J.~Dupačová}, \newblock Optimization under exogenous and endogenous uncertainty.
	\newblock In Lukáš Lukáš, editor, {\em Proceedings of MME06}, pages
	131--136, 2006.
	
	\bibitem{Goel2006}
	{\sc V.~Goel and I.~E. Grossmann}, 	\newblock A class of stochastic programs with decision dependent uncertainty.
	\newblock {\em Mathematical programming}, 108: 355--394, 2006.
	
	\bibitem{Hellemo2018decision}
	{\sc L.~Hellemo, P.~I. Barton, and A.~Tomasgard}, \newblock Decision-dependent probabilities in stochastic programs with
	recourse.
	\newblock {\em Computational Management Science}, 15: 369--395, 2018.
	
	
	\bibitem{huang2023enhancedadaptivegradientalgorithms}
	{\sc F.~Huang}, \newblock Enhanced adaptive gradient algorithms for nonconvex-pl minimax
	optimization.
	\newblock {\em arXiv preprint arXiv:2303.03984}.
	
	\bibitem{huang2022accelerated}
	{\sc F.~Huang, S.~Gao, J.~Pei, and H.~Huang}, \newblock Accelerated zeroth-order and first-order momentum methods from mini
	to minimax optimization.
	\newblock {\em Journal of Machine Learning Research}, 23: 1--70, 2022.
	
	\bibitem{huang2023adagda}
	{\sc F.~Huang, X.~Wu, and Z.~Hu}, \newblock Adagda: Faster adaptive gradient descent ascent methods for minimax
	optimization.
	\newblock In {\em International Conference on Artificial Intelligence and
		Statistics}, pages 2365--2389, PMLR, 2023.
		
	\bibitem{inga2022credit}
	{\sc J.~Inga and E.~Sacoto-Cabrera}, \newblock Credit default risk analysis using machine learning algorithms with
	hyperparameter optimization.
	\newblock In {\em International Conference on Science, Technology and
		Innovation for Society}, pages 81--95, 2022.
	
	\bibitem{jiang2025singleloopvariancereducedstochasticalgorithm}
	{\sc X.~Jiang, L.~Zhu, T.~Zheng, and A.~M.-C. So}, \newblock Single-loop variance-reduced stochastic algorithm for
	nonconvex-concave minimax optimization.
	\newblock {\em arXiv preprint arXiv:2501.05677}.
	
	\bibitem{jin2019minmax}
	{\sc C.~Jin, P.~Netrapalli, and M.~I. Jordan}, \newblock Minmax optimization: Stable limit points of gradient descent ascent
	are locally optimal.
	\newblock {\em arXiv preprint arXiv:1902.00618}.
	
	\bibitem{jonsbraaten1998class}
	{\sc T.~W. Jonsbr{\aa}ten, R.~J. Wets, and D.~L. Woodruff}, \newblock A class of stochastic programs withdecision dependent random
	elements.
	\newblock {\em Annals of Operations Research}, 82: 83--106, 1998.
	
	\bibitem{Jonsbraten1998}
	{\sc T.~Jonsbråten}, \newblock Oil field optimization under price uncertainty.
	\newblock {\em Journal of the Operational Research Society}, 49: 811--818,
	1998.
	
	\bibitem{kaggle2012givemesomecredit}
	{\sc Kaggle}, {\em Give me some credit}.
	\newblock \url{https://www.kaggle.com/c/GiveMeSomeCredit/data}, 2012.
	
	\bibitem{karimi2016linear}
	{\sc H.~Karimi, J.~Nutini, and M.~Schmidt}, \newblock Linear convergence of gradient and proximal-gradient methods under
	the Polyak-{\L}ojasiewicz condition.
	\newblock In {\em Machine Learning and Knowledge Discovery in Databases:
		European Conference, ECML PKDD 2016, Riva del Garda, Italy, September 19-23,
		2016, Proceedings, Part I 16}, pages 795--811, Springer, 2016.
	
	\bibitem{laguel2024high}
	{\sc Y.~Laguel, Y.~Syed, N.~S. Aybat, and M.~G{\"u}rb{\"u}zbalaban}, {
		High-probability complexity guarantees for nonconvex minimax problems},  {\em arXiv preprint arXiv:2405.14130}.
	
	\bibitem{li2022tiada}
	{\sc X.~Li, J.~Yang, and N.~He}, \newblock High-probability complexity guarantees for nonconvex minimax
	problems.
	\newblock {\em arXiv preprint arXiv:2405.14130}.
	
	\bibitem{liebig2017dynamic}
	{\sc T.~Liebig, N.~Piatkowski, C.~Bockermann, and K.~Morik}, 	
	\newblock Dynamic route planning with real-time traffic predictions.
	\newblock {\em Information Systems}, 64: 258--265, 2017.
	
	\bibitem{lin2020gradient}
	{\sc T.~Lin, C.~Jin, and M.~Jordan}, \newblock On gradient descent ascent for nonconvex-concave minimax problems.
	\newblock In {\em International Conference on Machine Learning}, pages
	6083--6093, PMLR, 2020.
	
	\bibitem{luo2020stochastic}
	{\sc L.~Luo, H.~Ye, Z.~Huang, and T.~Zhang}, \newblock Stochastic recursive gradient descent ascent for stochastic
	nonconvex-strongly-concave minimax problems.
	\newblock In {\em Advances in Neural Information Processing Systems},
	33: 20566--20577, 2020.
	
	\bibitem{miller2021outside}
	{\sc J.~P. Miller, J.~C. Perdomo, and T.~Zrnic}, {Outside the echo chamber:
		Optimizing the performative risk}, in {\em International Conference on Machine
	Learning}, pages 7710--7720, PMLR, 2021.
	
	\bibitem{narang2023multiplayer}
	{\sc A.~Narang, E.~Faulkner, D.~Drusvyatskiy, M.~Fazel, and L.~J. Ratliff},
	\newblock Multiplayer performative prediction: Learning in decision-dependent
	games.
	\newblock {\em Journal of Machine Learning Research}, 24: 1--56, 2023.
	
	\bibitem{nouiehed2019solving}
	{\sc M.~Nouiehed, M.~Sanjabi, T.~Huang, J.~D. Lee, and M.~Razaviyayn}, \newblock Solving a class of non-convex min-max games using iterative first
	order methods.
	\newblock {\em Advances in Neural Information Processing Systems}, 32, 2019.
	
	\bibitem{perdomo2020performative}
	{\sc J.~Perdomo, T.~Zrnic, C.~Mendler-D{\"u}nner, and M.~Hardt}, \newblock Performative prediction.
	\newblock In {\em International Conference on Machine Learning}, pages
	7599--7609, PMLR, 2020.
	
	\bibitem{polyak1963gradient}
	{\sc B.~T. Polyak}, 	\newblock Gradient methods for the minimisation of functionals.
	\newblock {\em USSR Computational Mathematics and Mathematical Physics},
	3: 864--878, 1963.
	
	\bibitem{rafique2022weakly}
	{\sc H.~Rafique, M.~Liu, Q.~Lin, and T.~Yang}, \newblock Weakly-convex--concave min--max optimization: provable algorithms and
	applications in machine learning.
	\newblock {\em Optimization Methods and Software}, 37: 1087--1121, 2022.
	
	\bibitem{robinson2024loan}
	{\sc N.~Robinson and N.~Sindhwani}, \newblock Loan default prediction using machine learning.
	\newblock In {\em International Conference on Reliability, Infocom Technologies
		and Optimization (Trends and Future Directions)(ICRITO)}, pages 1--5, IEEE,
	2024.
	
	\bibitem{wood2023stochastic}
	{\sc K.~Wood and E.~Dall’Anese}, \newblock Stochastic saddle point problems with decision-dependent
	distributions.
	\newblock {\em SIAM Journal on Optimization}, 33: 1943--1967, 2023.
	
	\bibitem{wood2024solvingdecisiondependentgameslearning}
	{\sc K.~Wood, A.~S. Zamzam, and E.~Dall'Anese}, \newblock Solving decision-dependent games by learning from feedback.
	\newblock {\em IEEE Open Journal of Control Systems}, 3: 295--309, 2024.
	
	\bibitem{xu2024stochastic}
	{\sc Q.~Xu, X.~Zhang, N.~S. Aybat, and M.~Gürbüzbalaban}, {A stochastic
		gda method with backtracking for solving nonconvex (strongly) concave minimax
		problems}, {\em arXiv preprint
		arXiv:arXiv:2403.07806}.
	
\bibitem{xu2021enhanced}
{\sc T.~Xu, Z.~Wang, Y.~Liang, and H.~V.~Poor}, {\em Enhanced first and zeroth order variance reduced algorithms for min-max optimization}, \url{https://openreview.net/forum?id=X5ivSy4AHx}, 2021.
	
	\bibitem{yang2022faster}
	{\sc J.~Yang, A.~Orvieto, A.~Lucchi, and N.~He}, \newblock Faster single-loop algorithms for minimax optimization without strong
	concavity.
	\newblock In {\em International Conference on Artificial Intelligence and
		Statistics}, pages 5485--5517, PMLR, 2022.
	
	\bibitem{zhang2024accelerated}
	{\sc H.~Zhang and Z.~Xu}, 	\newblock An accelerated first-order regularized momentum descent ascent
	algorithm for stochastic nonconvex-concave minimax problems.
	\newblock {\em Computational Optimization and Applications}, pages 1--26, 2024.
	
	
	\bibitem{zhang2020single}
	{\sc J.~Zhang, P.~Xiao, R.~Sun, and Z.~Luo}, {\em A single-loop smoothed
		gradient descent-ascent algorithm for nonconvex-concave min-max problems},
	Advances in neural information processing systems, 33: 7377--7389, 2020.
	
	\bibitem{zhang2022sapd+}
	{\sc X.~Zhang, N.~S. Aybat, and M.~Gurbuzbalaban}, 
	\newblock Sapd+: An accelerated stochastic method for nonconvex-concave minimax
	problems.
	\newblock {\em Advances in Neural Information Processing Systems},
	35: 21668--21681, 2022.
	
\end{thebibliography}

\appendix
\begin{appendices}
\newpage
\noindent\textbf{\Large{Appendix}}
\section{Proof of Lemmas in Section \ref{section2}}\label{Appendix_for_Section2}
\subsection{Proof of Lemma \ref{lem:expvar}}\label{Appendix_for_Section2_1_Lem2}
\begin{proof}
	By the definition of $G_x^{t}(x^t,y^t,z^t)$ in \eqref{eq:adaptive_stograd},
	\begin{equation*}\label{eq:app2_1}
		\begin{aligned}
			&\left\|\mathbb{E}_{t}\left[\frac{1}{M}\sum_{i=1}^{M}G_x^{t}(x^t,y^t,z^t_{i})\right]-\nabla_{x}\mathcal{L}(x^{t},y^{t})\right\|\\
			=&\left\|\left(\nabla_{x}\psi^{t}(x^t,y^t)-\nabla_{x}\psi(x^t,y^t)\right)^{\top}\mathbb{E}_{z\sim\mathcal{D}(x^t,y^t)}\nabla_{z}l(x^t,y^t,z) \right\|\\
			\leq& L_{1}\left\|\nabla_{x}\psi^{t}(x^t,y^t)-\nabla_{x}\psi(x^t,y^t) \right\|_{F},
		\end{aligned}
	\end{equation*}
	where the inequality follows from Assumption \ref{ass_bounded_moment}, which verifies Lemma \ref{lem:expvar} (a).
	
	By Assumption \ref{ass:l_variance_bound},
	\begin{small}
		\begin{equation*}\label{eq:app2_2}
			\begin{aligned}
				&\mathbb{E}_{t}\left\|\frac{1}{M}\sum_{i=1}^{M}G^{t}_x(x^t,y^t,z^t_{i})
				-\mathbb{E}_{t}\left[\frac{1}{M}\sum_{i=1}^{M}G_x^{t}(x^t,y^t,z^t_{i})\right] \right\|^{2}\\
				=&\frac{1}{M^{2}}\sum_{i=1}^{M}\mathbb{E}_{t}\left\|\left[\begin{array}{cc}
					I & 0 \\
					0 & (\nabla_{x}\psi^t(x^t,y^t))^{\top}
				\end{array}\right]\left(\nabla_{x, z} \ell\left(x^{t},y^{t},z_{i}^{t}\right)-\underset{z\sim \mathcal{D}\left(x^t,y^t\right)}{\mathbb{E}} \nabla_{x, z} \ell\left(x^{t},y^{t}, z\right)\right)\right\|^{2}\\
				\leq&\frac{1}{M^{2}}\left\|\left[\begin{array}{cc}
					I & 0 \\
					0 & (\nabla_{x}\psi^t(x^t,y^t))^{\top}
				\end{array}\right]\right\|_{F}^{2}\sum_{i=1}^{M}\mathbb{E}_{t}\left\|\nabla_{x, z} \ell\left(x^{t},y^{t},z_{i}^{t}\right)-\underset{z\sim \mathcal{D}\left(x^t,y^t\right)}{\mathbb{E}} \nabla_{x, z} \ell\left(x^{t},y^{t}, z\right)\right\|^{2}\\
				\leq&  \left(1+\left\|\nabla_{x}\psi^{t}(x^t,y^t)\right\|_F^{2} \right)\frac{\sigma^{2}}{M}.
			\end{aligned}
		\end{equation*}
	\end{small}Then combining with Lemma \ref{lem:expvar} (a), Lemma \ref{lem:expvar} (b) holds. 
	Similarly, we can deduce that Lemma \ref{lem:expvar} (c) and (d) hold.
	The proof is complete.
\end{proof}

\section{Proof of Lemmas in Subsection \ref{subsection3_1}}\label{Appendix_for_Section3_1}
\subsection{Proof of Lemma \ref{lem_value}}\label{Appendix_for_Section3_1_Lem1}

\begin{proof}
	\begin{equation*}\label{Lemma4_equ1}
		\Phi\left(x^{t+1}\right) \leq \Phi\left(x^t\right) + \left(x^{t+1} - x^t\right)^\top \nabla\Phi\left(x^t\right) + \kappa_y \ell \left\|x^{t+1} - x^t\right\|^2.
	\end{equation*}By the iteration of $x$ in \eqref{equ:alg_x},
	\begin{small}
		\begin{equation*}
			\begin{aligned}
				\Phi(x^{t+1}) \leq & \Phi(x^t) - \eta_{x}\left\|\nabla\Phi(x^t)\right\|^{2} + \kappa_{y}\ell\eta_{x}^{2}\left\|\frac{1}{M}\sum_{i=1}^{M}G_{x}^t(x^t,y^t,z_{i}^t)\right\|^{2} \\
				& + \eta_{x}\left(\nabla\Phi(x^t) - \frac{1}{M}\sum_{i=1}^{M}G_{x}^t(x^t,y^t,z_{i}^t)\right)^{\top}\nabla\Phi(x^t).
			\end{aligned}
		\end{equation*}
	\end{small}Taking expectation on both sides of the above inequality, conditioned on $\left(x^{t}, y^{t},\psi^{t}(\cdot)\right)$,
		\begin{equation}\label{Lemma2_equ2}
			\begin{aligned}
				&\mathbb{E}_{t}\left[\Phi(x^{t+1})\right] \\
				\leq& \Phi(x^t) - \eta_x\left\|\nabla\Phi(x^t)\right\|^2 + \eta_x\left(\nabla\Phi(x^t) - \nabla_x\mathcal{L}(x^t,y^t)\right)^\top\nabla\Phi(x^t) \\
				&+ \eta_x\left(\nabla_x\mathcal{L}(x^t,y^t) - \mathbb{E}_{t}\left[\frac{1}{M}\sum_{i=1}^M G_x^t(x^t,y^t,z_i^t)\right]\right)^\top\nabla\Phi(x^t) \\
				&+ \kappa_y\ell\eta_x^2\mathbb{E}_{t}\left\|\frac{1}{M}\sum_{i=1}^M G_x^t(x^t,y^t,z_i^t)\right\|^2,
			\end{aligned}
		\end{equation}where
	$\mathbb{E}_{t}[\,\cdot\,]=\mathbb{E}\left[\,\cdot\,|\left(x^{t}, y^{t},\psi^{t}(\cdot)\right)\right]$.
	
	By Young's inequality and Lemma \ref{lem:expvar} (a), we have that the third term
	\begin{equation*}\label{Lemma2_equ3}
		\begin{aligned}
			\left(\nabla\Phi(x^t) - \nabla_x\mathcal{L}(x^t, y^t)\right)^\top \nabla\Phi(x^t) 
			\leq \left\|\nabla\Phi(x^t) - \nabla_x\mathcal{L}(x^t, y^t)\right\|^2 + \frac{1}{4}\left\|\nabla\Phi(x^t)\right\|^2,
		\end{aligned}
	\end{equation*}
	and the fourth term
	\begin{equation*}
		\begin{aligned}
			&\left(\nabla_{x}\mathcal{L}(x^t,y^t)-\mathbb{E}_{t}\left[\frac{1}{M}\sum_{i=1}^{M}G_{x}^t(x^t,y^t,z_{i}^t)\right]\right)^{\top}\nabla\Phi(x^t) \\
			\leq& L_{1}^{2}\left\|\nabla_{x}\psi^t(x^t,y^t)-\nabla_{x}\psi(x^t,y^t)\right\|_{F}^{2}+\frac{1}{4}\left\|\nabla\Phi(x^t)\right\|^{2}.
		\end{aligned}
	\end{equation*}
	
	For the last term on the right-hand side of
	inequality \eqref{Lemma2_equ2},
	\begin{equation*}
		\begin{aligned}
			&\mathbb{E}_{t}\left\|\frac{1}{M}\sum_{i=1}^{M}G_{x}^t(x^t,y^t,z_{i}^t)\right\|^{2} \\
			\leq& 2L_{1}^{2}\left\|\nabla_{x}\psi^t(x^t,y^t)-\nabla_{x}\psi(x^t,y^t)\right\|_{F}^{2} + \frac{1}{M}\left(1+\left\|\nabla_{x}\psi^t(x^t,y^t)\right\|_{F}^{2}\right)\sigma^{2} \\
			&+ 4\left(\left\|\nabla_{x}\Phi(x^t)-\nabla_{x}\mathcal{L}(x^t,y^t)\right\|^{2} + \left\|\nabla\Phi(x^t)\right\|^{2}\right) \\
			\leq& 2L_{1}^{2}\left\|\nabla_{x}\psi^t(x^t,y^t)-\nabla_{x}\psi(x^t,y^t)\right\|_{F}^{2} + \frac{1}{M}\left(1+\left\|\nabla_{x}\psi^t(x^t,y^t)\right\|_{F}^{2}\right)\sigma^{2} \\
			&+ 4\left(\ell^{2}\left\|y^{\star}(x^t)-y^t\right\|^{2} + \left\|\nabla\Phi(x^t)\right\|^{2}\right),
		\end{aligned}
	\end{equation*}
	where the first inequality follows from Lemma~\ref{lem:expvar} (a)-(b) and the last inequality follows from $\ell$-smoothness of $\mathcal{L}(\cdot)$.
	
	\noindent Then
	\begin{equation*}
		\begin{aligned}
			\mathbb{E}_{t}\Phi(x^{t+1}) \leq & \Phi(x^t) - \left(\frac{1}{2}\eta_x - 4\kappa_y\ell\eta_x^2\right)\left\|\nabla\Phi(x^t)\right\|^2 
			+ \left(\eta_x + 4\kappa_y\ell\eta_x^2\right)\left\|\nabla\Phi(x^t) - \nabla_x L(x^t, y^t)\right\|^2 \\
			& + \left(\eta_x + 2\kappa_y\ell\eta_x^2\right)L_1^2\left\|\nabla_x\psi^t(x^t, y^t) - \nabla_x\psi(x^t, y^t)\right\|_F^2 \\
			& + \frac{1}{M}\kappa_y\ell\eta_x^2\left(1 + \left\|\nabla_x\psi^t(x^t, y^t)\right\|_F^2\right)\sigma^2.
		\end{aligned}
	\end{equation*}
	\noindent Subsequently,  we have by the fact $\eta_{x}\leq \frac{1}{16\kappa_{y}\ell}$ that
	\begin{equation*}\label{eq:lemma2_1}
		\begin{aligned}
			\mathbb{E}_{t}\Phi(x^{t+1}) \leq & \Phi(x^t) - \frac{1}{4}\eta_x \left\|\nabla\Phi(x^t)\right\|^2 + \frac{5}{4}\eta_x \ell^2 \left\|y^{\star}(x^t) - y^t\right\|^2 \\
			& + \left(\eta_x + 2\kappa_y \ell \eta_x^2\right) L_1^2 \left\|\nabla_x \psi^t(x^t, y^t) - \nabla_x \psi(x^t, y^t)\right\|_F^2 \\
			& + \frac{1}{M} \kappa_y \ell \eta_x^2 \left(1 + \left\|\nabla_x \psi^t(x^t, y^t)\right\|_F^2\right) \sigma^2.
		\end{aligned}
	\end{equation*}
	Taking expectation on both sides of the above inequality, we arrive at inequality \eqref{eq:stronglyconcave_descentPhi}. The proof is complete.
\end{proof}

\subsection{Proof of Lemma \ref{lemma_delta_recursion}}\label{Appendix_for_Section3_1_Lem2}
Before providing the proof, we establish a technical result first.
\begin{lem}[\textbf{Technical lemma}]\label{lemma_f_delta}
	Consider an optimization problem $\underset{x\in\mathcal{X}}{\min}f(x)$,
	where $f:\mathbb{R}^{n}\rightarrow\mathbb{R}$ is $L$-smooth and $\mu$-strongly convex with $\mu\in(0, L]$, and $\mathcal{X}\subset\mathbb{R}^{n}$ is a closed convex set. Denote $x^{\star}$ as the unique solution. A stochastic gradient method performs the update in each iteration as follows
	\begin{equation}\label{sto_iter}
		x^{t+1}=P_{\mathcal{X}}(x^t-\eta g^t),
	\end{equation}
	where $\eta>0$ is a fixed stepsize, and $g^t$ is an estimator of $\nabla f(x^t)$ satisfying 
	\begin{equation}\label{bias_variance_technical}
		\left\| \mathbb{E}_t [g^t] - \nabla f(x^t) \right\| \leq C_t,\, 
		\mathbb{E}_t \left\| g^t - \mathbb{E}_t [g^t] \right\|^2 \leq D_t,
	\end{equation}
	with $	C_t, D_t\geq 0$ and $\mathbb{E}_{t}[\,\cdot\,]=\mathbb{E}\left[\,\cdot\mid x_{t}\right]$ denotes the conditional expectation.
	With $\eta=\frac{1}{2(\mu+L)}$, we have 
	\begin{equation}
		\mathbb{E}_{t}\left[\|x^{t+1}-x^{\star}\|^{2}\right]\leq \left(1-\frac{\mu}{2(\mu+L)}\right)\|x^{t}-x^{\star}\|^{2}+\left(\frac{1}{2(\mu+L)^{2}}+\frac{1}{2\mu L}\right)C_t^{2}+\frac{1}{4(\mu+L)^{2}}D_t.
	\end{equation}
\end{lem}
\begin{proof}
	By the iteration \eqref{sto_iter},
	\begin{equation*}
		\begin{aligned}
			&\left\|x^{t+1}-x^{\star}\right\|^{2}\\
			=&\left\|P_{\mathcal{X}}\left(x^{t}-\eta g^{t}\right)-P_{\mathcal{X}}\left(x^{\star}-\eta\nabla f(x^{\star})\right)\right\|^{2} \\
			\leq&\left\|x^{t}-x^{\star}\right\|^{2}-2\eta\left\langle g^{t}-\nabla f(x^{\star}),x^{t}-x^{\star}\right\rangle+\eta^{2}\left\|g^{t}-\nabla f(x^{\star})\right\|^{2}.
		\end{aligned}
	\end{equation*}	
	
	\noindent Taking expectation on both sides of the above inequality, conditioned on $x^t$,
	\begin{small}
		\begin{equation}\label{eq:lem_delta_1}
			\begin{aligned}
				\mathbb{E}_t \left[ \left\| x^{t+1} - x^{\star} \right\|^2 \right]\leq &\left\|x^{t}-x^{\star}\right\|^{2}
				- 2\eta \mathbb{E}_t \left[\left\langle g^t - \nabla f(x^t), x^t - x^{\star} \right\rangle\right]\\
				&- 2\eta  \left\langle \nabla f(x^t) - \nabla f(x^{\star}), x^t - x^{\star} \right\rangle  \\
				&+\eta^2\mathbb{E}_t \left\| g^t - \mathbb{E}_t [g^t] \right\|^2 + \eta^2\left\| \mathbb{E}_t [g^t]- \nabla f(x^{\star}) \right\|^2.
			\end{aligned}
		\end{equation}
	\end{small}For the second term on the right-hand side of inequality \eqref{eq:lem_delta_1}, we have by Young's inequality that
	\begin{equation*}
		\begin{aligned}
			-2 \eta \mathbb{E}_{t}\left[\left\langle g^{t}-\nabla f\left(x^{t}\right), x^{t}-x^{\star}\right\rangle\right]\leq \frac{1}{\Delta}\|\mathbb{E}_{t}[g^t]-\nabla f(x^t)\|^{2}+\Delta\eta^{2}\|x^t-x^{\star}\|^{2},\;\;\forall \Delta>0.
		\end{aligned}
	\end{equation*}
	For the third term on the right-hand side of inequality \eqref{eq:lem_delta_1}, we have by \cite[Proposition 6.1.9 (b)]{bertsekas2015convex} that
	\begin{equation*}
		\left( \nabla f(x^t) - \nabla f(x^{\star}) \right)^{\top} (x^t - x^{\star}) \geq \frac{\mu L}{\mu + L} \|x^t - x^{\star}\|^2 + \frac{1}{\mu + L} \|\nabla f(x^t) - \nabla f(x^{\star})\|^2.
	\end{equation*}
	For the last term on the right-hand side of inequality \eqref{eq:lem_delta_1}, 
	\begin{equation*}
		\begin{aligned}
			\eta^{2}\left\|\mathbb{E}_{t}\left[g^{t}\right]-\nabla f\left(x^{\star}\right)\right\|^{2}
			\leq&2\eta^{2}\|\mathbb{E}_{t}[g^t]-\nabla f(x^t)\|^{2}+2\eta^{2}\|\nabla f(x^t)-\nabla f(x^{\star})\|^{2}.
		\end{aligned}
	\end{equation*}	
	Then, 
	\begin{equation*}
		\begin{aligned}
			\mathbb{E}_t \left[ \left\| x^{t+1} - x^{\star} \right\|^2 \right]\leq &\left\|x^{t}-x^{\star}\right\|^{2}-2\eta \frac{\mu L}{\mu + L} \|x^t - x^{\star}\|^2+\Delta\eta^{2}\|x^t-x^{\star}\|^{2}\\
			&+\left(2\eta^{2}-2\eta \frac{1}{\mu + L}\right)\|\nabla f(x^t) - \nabla f(x^{\star})\|^2\\
			&+\left(2\eta^{2}+\frac{1}{\Delta}\right)\|\mathbb{E}_{t}[g^t]-\nabla f(x^t)\|^{2}+\eta^{2}\mathbb{E}_{t}\|g^{t}-\mathbb{E}_{t}[g^t]\|^{2}.
		\end{aligned}
	\end{equation*}
	Given $\eta=\frac{1}{2(\mu+L)}$ and $\Delta=2\mu L$, we have
	\begin{equation*}
		\begin{aligned}
			\mathbb{E}_{t}\left[\|x^{t+1}-x^{\star}\|^{2}\right]
			\leq&\left(1-\frac{\mu}{2(\mu+L)}\right)\|x^{t}-x^{\star}\|^{2}+\frac{1}{4(\mu+L)^{2}}\mathbb{E}_{t}\|g^{t}-\mathbb{E}_{t}[g^t]\|^{2}\\
			&+\left(\frac{1}{2(\mu+L)^{2}}+\frac{1}{2\mu L}\right)\|\mathbb{E}_{t}[g^t]-\nabla f(x^t)\|^{2}\\
			\leq& \left(1-\frac{\mu}{2(\mu+L)}\right)\|x^{t}-x^{\star}\|^{2}+\frac{1}{4(\mu+L)^{2}}D_t+\left(\frac{1}{2(\mu+L)^{2}}+\frac{1}{2\mu L}\right)C_t^{2},
		\end{aligned}
	\end{equation*}
	where the first inequality follows from $\mu$-strong convexity of $f(\cdot)$ and the second inequality follows from inequality \eqref{bias_variance_technical}. The proof is complete.
\end{proof}

\begin{proof}
	By the definition of $\delta_t$ and Young's inequality,
	\begin{small}
		\begin{equation}\label{eq:recursion_1}
			\begin{aligned}
				\delta_{t}=&\mathbb{E}\Vert y^{\star}(x^{t})-y^{t}\Vert^{2}\\
				\leq& (1+\Delta)\mathbb{E}\left[\left\|y^{\star}\left(x^{t}\right)-y^{\star}\left(x^{t-1}\right)\right\|^{2}\right]+(1+\frac{1}{\Delta})\mathbb{E}\left[\left\|y^{\star}(x^{t-1})-y^t \right\|^2\right],\forall \Delta>0.
			\end{aligned}
		\end{equation}
	\end{small}
	
	For the first term on the right-hand side of inequality \eqref{eq:recursion_1},
	\begin{equation*}
		\begin{aligned}
			&\mathbb{E}\left[\left\|y^{\star}\left(x^{t}\right)-y^{\star}\left(x^{t-1}\right)\right\|^{2}\right]\\
			\leq&\kappa_{y}^{2}\mathbb{E}\left[\left\|x^{t}-x^{t-1}\right\|^2\right]\\
			\leq&
			2L_1^2\kappa_{y}^{2}\eta_{x}^{2}
			\mathbb{E}\Vert\nabla_{x}\psi^{t-1}(x^{t-1},y^{t-1})-\nabla_{x}\psi(x^{t-1},y^{t-1}) \Vert^2_{F}+2\kappa_{y}^{2}\eta_{x}^{2}\mathbb{E}\left\|\nabla_{x}\mathcal{L}(x^{t-1},y^{t-1})\right\|^2\\
			&+\kappa_{y}^{2}\eta_{x}^{2}\mathbb{E}\left(1+\left\|\nabla_{x}\psi^{t-1}(x^{t-1},y^{t-1})\right\|_F^{2} \right)\frac{\sigma^{2}}{M}\\
			\leq&2L_1^2\kappa_{y}^{2}\eta_{x}^{2}
			\mathbb{E}\Vert\nabla_{x}\psi^{t-1}(x^{t-1},y^{t-1})-\nabla_{x}\psi(x^{t-1},y^{t-1}) \Vert^2_{F}+4\ell^2\kappa_{y}^{2}\eta_{x}^{2}\delta_{t-1}+4\kappa_{y}^{2}\eta_{x}^{2}\mathbb{E}\Vert\nabla\Phi(x^{t-1})\Vert^2\\
			&+\kappa_{y}^{2}\eta_{x}^{2}\mathbb{E}\left(1+\left\|\nabla_{x}\psi^{t-1}(x^{t-1},y^{t-1})\right\|_F^{2} \right)\frac{\sigma^{2}}{M},
		\end{aligned}
	\end{equation*}
	where the first inequality follows from $\kappa_{y}$-Lipschitz continuity of $y^{\star}(\cdot)$ \cite[Lemma 4.3]{lin2020gradient}, 
	the second inequality follows from Lemma \ref{lem:expvar} (a)-(b) and the third inequality follows from $\ell$-smoothness of $\mathcal{L}(\cdot)$.
	
	Noting that the conditions of Lemma \ref{lemma_f_delta} holds, 
	we have by Lemma \ref{lem:expvar} and Lemma \ref{lemma_f_delta} that the iterates of $y$ obtained by the iteration step \eqref{equ:alg_y} with the stepsize $\eta_{y}=\frac{1}{2(\ell+\mu)}$ satisfy
	\begin{equation*}
		\begin{aligned}
			\mathbb{E}\left[\left\|y^{\star}(x^{t-1})-y^{t}\right\|^{2} \right]\leq& \left(1-\frac{1}{2(1+\kappa_{y})}\right)\delta_{t-1}+\frac{1}{4(\mu+\ell)^{2}}\left(1+\mathbb{E}\left\|\nabla_{y}\psi^{t-1}(x^{t-1},y^{t-1})\right\|_F^{2} \right)\frac{\sigma^{2}}{M}\\
			&+\left(\frac{1}{2(\mu+\ell)^{2}}+\frac{1}{2\mu \ell}\right)L_{1}^{2}\mathbb{E}\left\|\nabla_{y}\psi^{t-1}(x^{t-1},y^{t-1})-\nabla_{y}\psi(x^{t-1},y^{t-1}) \right\|_{F}^{2}.
		\end{aligned}
	\end{equation*}
	
	Setting $\Delta=4\kappa_{y}+3$ and summarizing the above inequalities,   we arrive at inequality \eqref{eq:f_delta}. 
	The proof is complete.
\end{proof}	

\section{Proof of  Lemmas in Subsection \ref{subsection3_2}}\label{Appendix_for_Section3_2}
\subsection{Proof of Lemma \ref{lemma_concave_1}}\label{Appendix_for_Section3_2_Lem1}
\begin{proof}
Denote $\hat{x}^{t-1}=\operatorname{prox}_{\Phi/2\ell}(x^{t-1})$. By the definition of $\Phi_{1 / 2 \ell}\left(\cdot\right)$ and the iteration of $x$ in \eqref{equ:alg_x}, we have
	\begin{equation*}
		\begin{aligned}
			\Phi_{1 / 2 \ell}\left(x^{t}\right) \leq & \Phi_{1 / 2 \ell}\left(x^{t-1}\right) +2\eta_{x}\ell \left\langle\hat{x}^{t-1}-x^{t-1}, \frac{1}{M} \sum_{i=1}^{M} G_{x}^{t-1}\left(x^{t-1}, y^{t-1}, z_{i}^{t-1}\right)\right\rangle\\
			&+\eta_{x}^{2}\ell\left\|\frac{1}{M} \sum_{i=1}^{M} G_{x}^{t-1}\left(x^{t-1}, y^{t-1}, z_{i}^{t-1}\right)\right\|^2.
		\end{aligned}
	\end{equation*}

	\noindent Taking expectation on both sides of the above inequality, conditioned on $\left(x^{t-1},y^{t-1},\psi^{t-1}(\cdot)\right)$,
	\begin{equation*}
		\begin{aligned}
			\mathbb{E}_{t-1}\left[\Phi_{1 / 2 \ell}\left(x^{t}\right)\right]
			\leq & 
			2\eta_{x}\ell \mathbb{E}_{t-1}\left\langle\hat{x}^{t-1}-x^{t-1}, \frac{1}{M} \sum_{i=1}^{M} G_{x}^{t-1}\left(x^{t-1}, y^{t-1}, z_{i}^{t-1}\right)\right\rangle\\
			&+ \Phi_{1 / 2 \ell}\left(x^{t-1}\right)+\eta_{x}^{2}\ell\left\|\mathbb{E}_{t-1}\left[\frac{1}{M}\sum_{i=1}^{M} G_{x}^{t-1}\left(x^{t-1}, y^{t-1}, z_{i}^{t-1}\right)\right]\right\|^{2}\\
			&+\eta_{x}^{2}\ell\left(1+\left\|\nabla_{x}\psi^{t-1}(x^{t-1},y^{t-1})\right\|_F^{2} \right)\frac{\sigma^{2}}{M},
		\end{aligned}
	\end{equation*}
	where $\mathbb{E}_{t-1}\left[\,\cdot\,\right]=\mathbb{E}\left[\,\cdot\,| (x^{t-1},y^{t-1},\psi^{t-1}(\cdot))\right]$ and the inequality follows from Lemma \ref{lem:expvar} (b).
	
	\noindent Then
	\begin{small}
		\begin{equation}\label{eq:concave_Lemma1_cex}
			\begin{aligned}
				&\mathbb{E}_{t-1}\left[\Phi_{1 / 2 \ell}\left(x^{t}\right)\right]\\
				\leq&\Phi_{1 / 2 \ell}\left(x^{t-1}\right)+2 \eta_{x}\ell \left\langle\hat{x}^{t-1}-x^{t-1}, \nabla_{x} \mathcal{L}\left(x^{t-1}, y^{t-1}\right)\right\rangle\\
				&+2 \eta_{x}\ell \mathbb{E}_{t-1}\left\langle\hat{x}^{t-1}-x^{t-1}, \frac{1}{M} \sum_{i=1}^{M} G_{x}^{t-1}\left(x^{t-1}, y^{t-1}, z_{i}^{t-1}\right)-\nabla_{x} \mathcal{L}\left(x^{t-1}, y^{t-1}\right)\right\rangle \\
				&+\eta_{x}^{2}\ell\left\|\mathbb{E}_{t-1}\left[\frac{1}{M} \sum_{i=1}^{M} G_{x}^{t-1}\left(x^{t-1}, y^{t-1}, z_{i}^{t-1}\right)\right]\right\|^{2}
				+\eta_{x}^{2}\ell\left(1+\left\|\nabla_{x}\psi^{t-1}(x^{t-1},y^{t-1})\right\|_{F}^{2}\right) \frac{\sigma^{2}}{M}.
			\end{aligned}
		\end{equation}
	\end{small}
	
	\noindent For the second term on the right-hand side of inequality \eqref{eq:concave_Lemma1_cex},
	\begin{equation*}\label{equ_concave_4}
		\begin{aligned}
			&2 \eta_{x}\ell\left\langle\hat{x}^{t-1}-x^{t-1}, \nabla_{x} \mathcal{L}\left(x^{t-1}, y^{t-1}\right)\right\rangle \\
			\leq &2 \eta_{x}\ell\left(\mathcal{L}\left(\hat{x}^{t-1}, y^{t-1}\right)-\mathcal{L}\left(x^{t-1}, y^{t-1}\right)+\frac{\ell}{2}\left\|\hat{x}^{t-1}-x^{t-1}\right\|^{2}\right)\\
			\leq& 2 \eta_{x}\ell\left(\Phi(x^{t-1})-\mathcal{L}\left(x^{t-1}, y^{t-1}\right)\right)-\eta_{x}\ell^{2}\left\|\hat{x}^{t-1}-x^{t-1}\right\|^{2},
		\end{aligned}
	\end{equation*}where the first inequality follows from $\ell$-smoothness of $\mathcal{L}(\cdot)$ and the second inequality follows from the  definition of $\hat{x}^{t-1}$.
	
	For the third term on the right-hand side of inequality \eqref{eq:concave_Lemma1_cex},
	\begin{small}
		\begin{equation*}
			\begin{aligned}
				&2 \eta_{x}\ell\mathbb{E}_{t-1}\left\langle\hat{x}^{t-1}-x^{t-1}, \frac{1}{M} \sum_{i=1}^{M} G_{x}^{t-1}\left(x^{t-1}, y^{t-1}, z_{i}^{t-1}\right)-\nabla_{x} \mathcal{L}\left(x^{t-1}, y^{t-1}\right)\right\rangle\\
				\leq& \eta_{x}\ell \bar\Delta\left\|\hat{x}^{t-1}-x^{t-1}\right\|^{2}+\frac{ \eta_{x}\ell}{\bar\Delta}L_{1}^{2}\left\|\nabla_{x}\psi^{t-1}(x^{t-1},y^{t-1})-\nabla_{x}\psi(x^{t-1},y^{t-1})\right\|_{F}^{2}, \,\forall \bar\Delta>0,
			\end{aligned}
		\end{equation*}
	\end{small}where the inequality follows from Young's inequality and the bias bound of $\frac{1}{M} \sum_{i=1}^{M} G_{x}^{t-1}\left(x^{t-1}, y^{t-1}, z_{i}^{t-1}\right)$ in Lemma \ref{lem:expvar} (a).
	
	For the fourth term on the right-hand side of inequality \eqref{eq:concave_Lemma1_cex},
	\begin{small}
		\begin{equation*}
			\begin{aligned}
				&\eta_{x}^{2}\ell\left\|\mathbb{E}_{t-1}\left[\frac{1}{M} \sum_{i=1}^{M} G_{x}^{t-1}\left(x^{t-1}, y^{t-1}, z_{i}^{t-1}\right)\right]\right\|^{2}\\
				=&\eta_{x}^{2}\ell\left\|\mathbb{E}_{t-1}\left[\frac{1}{M} \sum_{i=1}^{M} G_{x}^{t-1}\left(x^{t-1}, y^{t-1}, z_{i}^{t-1}\right)\right]-\nabla_{x}\mathcal{L}(x^{t-1},y^{t-1})+\nabla_{x}\mathcal{L}(x^{t-1},y^{t-1})\right\|^{2}\\
				\leq&2\eta_{x}^{2}\ell\left\|\mathbb{E}_{t-1}\left[\frac{1}{M} \sum_{i=1}^{M} G_{x}^{t-1}\left(x^{t-1}, y^{t-1}, z_{i}^{t-1}\right)\right]-\nabla_{x}\mathcal{L}(x^{t-1},y^{t-1})\right\|^{2}\\
				&+2\eta_{x}^{2}\ell\left\|\nabla_{x}\mathcal{L}(x^{t-1},y^{t-1}) \right\|^{2} \\
				\leq& 2L_{1}^{2}\eta_{x}^{2}\ell\left\|\nabla_{x}\psi^{t-1}(x^{t-1},y^{t-1})-\nabla_{x}\psi(x^{t-1},y^{t-1}) \right\|_{F}^{2}+2L^{2}\eta_{x}^{2}\ell,
			\end{aligned}
		\end{equation*}
	\end{small}where the first inequality follows from the fact $\|a+b\|^{2}\leq 2\|a\|^{2}+2\|b\|^{2}$ and the second inequality follows from Lemma  \ref{lem:expvar} (a) and $L$-Lipschitz continuity of $\mathcal{L}(\cdot,y)$.
	
	Then
	\begin{equation*}
		\begin{aligned}
			\mathbb{E}_{t-1}\left[\Phi_{1 / 2 \ell}\left(x^{t}\right)\right] & \leq\Phi_{1 / 2 \ell}\left(x^{t-1}\right)+2 \eta_{x}\ell\left(\Phi(x^{t-1})-\mathcal{L}\left(x^{t-1}, y^{t-1}\right)\right)\\
			&-\eta_{x}\ell(\ell-\bar\Delta)\left\|\hat{x}^{t-1}-x^{t-1}\right\|^{2}\\
			&+\left(\frac{\eta_{x}}{\bar\Delta}+2 \eta_{x}^{2}\right)\ell L_{1}^{2}\left\|\nabla_{x}\psi^{t-1}(x^{t-1},y^{t-1})-\nabla_{x}\psi(x^{t-1},y^{t-1})\right\|_{F}^{2}\\
			&+2 \eta_{x}^{2}\ell L^{2}+\eta_{x}^{2}\ell\left(1+\left\|\nabla_{x}\psi^{t-1}(x^{t-1},y^{t-1})\right\|_{F}^{2}\right) \frac{\sigma^{2}}{M}.
		\end{aligned}
	\end{equation*}  
	
	\noindent By the definition of $\Delta_t$, and
	taking expectation on both sides of the above inequality, we have
	\begin{equation*}
		\begin{aligned}
			\mathbb{E}\left[\Phi_{1 / 2 \ell}\left(x^{t}\right)\right] & \leq\mathbb{E}\left[\Phi_{1 / 2 \ell}\left(x^{t-1}\right)\right]+2 \eta_{x}\ell\Delta_{t-1}-\eta_{x}\ell(\ell-\bar\Delta)\mathbb{E}\left\|\hat{x}^{t-1}-x^{t-1}\right\|^{2} \\
			& +\left(\frac{\eta_{x}}{\bar\Delta}+2 \eta_{x}^{2}\right)\ell L_{1}^{2}\mathbb{E}\left\|\nabla_{x}\psi^{t-1}(x^{t-1},y^{t-1})-\nabla_{x}\psi(x^{t-1},y^{t-1})\right\|_{F}^{2}+2 \eta_{x}^{2}\ell L^{2}\\
			&+\eta_{x}^{2}\ell\mathbb{E}\left(1+\left\|\nabla_{x}\psi^{t-1}(x^{t-1},y^{t-1})\right\|_{F}^{2}\right) \frac{\sigma^{2}}{M}.
		\end{aligned}
	\end{equation*}
	
	\noindent Setting $\bar\Delta=\ell/2$ in the above inequality, the desired inequality \eqref{eq:concave_descentPhi} follows directly from the fact that $\left\|\hat{x}^{t-1}-x^{t-1}\right\|=\frac{\left\|\nabla\Phi_{1/2\ell}\left(x^{t-1}\right)\right\|}{2\ell}$ \cite[Theorem 6.60]{beck2017first}. The proof is complete.
\end{proof}

\subsection{Proof of Lemma \ref{lemma_concave_Delta}}\label{Appendix_for_Section3_2_Lem2}
\begin{proof}
	Let $0\leq s\leq t$ be any integer. By the definition of $\Delta_t$,
	\begin{equation*}\label{eq:concave_Delta_core}
		\begin{aligned}
			\Delta_{t}
			\leq & \underbrace{\mathbb{E}\left[\mathcal{L}\left(x^{t}, y^{\star}\left(x^{t}\right)\right)-\mathcal{L}(x^{s},y^{\star}(x^{t}))\right]}_{I_1}\\
			&+\underbrace{\mathbb{E}\left[\mathcal{L}(x^{s},y^{\star}(x^{s}))-\mathcal{L}\left(x^{t}, y^{\star}\left(x^{s}\right)\right)\right]}_{I_2}+\underbrace{\mathbb{E}\left[\mathcal{L}\left(x^{t}, y^{\star}\left(x^{s}\right)\right)-\mathcal{L}\left(x^{t}, y^{t+1}\right)\right]}_{I_3}\\
			&+\underbrace{\mathbb{E}\left[\left(\mathcal{L}\left(x^{t}, y^{t+1}\right)-\mathcal{L}\left(x^{t+1}, y^{t+1}\right)\right)\right]}_{I_4}+\mathbb{E}\left[\left(\mathcal{L}\left(x^{t+1}, y^{t+1}\right)-\mathcal{L}\left(x^{t}, y^{t}\right)\right)\right].
		\end{aligned}
	\end{equation*}
	
	For the term $I_{1}$,
	\begin{equation*}\label{eq:concave_Delta_1}
		\begin{aligned}
			I_{1}\leq&L  \mathbb{E}\left\|x^{t}-x^{s} \right\|\\
			\leq& L \eta_{x}\sum_{j=s}^{t-1}\mathbb{E}\left\| \frac{1}{M}\sum_{i=1}^{M}G_x^{j}(x^{j},y^{j},z^{j}_{i})\right\|\\
			\leq&\frac{L \eta_{x}}{M}\sum_{j=s}^{t-1}\sum_{i=1}^{M}\mathbb{E}\left[\left\|\left(\nabla_{x}\psi^j(x^j, y^j)-\nabla_{x}\psi(x^j, y^j)\right)^{\top}\nabla_{z}l(x^j,y^j,z^j_i)\right\|\right]\\
			&+\frac{L \eta_{x}}{M}\sum_{j=s}^{t-1}\sum_{i=1}^{M}\mathbb{E}\left[\left\|\nabla_{x}l(x^j,y^j,z^j_i)+\left(\nabla_{x}\psi(x^j, y^j)\right)^{\top}\nabla_{z}l(x^j,y^j,z^j_i) \right\|\right]\\
			\leq&\frac{L \eta_{x}}{M}\sum_{j=s}^{t-1}\sum_{i=1}^{M}\mathbb{E}\left[\frac{1}{4}\left\|\nabla_{x}\psi^j(x^j,y^j)-\nabla_{x}\psi(x^j, y^j)\right\|_F^2+
			\left(\mathbb{E}_{t}\left\|\nabla_{z}l(x^j,y^j,z^j_i)\right\|\right)^2\right]\\
			&+\frac{L \eta_{x}}{M}\sum_{j=s}^{t-1}\sum_{i=1}^{M}\mathbb{E}\left[\left\|\nabla_{x}l(x^j,y^j,z^j_i)\right\|+\left\|\nabla_{x}\psi(x^j, y^j) \right\|_{F}\left\|\nabla_{z}l(x^j,y^j,z^j_i) \right\| \right]\\
			\leq&\frac{1}{4}L \eta_{x}
			\sum_{j=s}^{t-1}\mathbb{E}\left\|\nabla_{x}\psi^j(x^j,y^j)-\nabla_{x}\psi(x^j,y^j)\right\|_F^2 +LL_{1}\eta_{x}(1+L_{0}+L_{1})(t-s),
		\end{aligned}
	\end{equation*}
	where the first inequality follows from $L$-Lipschitz continuity of $\mathcal{L}(\cdot,y)$, the second inequality follows from triangle inequality and the definition of $G_x(x^{j},y^{j},z^{j}_{i})$ in \eqref{equ:stograd_unbia_x} and $G_x^{j}(x^{j},y^{j},z^{j}_{i})$ in \eqref{eq:adaptive_stograd}, the third inequality follows from Young's inequality along with triangle inequality and Cauchy-Schwarz inequality, the last inequality follows from Assumption \ref{ass_bounded_moment} and $L_{0}$-Lipschitz continuity of $\psi(\cdot)$.
	
	By a similar analysis
	\begin{align*}
		I_{2}
		&\leq  \frac{1}{4}L\eta_{x}\sum_{j=s}^{t-1}\mathbb{E}\left\|\nabla_{x}\psi^j(x^j,y^j)-\nabla\psi_{x}(x^j,y^j)\right\|_F^2
		+LL_{1}(1+L_{0}+L_{1})\eta_{x}(t-s),\\
		I_{4}
		&\leq\frac{1}{4}
		L\eta_{x}\mathbb{E}\left\|\nabla_{x}\psi^{t}(x^{t},y^{t})-\nabla_{x}\psi(x^{t},y^{t})\right\|_F^2+LL_{1}(1+L_{0}+L_{1})\eta_{x}.
	\end{align*}
	
	Next, we provide an upper bound for $I_3$. Given $y^{t+1}=\operatorname{proj}_{\mathcal{Y}}\left(y^{t}+\eta_{y}\frac{1}{M}\sum_{i=1}^{M}G_{y}^{t}(x^{t}, y^{t}, z^{t}_{i})\right)$ and $\mathcal{Y}$ is a closed convex set,
	\begin{equation*}
		\begin{aligned}
			\left(y-y^{t+1}\right)^{\top}\left(y^{t+1}-y^{t}-\eta_{y}\frac{1}{M}\sum_{i=1}^{M}G_{y}^{t} \left(x^{t}, y^{t},z^{t}_{i}\right)\right) \geq 0,
		\end{aligned}
	\end{equation*}
	which implies
	\begin{equation*}
		\begin{aligned}
			\left\|y-y^{t+1}\right\|^{2} \leq &\left\|y-y^{t}\right\|^{2}+2\eta_{y}(y^{t}-y)^{\top}\left[ \frac{1}{M}\sum_{i=1}^{M}G_{y}^{t} \left(x^{t}, y^{t},z^{t}_{i}\right)\right]\\
			&+2\eta_{y}(y^{t+1}-y^{t})^{\top}\left[ \frac{1}{M}\sum_{i=1}^{M}G_{y}^{t} \left(x^{t}, y^{t},z^{t}_{i}\right)\right]-\left\|y^{t+1}-y^{t}\right\|^{2}\\
			=&\left\|y-y^{t}\right\|^{2}+2\eta_{y}(y^{t}-y)^{\top}\left(\frac{1}{M}\sum_{i=1}^{M}G_{y}\left(x^{t}, y^{t},z^{t}_{i}\right)\right)\\
			&+2\eta_{y}(y^{t}-y)^{\top}\left[ \frac{1}{M}\sum_{i=1}^{M}G_{y}^{t} \left(x^{t}, y^{t},z^{t}_{i}\right)-\frac{1}{M}\sum_{i=1}^{M}G_{y}\left(x^{t}, y^{t},z^{t}_{i}\right)\right]\\
			&+2\eta_{y}(y^{t+1}-y^{t})^{\top}\left[ \frac{1}{M}\sum_{i=1}^{M}G_{y}^{t} \left(x^{t}, y^{t},z^{t}_{i}\right)-\nabla_{y}\mathcal{L}(x^{t},y^{t})\right]\\
			&+2\eta_{y}(y^{t+1}-y^{t})^{\top}\nabla_{y}\mathcal{L}(x^{t},y^{t})-\|y^{t+1}-y^{t}\|^{2}\\
			\leq&\left\|y-y^{t}\right\|^{2}+2\eta_{y}(y^{t}-y)^{\top}\left(\frac{1}{M}\sum_{i=1}^{M}G_{y}\left(x^{t}, y^{t},z^{t}_{i}\right)\right)\\
			&+2\eta_{y}\|y^{t}-y\|\left\|\nabla_{y}\psi^{t}(x^{t},y^{t})-\nabla_{y}\psi(x^{t},y^{t})\right\|_{F} \left\|\frac{1}{M}\sum_{i=1}^{M}\nabla_{z}l(x^{t},y^{t},z_i^{t})\right\|\\
			&+\|y^{t+1}-y^{t}\|^{2}+\eta_{y}^{2}\left\| \frac{1}{M}\sum_{i=1}^{M}G_{y}^{t} \left(x^{t}, y^{t},z^{t}_{i}\right)-\nabla_{y}\mathcal{L}(x^{t},y^{t})\right\|^{2}\\
			&+2\eta_{y}(y^{t+1}-y^{t})^{\top}\nabla_{y}\mathcal{L}(x^{t},y^{t})-\|y^{t+1}-y^{t}\|^{2}
		\end{aligned}
	\end{equation*}where the second inequality follows from Cauchy-Schwarz inequality and Young's inequality. 
	
	\noindent Taking expectation on both sides of the above inequality, conditioned on $(x^{t},y^{t},\psi^{t}(\cdot))$,
	\begin{equation}\label{eq:NCC_lem2_1}
		\begin{aligned}
			&\mathbb{E}_{t}\left\|y-y^{t+1}\right\|^{2}\\
			\leq&\|y-y^{t}\|^{2}+2\eta_{y}\left(\mathcal{L}(x^{t},y^{t})-\mathcal{L}(x^{t},y) \right)\\
			&+2\eta_{y}L_{1}\|y^{t}-y\|\left\|\nabla_{y}\psi^{t}(x^{t},y^{t})-\nabla_{y}\psi(x^{t},y^{t})\right\|_{F} \\
			&+\eta_{y}^{2}\underbrace{\mathbb{E}_{t}\left\| \frac{1}{M}\sum_{i=1}^{M}G_{y}^{t} \left(x^{t}, y^{t},z^{t}_{i}\right)-\nabla_{y}\mathcal{L}(x^{t},y^{t})\right\|^{2}}_{I_5}+2\eta_{y}\left(\mathcal{L}(x^{t},y^{t+1})-\mathcal{L}(x^{t},y^{t}) \right),		
		\end{aligned}
	\end{equation}
	where the inequality follows from concavity and $\ell$-smoothness of $\mathcal{L}(\cdot)$ along with Assumption~\ref{ass_bounded_moment}. 
	
	\noindent Note that 
	\begin{equation*}
		\begin{aligned}
			I_{5}=&\mathbb{E}_{t}\left\|\frac{1}{M}\sum_{i=1}^{M}G_{y}^{t} \left(x^{t}, y^{t},z^{t}_{i}\right)-\mathbb{E}_{t}\frac{1}{M}\sum_{i=1}^{M}G_{y}^{t} \left(x^{t}, y^{t},z^{t}_{i}\right)\right\|^{2}\\
			&+\left\|\mathbb{E}_{t}\frac{1}{M}\sum_{i=1}^{M}G_{y}^{t} \left(x^{t}, y^{t},z^{t}_{i}\right)-\nabla_{y}\mathcal{L}(x^{t},y^{t})\right\|^{2}\\
			\leq&\left(1+\left\|\nabla_{y}\psi^{t}(x^{t},y^{t})\right\|_F^{2} \right)\frac{\sigma^{2}}{M}+L_{1}^{2}\left\|\nabla_{y}\psi^{t}(x^{t},y^{t})-\nabla_{y}\psi(x^{t},y^{t}) \right\|_{F}^{2},
		\end{aligned}
	\end{equation*}where the inequality follows from Lemma \ref{lem:expvar} (c)-(d).
	Taking expectation on both sides of inequality \eqref{eq:NCC_lem2_1}, we have
	\begin{equation*}
		\begin{aligned}
			\mathbb{E}\left\|y-y^{t+1}\right\|^{2}
			\leq&\mathbb{E}\|y-y^{t}\|^{2}+2\eta_{y}\mathbb{E}\left[\mathcal{L}(x^{t},y^{t+1})-\mathcal{L}(x^{t},y) \right]\\
			&+2\eta_{y}L_{1}\mathbb{E}\|y^{t}-y\|\left\|\nabla_{y}\psi^{t}(x^{t},y^{t})-\nabla_{y}\psi(x^{t},y^{t}) \right\|_{F} \\
			&+\eta_{y}^{2}\mathbb{E}\left(1+\left\|\nabla_{y}\psi^{t}(x^{t},y^{t})\right\|_F^{2} \right)\frac{\sigma^{2}}{M}\\
			&+\eta_{y}^{2}L_{1}^{2}\mathbb{E}\left\|\nabla_{y}\psi^{t}(x^{t},y^{t})-\nabla_{y}\psi(x^{t},y^{t}) \right\|_{F}^{2}.	
		\end{aligned}
	\end{equation*}
	Setting $y=y^{\star}\left(x^{s}\right)$ in the above inequality and rearranging the terms,
	\begin{equation*}
		\begin{aligned}
			I_3\leq& \frac{1}{2 \eta_{y}}\left(\mathbb{E}\left[\left\|y^{t}-y^{\star}\left(x^{s}\right)\right\|^{2}\right]-\mathbb{E}\left[\left\|y^{t+1}-y^{\star}\left(x^{s}\right)\right\|^{2}\right]\right)\\
			&+L_{1}\mathbb{E}\|y^{t}-y^{\star}\left(x^{s}\right)\|\left\|\nabla_{y}\psi^{t}(x^{t},y^{t})-\nabla_{y}\psi(x^{t},y^{t}) \right\|_{F} \\
			&+\frac{\eta_{y}}{2}\mathbb{E}\left(1+\left\|\nabla_{y}\psi^{t}(x^{t},y^{t})\right\|_F^{2} \right)\frac{\sigma^{2}}{M}\\
			&+\frac{\eta_{y}}{2}L_{1}^{2}\mathbb{E}\left\|\nabla_{y}\psi^{t}(x^{t},y^{t})-\nabla_{y}\psi(x^{t},y^{t}) \right\|_{F}^{2}.
		\end{aligned}
	\end{equation*}
	
	Then, 
		\begin{equation}\label{eq:concave_Delta_upperbound}
			\begin{aligned}
				\Delta_t \leq & \eta_x L L_1 (1 + L_0 + L_1) (2t - 2s +1) + \frac{1}{2} L \eta_x \sum_{j=s}^{t-1} \mathbb{E} \left\| \nabla_x \psi^j \left( x^j, y^j \right) - \nabla_x \psi \left( x^j, y^j \right) \right\|_F^2 \\
				& + \frac{1}{4} L \eta_x \mathbb{E} \left\| \nabla_x \psi^t \left( x^t, y^t \right) - \nabla_x \psi \left( x^t, y^t \right) \right\|_F^2 \\
				& + \mathbb{E} \left[ \mathcal{L} \left( x^{t+1}, y^{t+1} \right) - \mathcal{L} \left( x^t, y^t \right) \right] + \frac{1}{2 \eta_y} \left( \mathbb{E} \left[ \left\| y^t - y^\star \left( x^s \right) \right\|^2 \right] - \mathbb{E} \left[ \left\| y^{t+1} - y^\star \left( x^s \right) \right\|^2 \right] \right) \\
				& + L_1 \mathbb{E} \left\| y^t - y^\star \left( x^s \right) \right\| \left\| \nabla_y \psi^t \left( x^t, y^t \right) - \nabla_y \psi \left( x^t, y^t \right) \right\|_{F} \\
				& + \frac{\eta_y}{2} \mathbb{E} \left( 1 + \left\| \nabla_y \psi^t \left( x^t, y^t \right) \right\|_F^2 \right) \frac{\sigma^2}{M} \\
				& + \frac{\eta_y}{2} L_1^2 \mathbb{E} \left\| \nabla_y \psi^t \left( x^t, y^t \right) - \nabla_y \psi \left( x^t, y^t \right) \right\|_F^2.
			\end{aligned}
		\end{equation}
	
	For any  integers $B$ and $T$ $(0<B\leq T+1)$,
	\begin{small}
		\begin{equation*}
			\begin{aligned}
				\sum_{t=0}^{T} \Delta_{t}=\sum_{j=0}^{\lfloor(T+1)/B\rfloor-1}\sum_{t=j B}^{(j+1) B-1} \Delta_{t}+\sum_{t=\lfloor(T+1)/B\rfloor B}^{T}\Delta_{t}\,.
			\end{aligned}
		\end{equation*}
	\end{small}

	\noindent Choosing $s=\lfloor t/B\rfloor B$ on the right-hand side of inequality \eqref{eq:concave_Delta_upperbound}, we have by the boundedness of the constraint set $\mathcal{Y}$ that
	\begin{equation*}
		\begin{aligned}
			\sum_{j=0}^{\lfloor(T+1)/B\rfloor-1}\sum_{t=j B}^{(j+1) B-1} \Delta_{t}\leq&\eta_{x} L L_{1}\left(1+L_{0}+L_{1}\right) B^{2}\lfloor  (T+1)/B\rfloor\\
			&+\frac{1}{4}\eta_{x}L (2 B-1)
			\sum_{t=0}^{\lfloor (T+1)/B\rfloor B-1}\mathbb{E}\left\|\nabla \psi^{t}\left(x^{t},y^{t}\right)-\nabla \psi\left(x^{t},y^{t}\right)\right\|_{F}^{2}\\
			&+\mathbb{E}\left[\mathcal{L}\left(x^{\lfloor\frac{T+1}{B}\rfloor B},y^{\lfloor\frac{T+1}{B}\rfloor B} \right)-\mathcal{L}\left(x^{0},y^{0}\right) \right]
			+\frac{D^{2}}{2 \eta_{y}}\lfloor(T+1)/B \rfloor\\
			&+L_{1}D\sum_{t=0}^{\lfloor (T+1)/B\rfloor B-1}\left\|\nabla_{y}\psi^{t}(x^{t},y^{t})-\nabla_{y}\psi(x^{t},y^{t}) \right\|_{F}\\
			&+\frac{\eta_{y}}{2}\sum_{t=0}^{\lfloor (T+1)/B\rfloor B-1}\mathbb{E}\left(1+\left\|\nabla_{y}\psi^{t}(x^{t},y^{t})\right\|_F^{2} \right)\frac{\sigma^{2}}{M}\\
			&+\frac{\eta_{y}}{2}L_{1}^{2}\sum_{t=0}^{\lfloor (T+1)/B\rfloor B-1}\mathbb{E}\left\|\nabla_{y}\psi^{t}(x^{t},y^{t})-\nabla_{y}\psi(x^{t},y^{t}) \right\|_{F}^{2},
		\end{aligned}
	\end{equation*}
and
	\begin{equation*}
		\begin{aligned}
			\sum_{t=\lfloor(T+1)/B\rfloor B}^{T}\Delta_{t}\leq& \eta_{x} L L_{1}\left(1+L_{0}+L_{1}\right) B^{2}\\
			+&\frac{1}{4}\eta_{x}L(2 B-1)
			\sum_{t=\lfloor (T+1)/B\rfloor B}^{T}\mathbb{E}\left\|\nabla \psi^{t}\left(x^{t},y^{t}\right)-\nabla \psi\left(x^{t},y^{t}\right)\right\|_{F}^{2}\\
			+&\mathbb{E}\left[\mathcal{L}(x^{T+1},y^{T+1})-\mathcal{L}\left(x^{\lfloor\frac{T+1}{B}\rfloor B},y^{\lfloor\frac{T+1}{B}\rfloor B} \right) \right]+\frac{D^2}{2\eta_y}\\
			+&L_{1}D\sum_{t=\lfloor(T+1)/B\rfloor B}^{T}\left\|\nabla_{y}\psi^{t}(x^{t},y^{t})-\nabla_{y}\psi(x^{t},y^{t}) \right\|_{F}\\
			+&\frac{\eta_y}{2}\sum_{t=\lfloor(T+1)/B\rfloor B}^{T}\mathbb{E}\left(1+\left\|\nabla_{y}\psi^{t}(x^{t},y^{t})\right\|_F^{2} \right)\frac{\sigma^{2}}{M}\\
			+&\frac{\eta_y}{2}L_{1}^{2}\sum_{t=\lfloor(T+1)/B\rfloor B}^{T}\mathbb{E}\left\|\nabla_{y}\psi^{t}(x^{t},y^{t})-\nabla_{y}\psi(x^{t},y^{t}) \right\|_{F}^{2}.
		\end{aligned}
	\end{equation*}
	
	\noindent Summarize the above two inequalities, we obtain
	\begin{equation*}\label{eq:concave_summing}
		\begin{aligned}
			\sum_{t=0}^{T} \Delta_{t}\leq& \eta_{x} L L_{1}\left(1+L_{0}+L_{1}\right) B (T+1)+\eta_{x} L L_{1}\left(1+L_{0}+L_{1}\right) B^{2}\\
			&+\frac{1}{4}\eta_{x}L (2 B-1)
			\sum_{t=0}^{T}\mathbb{E}\left\|\nabla \psi^{t}\left(x^{t},y^{t}\right)-\nabla \psi\left(x^{t},y^{t}\right)\right\|_{F}^{2}\\
			&+\mathbb{E}\left[\mathcal{L}(x^{T+1},y^{T+1})-\mathcal{L}(x^{0},y^{0})\right]+\frac{D^{2}}{2 \eta_{y}}\left(\frac{T+1}{B}+1\right)\\
			&+L_{1}D\sum_{t=0}^{T}\left\|\nabla_{y}\psi^{t}(x^{t},y^{t})-\nabla_{y}\psi(x^{t},y^{t}) \right\|_{F}
			+\frac{\eta_{y}}{2}\sum_{t=0}^{T}\mathbb{E}\left(1+\left\|\nabla_{y}\psi^{t}(x^{t},y^{t})\right\|_F^{2} \right)\frac{\sigma^{2}}{M}\\
			&+\frac{\eta_{y}}{2}L_{1}^{2}\sum_{t=0}^{T}\mathbb{E}\left\|\nabla_{y}\psi^{t}(x^{t},y^{t})-\nabla_{y}\psi(x^{t},y^{t}) \right\|_{F}^{2}.
		\end{aligned}
	\end{equation*}
	For the fourth term on the right-hand side of the above inequality, we have
	by the definition of $\hat{\Delta}_{0}$ and $L$-Lipschitz continuity of $\mathcal{L}(\cdot,y)$ that
	\begin{equation*}
		\begin{aligned}
			\mathbb{E}\left[\mathcal{L}\left(x^{T+1}, y^{T+1}\right)-\mathcal{L}\left(x^{0}, y^{0}\right)\right]
			\leq& \hat{\Delta}_{0}+\frac{1}{4}L\eta_{x}\sum_{t=0}^{T}\left\|\nabla_{x} \psi^{t}\left(x^{t},y^{t}\right)-\nabla_{x} \psi\left(x^{t},y^{t}\right)\right\|_{F}^{2}\\
			&+LL_{1}\eta_{x}\left(1+L_{0}+L_{1}\right)(T+1).
		\end{aligned}
	\end{equation*}
	Subsequently,
	\begin{equation*}
		\begin{aligned}
			\frac{1}{T+1}\sum_{t=0}^{T} \Delta_{t}\leq& \eta_{x} L L_{1}\left(1+L_{0}+L_{1}\right) (B+1) +\frac{\hat{\Delta}_{0}+\eta_{x} L L_{1}\left(1+L_{0}+L_{1}\right) B^{2}}{T+1}\\
			&+\frac{D^{2}}{2 \eta_{y}B}+\frac{D^{2}}{2 \eta_{y}(T+1)}+\frac{\eta_{x}}{2}L B
			\frac{1}{T+1}\sum_{t=0}^{T}\mathbb{E}\left\|\nabla_{x} \psi^{t}\left(x^{t},y^{t}\right)-\nabla_{x} \psi\left(x^{t},y^{t}\right)\right\|_{F}^{2}
			\\
			&+L_{1}D\frac{1}{T+1}\sum_{t=0}^{T}\left\|\nabla_{y}\psi^{t}(x^{t},y^{t})-\nabla_{y}\psi(x^{t},y^{t}) \right\|_{F}\\
			&
			+\frac{\eta_{y}}{2}\frac{1}{T+1}\sum_{t=0}^{T}\mathbb{E}\left(1+\left\|\nabla_{y}\psi^{t}(x^{t},y^{t})\right\|_F^{2} \right)\frac{\sigma^{2}}{M}\\
			&+\frac{\eta_{y}}{2}L_{1}^{2}\frac{1}{T+1}\sum_{t=0}^{T}\mathbb{E}\left\|\nabla_{y}\psi^{t}(x^{t},y^{t})-\nabla_{y}\psi(x^{t},y^{t}) \right\|_{F}^{2}.
		\end{aligned}
	\end{equation*}
	\noindent The proof is complete.
\end{proof}

\end{appendices}
\end{document}